\documentclass[11pt]{amsart}
\usepackage{amsmath, amssymb, latexsym}
\usepackage{mathrsfs}
\usepackage[all, knot]{xy}
\xyoption{arc}

\newtheorem{theorem}{Theorem}[section]
\newtheorem{proposition}[theorem]{Proposition}
\newtheorem{corollary}[theorem]{Corollary}
\newtheorem{lemma}[theorem]{Lemma}

\theoremstyle{definition}
\newtheorem{definition}[theorem]{Definition}
\newtheorem{examples}[theorem]{Examples}
\newtheorem{example}[theorem]{Example}

\numberwithin{equation}{section}

\usepackage{color}

\newenvironment{red}
{\relax\color{red}}
{\hspace*{.5ex}\relax}

\newcommand{\ber}{\begin{red}}
\newcommand{\er}{\end{red}}

\newcommand{\Z}{\mathbb{Z}}
\newcommand{\Q}{\mathbb{Q}}

\newcommand{\C}{\mathbb{C}}

\newcommand{\g}{\mathfrak{g}}
\newcommand{\B}{\mathbb{B}}
\newcommand{\h} {\mathfrak{h}}

\newcommand{\wt}{\mathsf{wt}}

\newcommand{\fmod}{\text{-}\mathsf{fmod}}

\newcommand{\im}{{\mathsf {im}}}
\newcommand{\ind}{\mathsf {Ind}}
\newcommand{\coind}{{\mathsf {Coind}}}
\newcommand{\res}{{\mathsf{Res}}}
\newcommand{\soc}{\mathsf{ soc}}
\newcommand{\hd}{\mathsf{ hd}}
\newcommand{\Hom}{\mathsf{ Hom}}
\newcommand{\HOM}{\mathsf{ HOM}}

\newcommand{\fMod}{ \mathsf{fmod}}

\newcommand{\ch}{\mathsf{ ch}}
\newcommand{\qch}{\mathsf{ch}_{q}}
\newcommand{\dm}{\mathsf{dim}}
\newcommand{\qdim}{\mathsf{dim}_q }
\newcommand{\infl}{ \mathsf{ infl} }
\newcommand{\pr}{ \mathsf{ pr} }

\newcommand{\ep} { \varepsilon}

\newcommand{\ke} { \tilde{e}}
\newcommand{\kf} { \tilde{f}}

\newcommand{\id} { \mathsf{id}}

\newcommand{\e} {e}

\newcommand{\ad} {\mathsf{a}}

\newcommand{\cB} {\mathbf{B}}

\newcommand{\dbt}[2]{\overset{#2}{\underset{#1}{\boxtimes}} }

\newcommand{\fourthree}{\buildrel { (4.3)} \over  \Rightarrow}

\newcommand{\fourfour}{\buildrel { (4.4)} \over  \Rightarrow}

\newcommand{\ftf}{\buildrel {\hspace{-.3truecm} (4.2),(4.4)} \over  \Longrightarrow}
\newcommand{\ftfi}{\buildrel {\hspace{-.3truecm}(4.2),(4.5)} \over  \Longrightarrow}
\newcommand{\ftthf}{\buildrel { \hspace{-.5truecm}(4.2),(4.3),(4.4)} \over  \Longrightarrow}

\newcommand{\DynkinA}
{
\xy
(0.5,0)*{};(7.5,0)*{} **\dir{-};
(8.5,0)*{};(12.5,0)*{} **\dir{-};
(12.5,0)*{};(24.5,0)*{} **\dir{.};
(24.5,0)*{};(28.5,0)*{} **\dir{-};
(29.5,0)*{}; (36.5,0)*{} **\dir{-};
(0,0)*{\circ}; (8,0)*{\circ};
(29,0)*{\circ}; (37,0)*{\circ};
(0,-3)*{1}; (8,-3)*{2};
(29,-3)*{n-1};(37,-3)*{n};
\endxy
}

\newcommand{\DynkinB}
{
\xy
(0.5,0)*{};(7.5,0)*{} **\dir{-};
(8.5,0)*{};(12.5,0)*{} **\dir{-};
(12.5,0)*{};(24.5,0)*{} **\dir{.};
(24.5,0)*{};(28.5,0)*{} **\dir{-};
{\ar@{=>} (29.5,0)*{}; (36.5,0)*{}};
(0,0)*{\circ}; (8,0)*{\circ};
(29,0)*{\circ}; (37,0)*{\circ};
(0,-3)*{1}; (8,-3)*{2};
(29,-3)*{n-1};(37,-3)*{n};
\endxy
}

\newcommand{\DynkinC}
{
\xy
(0.5,0)*{};(7.5,0)*{} **\dir{-};
(8.5,0)*{};(12.5,0)*{} **\dir{-};
(12.5,0)*{};(24.5,0)*{} **\dir{.};
(24.5,0)*{};(28.5,0)*{} **\dir{-};
{\ar@{<=} (29.5,0)*{}; (36.5,0)*{}};
(0,0)*{\circ}; (8,0)*{\circ};
(29,0)*{\circ}; (37,0)*{\circ};
(0,-3)*{1}; (8,-3)*{2};
(29,-3)*{n-1};(37,-3)*{n};
\endxy
}

\newcommand{\DynkinD}
{
\xy
(0.5,0)*{};(7.5,0)*{} **\dir{-};
(8.5,0)*{};(12.5,0)*{} **\dir{-};
(12.5,0)*{};(24.5,0)*{} **\dir{.};
(24.5,0)*{};(28.5,0)*{} **\dir{-};
(29.5,0.5)*{};(33.5,2.5)*{} **\dir{-};
(29.2,-0.3)*{};(33.5,-2.5)*{} **\dir{-};
(0,0)*{\circ}; (8,0)*{\circ};
(29,0)*{\circ}; (34,2.5)*{\circ}; (34,-2.8)*{\circ};
(0,-3)*{1}; (8,-3)*{2};
(26,-3)*{n-2};(35,0.5)*{n}; (35,-5)*{n-1};
\endxy
}

\newcommand{\DynkinEs}
{
\xy
(0.5,0)*{};(7.5,0)*{} **\dir{-};
(8.5,0)*{};(15.5,0)*{} **\dir{-};
(16.5,0)*{};(23.5,0)*{} **\dir{-};
(24.5,0)*{};(31.5,0)*{} **\dir{-};
(16,0.5)*{};(16,5.5)*{} **\dir{-};
(0,0)*{\circ}; (8,0)*{\circ}; (16,0)*{\circ};
(24,0)*{\circ}; (32,0)*{\circ};
(16,6)*{\circ};
(0,-3)*{1}; (8,-3)*{2}; (16,-3)*{3};
(24,-3)*{5}; (32,-3)*{6};
(14,5)*{4};
\endxy
}

\newcommand{\DynkinEsv}
{
\xy
(0.5,0)*{};(7.5,0)*{} **\dir{-};
(8.5,0)*{};(15.5,0)*{} **\dir{-};
(16.5,0)*{};(23.5,0)*{} **\dir{-};
(24.5,0)*{};(31.5,0)*{} **\dir{-};
(32.5,0)*{};(39.5,0)*{} **\dir{-};
(16,0.5)*{};(16,5.5)*{} **\dir{-};
(0,0)*{\circ}; (8,0)*{\circ}; (16,0)*{\circ};
(24,0)*{\circ}; (32,0)*{\circ};
(40,0)*{\circ};
(16,6)*{\circ};
(0,-3)*{1}; (8,-3)*{2}; (16,-3)*{3};
(24,-3)*{5}; (32,-3)*{6}; (40,-3)*{7};
(14,5)*{4};
\endxy
}

\newcommand{\DynkinEe}
{
\xy
(0.5,0)*{};(7.5,0)*{} **\dir{-};
(8.5,0)*{};(15.5,0)*{} **\dir{-};
(16.5,0)*{};(23.5,0)*{} **\dir{-};
(24.5,0)*{};(31.5,0)*{} **\dir{-};
(32.5,0)*{};(39.5,0)*{} **\dir{-};
(40.5,0)*{};(47.5,0)*{} **\dir{-};
(16,0.5)*{};(16,5.5)*{} **\dir{-};
(0,0)*{\circ}; (8,0)*{\circ}; (16,0)*{\circ};
(24,0)*{\circ}; (32,0)*{\circ};
(40,0)*{\circ}; (48,0)*{\circ};
(16,6)*{\circ};
(0,-3)*{1}; (8,-3)*{2}; (16,-3)*{3};
(24,-3)*{5}; (32,-3)*{6}; (40,-3)*{7};
(48,-3)*{8}; (14,5)*{4};
\endxy
}

\newcommand{\DynkinF}
{
\xy
(0.5,0)*{};(7.5,0)*{} **\dir{-};
{\ar@{=>} (8.5,0)*{}; (15.5,0)*{}};
(16.5,0)*{};(23.5,0)*{} **\dir{-};
(0,0)*{\circ}; (8,0)*{\circ};
(16,0)*{\circ}; (24,0)*{\circ};
(0,-3)*{1}; (8,-3)*{2};
(16,-3)*{3};(24,-3)*{4};
\endxy
}

\newcommand{\DynkinG}
{
\xy
(0.5,0.5)*{};(7,0.5)*{} **\dir{-};
(0.5,0)*{};(7.5,0)*{} **\dir{-};
(0.5,-0.5)*{};(7,-0.5)*{} **\dir{-};
(6.5,1.5)*{};(7.5,0)*{} **\dir{-};
(6.5,-1.5)*{};(7.5,0)*{} **\dir{-};
(0,0)*{\circ}; (8,0)*{\circ};
(0,-3)*{1}; (8,-3)*{2};
\endxy
}

\newcommand{\exTriangleB}
{
\xy
(0,7.5)*{}; (25,7.5)*{} **\dir{-};
(0,2.5)*{}; (25,2.5)*{} **\dir{-};
(5,-2.5)*{}; (20,-2.5)*{} **\dir{-};
(10,-7.5)*{}; (15,-7.5)*{} **\dir{-};
(0,7.5)*{}; (0,2.5)*{} **\dir{-};
(5,7.5)*{}; (5,-2.5)*{} **\dir{-};
(10,7.5)*{}; (10,-7.5)*{} **\dir{-};
(15,7.5)*{}; (15,-7.5)*{} **\dir{-};
(20,7.5)*{}; (20,-2.5)*{} **\dir{-};
(25,7.5)*{}; (25,2.5)*{} **\dir{-};
(2.5,5)*{9};(7.5,5)*{8};(12.5,5)*{4};(17.5,5)*{2};(22.5,5)*{1};
(7.5,0)*{3};(12.5,0)*{1};(17.5,0)*{0};
(12.5,-5)*{2};
\endxy
}

\newcommand{\exTriangleBB}
{
\xy
(0,7.5)*{}; (25,7.5)*{} **\dir{-};
(0,2.5)*{}; (25,2.5)*{} **\dir{-};
(5,-2.5)*{}; (20,-2.5)*{} **\dir{-};
(10,-7.5)*{}; (15,-7.5)*{} **\dir{-};
(0,7.5)*{}; (0,2.5)*{} **\dir{-};
(5,7.5)*{}; (5,-2.5)*{} **\dir{-};
(10,7.5)*{}; (10,-7.5)*{} **\dir{-};
(15,7.5)*{}; (15,-7.5)*{} **\dir{-};
(20,7.5)*{}; (20,-2.5)*{} **\dir{-};
(25,7.5)*{}; (25,2.5)*{} **\dir{-};
(2.5,5)*{4};(7.5,5)*{3};(12.5,5)*{5};(17.5,5)*{2};(22.5,5)*{1};
(7.5,0)*{3};(12.5,0)*{3};(17.5,0)*{0};
(12.5,-5)*{3};
\endxy
}

\newcommand{\exTriangleBBi}
{
\xy
(0,7.5)*{}; (25,7.5)*{} **\dir{-};
(0,2.5)*{}; (25,2.5)*{} **\dir{-};
(5,-2.5)*{}; (20,-2.5)*{} **\dir{-};
(10,-7.5)*{}; (15,-7.5)*{} **\dir{-};
(0,7.5)*{}; (0,2.5)*{} **\dir{-};
(5,7.5)*{}; (5,-2.5)*{} **\dir{-};
(10,7.5)*{}; (10,-7.5)*{} **\dir{-};
(15,7.5)*{}; (15,-7.5)*{} **\dir{-};
(20,7.5)*{}; (20,-2.5)*{} **\dir{-};
(25,7.5)*{}; (25,2.5)*{} **\dir{-};
(2.5,5)*{5};(7.5,5)*{4};(12.5,5)*{7};(17.5,5)*{2};(22.5,5)*{1};
(7.5,0)*{2};(12.5,0)*{1};(17.5,0)*{0};
(12.5,-5)*{3};
\endxy
}

\newcommand{\exTriangleBBii}
{
\xy
(0,7.5)*{}; (25,7.5)*{} **\dir{-};
(0,2.5)*{}; (25,2.5)*{} **\dir{-};
(5,-2.5)*{}; (20,-2.5)*{} **\dir{-};
(10,-7.5)*{}; (15,-7.5)*{} **\dir{-};
(0,7.5)*{}; (0,2.5)*{} **\dir{-};
(5,7.5)*{}; (5,-2.5)*{} **\dir{-};
(10,7.5)*{}; (10,-7.5)*{} **\dir{-};
(15,7.5)*{}; (15,-7.5)*{} **\dir{-};
(20,7.5)*{}; (20,-2.5)*{} **\dir{-};
(25,7.5)*{}; (25,2.5)*{} **\dir{-};
(2.5,5)*{4};(7.5,5)*{3};(12.5,5)*{5};(17.5,5)*{2};(22.5,5)*{1};
(7.5,0)*{4};(12.5,0)*{3};(17.5,0)*{0};
(12.5,-5)*{3};
\endxy
}

\newcommand{\exTriangleD}
{
\xy
(0,7.5)*{}; (30,7.5)*{} **\dir{-};
(0,2.5)*{}; (30,2.5)*{} **\dir{-};
(5,-2.5)*{}; (25,-2.5)*{} **\dir{-};
(10,-7.5)*{}; (20,-7.5)*{} **\dir{-};
(0,7.5)*{}; (0,2.5)*{} **\dir{-};
(5,7.5)*{}; (5,-2.5)*{} **\dir{-};
(10,7.5)*{}; (10,-7.5)*{} **\dir{-};
(15,7.5)*{}; (15,-7.5)*{} **\dir{-};
(20,7.5)*{}; (20,-7.5)*{} **\dir{-};
(25,7.5)*{}; (25,-2.5)*{} **\dir{-};
(30,7.5)*{}; (30,2.5)*{} **\dir{-};
(2.5,5)*{9};(7.5,5)*{6};(12.5,5)*{4};(17.5,5)*{5};(22.5,5)*{3};(27.5,5)*{2};
(7.5,0)*{7};(12.5,0)*{4};(17.5,0)*{3};(22.5,0)*{1};
(12.5,-5)*{5};(17.5,-5)*{2};
\endxy
}

\newcommand{\Bs}[1] 
{
\xy
(-2.5,2.5)*{}; (2.5,2.5)*{} **\dir{-};
(-2.5,-2.5)*{}; (2.5,-2.5)*{} **\dir{-};
(-2.5,2.5)*{}; (-2.5,-2.5)*{} **\dir{-};
(2.5,2.5)*{}; (2.5,-2.5)*{} **\dir{-};
(0,0)*{#1};
\endxy
}

\newcommand{\Bl}[1] 
{
\xy
(-5,2.5)*{}; (5,2.5)*{} **\dir{-};
(-5,-2.5)*{}; (5,-2.5)*{} **\dir{-};
(-5,2.5)*{}; (-5,-2.5)*{} **\dir{-};
(5,2.5)*{}; (5,-2.5)*{} **\dir{-};
(0,0)*{#1};
\endxy
}

\newcommand{\CrystalA}
{
\xymatrix{
 \Bs{\overline{1}}  &  \Bs{\overline{2}} \ar[l]_{1} & \ar[l]_{2} \cdots  &
\Bs{\overline{n}} \ar[l]_{n-1} & \Bl{\overline{n+1}} \ar[l]_{n}
}
}

\newcommand{\CrystalB}
{

\xymatrix{
 \Bs{\overline{1}}   & \ar[l]_{1}  \cdots  & \ar[l]_{n-1}
\Bs{\overline{n}}  & \ar[l]_n \Bs{0}  & \ar[l]_n
\Bs{n}  & \ar[l]_{n-1} \cdots &  \ar[l]_{1} \Bs{1}
}
}

\newcommand{\CrystalC}
{
\xymatrix{
 \Bs{\overline{1}}   & \ar[l]_1  \cdots  & \ar[l]_{n-2\ \ \ }
\Bl{\overline{n-1}} & \ar[l]_{\ \ \ n-1} \Bs{\overline{n}}  & \ar[l]_n
\Bs{n} & \ar[l]_{n-1\ \ \ } \Bl{n-1} & \ar[l]_{\ \ \ n-2}
 \cdots & \ar[l]_{1} \Bs{1}
}
}

\newcommand{\CrystalD}
{

\xymatrix{
 & & & \ar[dl]_{n} \Bs{n} &  & & \\
\Bs{\overline{1}}   & \ar[l]_{1}  \cdots  & \ar[l]_{n-2 \ \ \ \ }
\Bl{\overline{n-1}}  & & \ar[ul]_{n-1} \ar[dl]^{n} \Bl{n-1} & \ar[l]_{\ \ \ \ n-2} \cdots & \ar[l]_{1} \Bs{1} \\
& & & \ar[ul]^{n-1} \Bs{\overline{n}}  &  & &
}
}

\newcommand{\CrystalDiagram}
{
\xymatrix{
B(\lambda) \ar@{_(->}[dd]^{\imath_\lambda} \ar[rr]^{\sim}_{\Psi^\lambda} & & \B(\lambda) \ar@{_(->}[dd]^{\jmath_\lambda} \\
&  & \\
B(\infty) \otimes T_\lambda \otimes C \ar[rr]^{\sim}_{\Psi \otimes {\rm id} \otimes {\rm id}} & & \B(\infty) \otimes T_\lambda \otimes C
}
}

\newcommand{\KNT}[9]
{
\xy
(0,7.5)*{}; (17.5,7.5)*{} **\dir{-};
(0,2.5)*{}; (17.5,2.5)*{} **\dir{-};
(0,-2.5)*{}; (12.5,-2.5)*{} **\dir{-};
(0,-7.5)*{}; (7.5,-7.5)*{} **\dir{-};
(0,7.5)*{}; (0,-7.5)*{} **\dir{-};
(2.5,7.5)*{}; (2.5,-7.5)*{} **\dir{-};
(7.5,7.5)*{}; (7.5,-7.5)*{} **\dir{-};
(12.5,7.5)*{}; (12.5,-2.5)*{} **\dir{-};
(17.5,7.5)*{}; (17.5, 2.5)*{} **\dir{-};
(15,5)*{#1};
(10,5)*{#2}; (10,0)*{#3};
(5,5)*{#4}; (5,0)*{#5}; (5,-5)*{#6};
(1.4,5)*{#7}; (1.4,0)*{#8}; (1.4,-5)*{#9};
\endxy
}

\newcommand{\CrystalBf}
{
\xymatrix{
 \Bs{\overline{1}}   &
 \ar[l]_{1}  \Bs{\overline{2}}  &
 \ar[l]_{2}  \Bs{\overline{3}}  &
 \ar[l]_{3}  \Bs{0}  &
 \ar[l]_3 \Bs{3}  &
 \ar[l]_2 \Bs{2}  &
 \ar[l]_{1} \Bs{1}
}
}

\newcommand{\CrystalBfi}
{
\xymatrix{
 \Bs{\overline{1}}   &
 \ar[l]_{1}  \Bs{\overline{2}}
}
}

\newcommand{\CrystalBfii}
{
\xymatrix{
 \Bs{\overline{1}}   &
 \ar[l]_{1}  \Bs{\overline{2}}  &
 \ar[l]_{2}  \Bs{\overline{3}}  &
 \ar[l]_{3}  \Bs{0}
}
}

\newcommand{\CrystalBfiii}
{
\xymatrix{
 \Bs{\overline{1}}  &
 \ar[l]_{1}  \Bs{\overline{2}}  &
 \ar[l]_{2}  \Bs{\overline{3}}  &
 \ar[l]_{3}  \Bs{0}  &
 \ar[l]_3 \Bs{3} &
 \ar[l]_2 \Bs{2}
}
}

\newcommand{\CrystalBfiv}
{
\xymatrix{
 \Bs{\overline{1}}  &
 \ar[l]_{1}  \Bs{\overline{2}} &
 \ar[l]_{2}  \Bs{\overline{3}}  &
 \ar[l]_{3}  \Bs{0}  &
 \ar[l]_3 \Bs{3}  &
 \ar[l]_2 \Bs{2}  &
 \ar[l]_{1} \Bs{1}
}
}

\begin{document}

\title[Representations over KLR algebras of finite classical type]
{Construction of irreducible representations \\ over
Khovanov-Lauda-Rouquier algebras \\ of finite classical type  }

\author[Georgia Benkart]{Georgia Benkart}
\address{Department of Mathematics, University of Wisconsin-Madison, 480 Lincoln Drive, Madison, WI 53706, USA}
\email{benkart@math.wisc.edu}
\author[Seok-Jin Kang]{Seok-Jin Kang$^{1,2}$ }
\thanks{$^1$ This work was supported by KRF Grant \# 2007-341-C00001.}
\thanks{$^2$ This work was supported by NRF Grant \# 2010-0010753.}
\address{Department of Mathematical Sciences and Research Institute of Mathematics,
Seoul National University, 599 Gwanak-ro, Gwanak-gu, Seoul 151-747, Korea} \email{sjkang@snu.ac.kr}
\author[Se-jin Oh]{Se-jin Oh$^{3,4}$}
\thanks{$^3$ This work was supported by NRF Grant \# 2010-0019516.}
\thanks{$^4$ This work was supported by BK21 Mathematical Sciences Division.}
\address{Department of Mathematical Sciences, Seoul National University,
599 Gwanak-ro, Gwanak-gu, Seoul 151-747, Korea} \email{sj092@snu.ac.kr}
\author[Euiyong Park]{Euiyong Park}
\address{Research Institute of Mathematics, Seoul National University, 599 Gwanak-ro, Gwanak-gu, Seoul 151-747, Korea}
\email{pwy@snu.ac.kr}

\maketitle

\begin{center}
{\it In Memory of Professor Hyo Chul Myung}
\end{center}

\begin{abstract}
We give an explicit construction of irreducible modules over
Khovanov-Lauda-Rouquier algebras $R$ and their cyclotomic quotients
$R^{\lambda}$ for finite classical types using a crystal basis
theoretic approach. More precisely, for each element $v$ of the
crystal $B(\infty)$ (resp.\ $B(\lambda)$), we first construct
certain modules $\Delta(\mathbf{a};k)$ labeled by the adapted string
$\mathbf{a}$ of $v$. We then prove that the head of the induced
module  $\ind \big(\Delta(\mathbf{a};1) \boxtimes \cdots \boxtimes
\Delta(\mathbf{a};n)\big)$ is irreducible and that every irreducible
$R$-module (resp.\ $R^{\lambda}$-module) can be realized as the
irreducible head of one of the induced modules $\ind
(\Delta(\mathbf{a};1) \boxtimes \cdots \boxtimes
\Delta(\mathbf{a};n))$.   Moreover, we show that our construction is
compatible with the crystal structure on $B(\infty)$ (resp.\
$B(\lambda)$).
\end{abstract}

\section*{Introduction}

The {\it Khovanov-Lauda-Rouquier algebras} ({\it KLR algebras}) were
introduced independently by Khovanov-Lauda \cite{KL09,KL11} and
Rouquier \cite{R08} to give a categorification of quantum groups.
Let $U_q(\g)$ be the quantum group associated with a symmetrizable
Cartan datum,  and let $R$ be the corresponding KLR algebra. For a
dominant integral weight $\lambda$ of $U_q(\g)$, the algebra $R$ has
a special quotient $R^{\lambda}$ corresponding to $\lambda$, which
is called the {\it cyclotomic Khovanov-Lauda-Rouquier algebra} ({\it
cyclotomic KLR algebra}) of weight $\lambda$. It was conjectured
that the cyclotomic quotient $R^{\lambda}$ gives a categorification
of the irreducible highest weight $U_q(\g)$-module $V(\lambda)$.
This was shown for type $\mathsf{A}$ in \cite{BK08,BK09,BS08}.
In \cite{KK11}, Kang and Kashiwara proved this conjecture for all
symmetrizable Cartan data.   Webster  \cite{Webster10} has announced
a categorification of tensor products of highest weight modules.  In
a recent paper, Kang, Oh and Park  \cite{KOP11a}
extended the study of KLR algebras  to provide a categorification of
quantum generalized Kac-Moody algebras and their crystals. Moreover,
Kang, Kashiwara and Oh \cite{KKO11} proved the cyclotomic
categorification conjecture for irreducible highest weight modules
over quantum generalized Kac-Moody algebras.

For symmetrizable Cartan data, the set $\B(\infty)$ (resp.\ $\B(\lambda)$) of
isomorphism classes of finite-dimensional irreducible graded modules
over $R$ (resp.\ $R^{\lambda}$) can be given a crystal structure,
and Lauda and Vazirani \cite{LV09} have shown that there exist
crystal isomorphisms $ B(\infty) \buildrel \sim \over
\longrightarrow \B(\infty)$ and $ B(\lambda) \buildrel \sim \over
\longrightarrow \B(\lambda)$, where $B(\infty)$ (resp.\
$B(\lambda)$) is the crystal of $U_q^-(\g)$ (resp.~$V(\lambda)$).
Kleshchev and Ram \cite{KR09}  gave a  construction of
irreducible graded $R$-modules for all finite types by using the
combinatorics of Lyndon words to construct irreducible
$R$-modules as the irreducible heads of the induced modules of the outer
tensor products of {\it cuspidal modules}.  In this approach, the action of Kashiwara operators
on the crystal of irreducible modules  is hidden in the combinatorics of Lyndon words.   Hill,
Melvin and Mondragon \cite{HMM09} completed the classification of irreducible
$R$-modules begun by Kleshchev and Ram by determining the cuspidal
modules.
Recently, McNamara \cite{McN12} investigated this approach using PBW basis theory and computed the global dimension of $R$ for finite type.
However, it remains
 an open problem to construct irreducible graded
$R^\lambda$-modules.
For finite and affine type $\mathsf{A}$,
results in \cite{BK08, BKW11, HM10} tell us
that the algebra $R^\lambda$ has
a cellular basis, which yields irreducible modules by using the
cellular basis techniques introduced in \cite{GL96}.

In this paper, we give an explicit construction of  all irreducible
graded modules over $R$ and $R^\lambda$ for KLR algebras of finite classical type as
the irreducible heads of certain induced modules. This generalizes the
type $\mathsf{A}_n$ result in \cite{KP10} to all finite classical types. Our construction differs
from the one given by Kleshchev and Ram and is based on the theory of crystal bases.
As a result,  the action of the Kashiwara operators is an integral part of the construction.

Here is a brief description of our work. Let $(\mathfrak{A}, \mathsf{P}, \Pi,
\mathsf{P}^{\vee}, \Pi^{\vee})$ be a symmetrizable Cartan datum. Let
$I$ be the index set of the simple roots, and let $n = |I|$. Set $I_{(n+1)} = I$ and take
$I_{(k)}\ (k=1,\ldots,n)$ to be subsets of $I$ such that $I_{(k)}
\subset I_{(k+1)}$ and $|I_{(k)}|=k$ for all $k$. Let
$\mathfrak{B}_k$ be the crystal obtained from $B(\infty)$ by
forgetting the $i$-arrows for $i \notin I_{(k)}$. For $v \in
B(\infty)$, let $u_0 = v$ and let $u_k$ be the highest weight vector
of the connected component of $\mathfrak{B}_k$ containing $v$ for
$k=1,\ldots,n$. Then, there exists a sequence $\mathbf{i}_k$ of
elements in $I$ such that $ u_{k-1} = \kf_{\mathbf{i}_k}u_k $,
where $\kf_{\mathbf{i}_k}$ is a product of the Kashiwara operators corresponding
to the terms in the sequence $\mathbf{i}_k$.   For
$k=1,\ldots,n$, define
$$ \mathcal{N}_k(v) = \kf_{\mathbf{i}_k} \mathbf{1},$$
where $\mathbf{1}$ is the trivial for the KLR algebra $R(0) := \C$. In Proposition
\ref{Prop: crystal iso},  we prove for any symmetrizable Cartan datum
that $\hd\,\ind ( \mathcal{N}_1(v) \boxtimes \cdots \boxtimes \mathcal{N}_n(v)
)$  is an irreducible graded $R$-module and that the map $\Phi: B(\infty)
\longrightarrow \B(\infty)$ defined by
\begin{equation}\label{eq:crystaliso}  \Phi(v) = \hd\,\ind\left( \mathcal{N}_1(v) \boxtimes \cdots \boxtimes \mathcal{N}_n(v)\right) \ \  \text{ for } v \in B(\infty)
\end{equation}
is a crystal isomorphism.

When $\mathfrak{A}$ is of finite type, the modules
 $\mathcal{N}_k(v)$ can be given a more explicit description in terms of the Kashiwara operators
$\kf_i$ via the {\it adapted strings} introduced by Littelmann in
\cite{Lit98}.   For this, we  choose a special expression  $w_0= r_{\mathbf{s}_1} \cdots r_{\mathbf{s}_n}$
of the longest
element $w_0$  in the Weyl group $W$.  Here the $\mathbf{s}_k$ are sequences of
indices in $I$ (see Table \ref{Tb: w0}).
Using the description of adapted strings given in
\cite{Lit98}, we prove  for $v \in
B(\infty)$ that
$$ \mathcal{N}_k(v) = \kf_{\mathbf{s}_k} ^{\mathbf{a}(v)_k} \mathbf{1},  $$
where $\mathbf{a}(v)$ is the adapted string of $v$ with respect to this
expression for $w_0$,
and ${\mathbf{a}(v)_k}$ is the subsequence of $\mathbf{a}(v)$ defined by $\eqref{Eq: subseq of a}$ (see Proposition \ref{Prop: description of Nk} below).

Our main result  is an explicit construction of the
irreducible graded modules over the KLR algebras $R$ and $R^\lambda$ of finite
classical type $\mathsf{A}_n, \mathsf{B}_n, \mathsf{C}_n, \mathsf{D}_n$ in terms of adapted strings. Let $\mathcal{S}$ (resp.\
$\mathcal{S}^\lambda$) be the set of adapted strings of $B(\infty)$
(resp.\ $B(\lambda)$) given in Proposition \ref{Prop: S lambda}, and
let $\cB$ be the crystal of the irreducible module $V(\Lambda_n)$ labeled by
the fundamental weight $\Lambda_n$  if $\mathfrak{A}$ is of
type $\mathsf{A}_n$ (resp. $V(\Lambda_1)$ if $\mathfrak{A}$ is of type $\mathsf{B}_n,
\mathsf{C}_n, \mathsf{D}_n$). For $a,b \in \cB$ with $a \succ b$, using the structure
of the crystal $\cB$,  in  \eqref{Eq: def of nabla} we associate to $a,b$ an irreducible graded module $\Delta_{(a,b)}$.
The modules $\Delta_{(a,b)}$ are 1- or 2-dimensional, and in general they are {\it not} the
cuspidal modules found in \cite{HMM09,KR09}. Using the description
of $\mathcal{S}$ (resp.\ $\mathcal{S}^\lambda$), we define for $v \in B(\infty)$ (resp.\ $v \in B(\lambda)$)
 the module $\Delta(\mathbf{a}(v);k)$ to be the
outer tensor product of modules $\Delta_{(a,b)}$ as in
\eqref{Eq: nabla}. Then, it follows from Lemma \ref{Lem: gluing
lemma} that
\begin{equation}\label{eq:NhdInd} \mathcal{N}_k(v) = \hd\,\ind \Delta(\mathbf{a}(v);k)\ \ \text{ for } \  k=1,\ldots,n. \end{equation}

The description of the irreducible modules in  \eqref{eq:crystaliso} works for all symmetrizable Cartan data.
It follows from  \eqref{eq:crystaliso}  and \eqref{eq:NhdInd} that the irreducible modules for
the KLR algebras of finite classical type can be obtained from taking the outer tensor product
of heads of induced modules, inducing, then taking the head of the induced module.
But  this unduly complicated process can be simplified.  Indeed, we prove for the module $\Delta(\mathbf{a}(v)): =
\Delta(\mathbf{a}(v);1) \boxtimes
\cdots \boxtimes \Delta(\mathbf{a}(v);n)$ that
$\hd\,\ind \Delta(\mathbf{a}(v))$ is irreducible for  $v
\in B(\infty)$ (resp.\ $v \in B(\lambda)$) for all finite classical types and that the maps
\begin{align*}
\Psi:& B(\infty) \longrightarrow \B(\infty)\ \  \ \text{ given  by }    \Psi(v) = \hd\,\ind\,\Delta(\mathbf{a}(v)) \  \text{ for } \ v \in B(\infty),\\
\Psi^\lambda:&  B(\lambda)\ \longrightarrow \B(\lambda)\ \ \ \  \text{ given by }   \Psi^\lambda(v) = \hd\,\ind\,\Delta(\mathbf{a}(v)) \   \text{ for } \  v \in B(\lambda)
\end{align*}
are crystal isomorphisms (Theorem \ref{Thm: main theorem}).  Using
that fact, we show for the finite classical types that
\begin{align*}
\mathcal{A} &= \{ \hd\,\ind \Delta(\mathbf{a})\mid \mathbf{a} \in \mathcal{S}  \} \ \  \hbox{\rm and}  \\
\mathcal{A}^\lambda &= \{ \hd\,\ind \Delta(\mathbf{a} )\mid \mathbf{a} \in \mathcal{S}^\lambda \},
\end{align*}
as $\mathbf{a}$ ranges over the adapted strings,  provide complete
lists of all the irreducible graded modules over  $R$ and
$R^\lambda$, respectively, up to isomorphism and grading shift
(Corollary \ref{Cor: complete lists}).

Our paper is organized as follows. Section 1 contains a brief
review of crystal bases and KLR algebras
associated with any symmetrizable Cartan datum.  This section culminates with the proof of the
crystal isomorphism $\Phi: B(\infty) \longrightarrow \B(\infty)$
in \eqref{eq:crystaliso}.  Section 2 specializes to the case of finite type,  first
reviewing  Littelmann's result (see  \cite{Lit98})  on adapted stings
and  then giving an
explicit description of the modules $\mathcal{N}_k(v)$ for all finite types via adapted
stings.   Combining this description with the definition of
$\mathcal{N}_k(v)$,  we obtain an expression for $\mathcal{N}_k(v)$ in terms of Kashiwara
operators.

In Section 3, we restrict to the case of finite classical types. We
first define the modules $\Delta_{(a,b)}$ using the structure of the
crystals $\cB$ and construct the module $\Delta(\mathbf{a}; k)$ for
$\mathbf{a} \in \mathcal{S}$ (resp.\ $\mathbf{a} \in
\mathcal{S}^\lambda$)  as an outer tensor product of modules
$\Delta_{(a,b)}$  in $\eqref{Eq: nabla}$. We then construct the maps
$\Psi: B(\infty) \longrightarrow \B(\infty)$ and
$\Psi^\lambda: B(\lambda) \longrightarrow \B(\lambda)$ by taking the
head of $\ind \Delta(\mathbf{a}(v))$ for $v \in B(\infty)$ (resp.\
$v \in B(\lambda)$). We illustrate this
construction by presenting an example for type $\mathsf{B}_3$ using
the Kashiwara-Nakashima tableaux  in \cite{KN94} to realize the
crystal $B(\lambda)$.  We prove in Proposition  \ref{Lem: upper bound}  for
any finite classical type and any $v \in B(\lambda)$ that the number $\eta(v)$
of $\Delta_{(a,b)}$'s in
 $\Delta(\mathbf{a}(v))$
has an upper bound; i.e., $ \eta(v) \le n \lambda(h) $ for a certain
element $h \in \mathsf{P}^{\vee}$ (which depends on the type).

Section 4 is devoted to
proving that the maps $\Psi: B(\infty) \longrightarrow
\B(\infty)$ and $\Psi^\lambda: B(\lambda) \longrightarrow
\B(\lambda)$ are crystal isomorphisms.  To accomplish this, we give a sufficient condition in Lemma \ref{Lem: commuting lemma}  for the isomorphism
$\ind ( \Delta_{(a,b)} \boxtimes \Delta_{(c,d)} ) \simeq \ind (
\Delta_{(c,d)} \boxtimes \Delta_{(a,b)})$ to exist for finite classical types. Using this condition  together with
Lemma 4.3 of \cite{KP10} and the surjective homomorphism $\ind
(\Delta_{(a,b)} \boxtimes \Delta_{(b,c)}) \twoheadrightarrow
\Delta_{(a,c)}$, we prove in Lemma \ref{Lem: gluing lemma}  that $\mathcal{N}_n(v)  = \hd\,\ind
\Delta(\mathbf{a}(v);n)$.  It follows from that result
and the choice of the
sequence corresponding to $k$ that $\mathcal{N}_k(v) = \hd\,\ind
\Delta(\mathbf{a}(v);k)$ for all $k=1,\ldots,n$. Combining this
with the crystal isomorphism $\Phi$ in \eqref{eq:crystaliso}, we establish the crystal
isomorphisms $\Psi: B(\infty) \longrightarrow \B(\infty)$ and
$\Psi^\lambda: B(\lambda) \longrightarrow \B(\lambda)$.  This then
gives an explicit realization compatible with the Kashiwara operators  of
all the irreducible graded modules over $R$ and $R^\lambda$ for finite
classical types.  \vspace{-.1truecm}

\section{Crystals and Khovanov-Lauda-Rouquier Algebras} \label{Sec: Crystals and KLR algs}
 \vspace{-.1truecm}
\subsection{Crystals}\

 Let $I$ be a finite index set. A square matrix $\mathfrak{A} =
(a_{ij})_{i,j \in I}$ is a {\it symmetrizable generalized Cartan
matrix} if it satisfies (i) $a_{ii}=2$ for $i\in I$, (ii) $a_{ij}
\in \Z_{\le 0}$ for $i \ne j$, (iii) $a_{ij} = 0 $ if $a_{ji}=0$ for
$i,j \in I$, (iv) there is a diagonal matrix $\mathfrak{D} = {\rm
diag}(\delta_i \in \Z_{>0} \mid i\in I)$ such that
$\mathfrak{D}\mathfrak{A}$ is symmetric.

A {\it Cartan datum} $(\mathfrak{A}, \mathsf{P}, \Pi, \mathsf{P}^{\vee}, \Pi^{\vee})$ consists of  \begin{itemize}
\item[(1)] a symmetrizable generalized Cartan matrix $\mathfrak{A}$,
\item[(2)] a free abelian group $\mathsf{P}$ of finite rank, called the {\it weight lattice},
\item[(3)] the set $\Pi = \{ \alpha_i \mid i\in I \} \subset \mathsf{P}$ of {\it simple roots},
\item[(4)] the  {\it dual weight lattice}  $\mathsf{P}^{\vee} := \Hom(\mathsf{P}, \Z)$,
\item[(5)] the set $\Pi^{\vee} = \{ h_i \mid i\in I\} \subset \mathsf{P}^{\vee}$ of {\it simple coroots},
\end{itemize}
which satisfy the following properties:
\begin{itemize}
\item[(i)] $\langle h_i, \alpha_j \rangle := \alpha_j(h_i) = a_{ij}$ for all $i,j\in I$,
\item[(ii)] $\Pi \subset \h^*$ is linearly independent, where $\h := \C \otimes_{\Z} \mathsf{P}^{\vee}$ and $\h^*$ is the dual space,
\item[(iii)] for each $i \in I$, there exists $\Lambda_i \in \mathsf{P}$ such that $\langle h_j, \Lambda_i \rangle = \delta_{ij}$ for all $j\in I$.
\end{itemize}

The $\Lambda_i$ are the {\it fundamental weights}.  We denote by $\mathsf{P}^+ = \{ \lambda \in \mathsf{P} \mid \lambda(h_i) \in \Z_{\ge 0}, \ i \in I \}$ the set of
{\it dominant integral weights}.
The free abelian group $\mathsf{Q} = \bigoplus_{i \in I} \Z \alpha_i$ is the {\it root lattice},  and $\mathsf{Q}^+ = \sum_{i \in I} \Z_{\ge 0} \alpha_i$
is the {\it positive root lattice}. For $\alpha=\sum_{i \in I} k_i \alpha_i \in \mathsf{Q}^+$,  the \emph{height} of $\alpha$
is $|\alpha|:=\sum_{i \in I} k_i$.
There is a symmetric bilinear form $(\ | \ )$ on $\h^*$ such that
$$(\alpha_i| \alpha_j) = \delta_i a_{ij} \text{ for } i,j\in I, \quad \quad \langle h_i, \lambda \rangle =
\frac{2(\alpha_i | \lambda )}{(\alpha_i | \alpha_i)} \ \text{ for } \lambda \in \h^* \text{ and } i\in I.$$
Let $W$ be the {\it Weyl group}, which is the subgroup of ${ \rm Aut}(\h^*)$ generated by simple reflections $\{r_i\}_{i\in I}$
defined by $r_i(\lambda) := \lambda - \langle h_i, \lambda \rangle \alpha_i$ for $\lambda \in \h^*$ and $i\in I$.

Let $q$ be an indeterminate. For $i\in I$ and $m,n \in \Z_{\ge 0}$, define
$$ q_i=q^{\delta_i},\ \
 [n]_{q_i} =\frac{ {q_i}^n - {q_i}^{-n} }{ {q_i} - {q_i}^{-1} }, \ \
[n]_{q_i}! = \prod^{n}_{k=1} [k]_{q_i} , \ \
\left[\begin{matrix}m \\ n\\ \end{matrix} \right]_{q_i}=  \frac{ [m]_{q_i}! }{[m-n]_{q_i}! [n]_{q_i}! }.
$$

\begin{definition}
The {\em quantum group} $U_q(\g)$ associated
with the Cartan datum $(\mathfrak{A},\mathsf{P}, \Pi,\mathsf{P}^{\vee},\Pi^{\vee})$ is the associative
algebra over $\Q(q)$ with $1$ generated by $e_i,f_i$ $(i \in I)$ and
$q^{h}$ $(h \in \mathsf{P}^{\vee})$ satisfying the following relations:
\begin{enumerate}
  \item  $q^0=1, \ q^{h} q^{h'}=q^{h+h'} $ for $ h,h' \in \mathsf{P}^{\vee},$
  \item  $q^{h}e_i q^{-h}= q^{ \langle h,\alpha_i \rangle} e_i,
          \ q^{h}f_i q^{-h} = q^{- \langle h,\alpha_i \rangle }f_i$ for $h \in \mathsf{P}^{\vee}, i \in I$,
  \item  $e_if_j - f_je_i =  \delta_{ij} \dfrac{K_i -K^{-1}_i}{q_i- q^{-1}_i }, \ \ \mbox{ where } K_i=q^{\delta_i h_i},$
  \item  $\displaystyle \sum^{1-a_{ij}}_{k=0}(-1)^k \left[\begin{matrix}1-{a}_{ij} \\ k\\ \end{matrix} \right]_{q_i} e^{1-{a}_{ij}-k}_i
         e_j e^{k}_i =0 \quad \text{ if }  i \ne j, $
  \item $\displaystyle \sum^{1-a_{ij}}_{k=0}(-1)^k  \left[\begin{matrix}1-{a}_{ij} \\ k\\ \end{matrix} \right]_{q_i} f^{1-{a}_{ij}-k}_if_j
        f^{k}_i=0 \quad \text{ if }  i \ne j. $
\end{enumerate}
\end{definition}

The definition of the {\it category} $O_{int}^q$ of integrable $U_q(\g)$-modules, {\it crystal bases}, and {\it Kashiwara operators} can be found, for example, in \cite{HK02,Kash91}.
It was proved in \cite{Kash91} that every $U_q(\g)$-module in the category $O_{int}^q$ has a unique crystal basis $(L,B)$.
Let us recall the notion of a crystal as defined in \cite{Kash93a}.

\begin{definition}
A {\em crystal} is a set $B$ together with maps
$\  \mathsf{ wt} : B \to \mathsf{P},\ \varphi_{i},\varepsilon_{i} : B \to \Z \sqcup \{-\infty \} \mbox{ and } \
\tilde{e}_i,\tilde{f}_i : B \to B \sqcup \{0\} \ (i \in I) $ which satisfy the following conditions:
\begin{enumerate}
\item $\varphi_{i}(b) = \varepsilon_{i}(b) + \langle h_i, \mathsf{ wt}(b) \rangle,$
\item $ \mathsf{ wt}(\tilde{e}_i b)= \mathsf{ wt}(b)+\alpha_i,  \ \ \mathsf{ wt}(\tilde{f}_i b)= \mathsf{ wt}(b)-\alpha_i \mbox{ if } \tilde{e}_i b, \tilde{f}_i b \in B,$
\item for $ b,b^{\prime} \in B $ and $ i \in I,$\ $ b'=\tilde{e}_i b$ if and only if $b = \tilde{f}_i b',$
\item for $ b \in B$, if $ \varphi_{i}(b) = -\infty$, then $ \tilde{e}_i b = \tilde{f}_i b=0,$
\item if $ b \in B $ and $\tilde{e}_i b \in B$, then
        $\varepsilon_{i}(\tilde{e}_i b)= \varepsilon_{i}(b)-1, \ \  \varphi_{i}(\tilde{e}_i b)= \varphi_{i}(b)+1 $,
\item if $ b \in B$ and $\tilde{f}_i b \in B$, then
         $\varepsilon_{i}(\tilde{f}_i b)=\varepsilon_{i}(b)+1 ,
       \ \  \varphi_{i}(\tilde{f}_i b)=\varphi_{i}(b)-1. $
\end{enumerate}
\end{definition}
For $\mathbf{i}=(i_1, \ldots, i_m) \in I^m$ and $\mathbf{k} = (k_1, \ldots, k_m) \in (\Z_{\ge0})^m$,
let $\kf_{\mathbf{i}}^{\mathbf{k}} = \kf_{i_1}^{k_1} \cdots \kf_{i_m}^{k_m}$ (resp.\ $\ke_{\mathbf{i}}^{\mathbf{k}}
 = \ke_{i_1}^{k_1} \cdots \ke_{i_m}^{k_m}$). If $\mathbf{k} = (1, \ldots,1)$, then we write $\kf_{\mathbf{i}}$ (resp.\ $\ke_{\mathbf{i}}$)
for $\kf_{\mathbf{i}}^{\mathbf{k}}$ (resp.\ $\ke_{\mathbf{i}}^{\mathbf{k}}$).

\begin{examples} \  \begin{enumerate}
\item Associated with each module $M \in O_{int}^q$ is a crystal basis  $(L,B)$, where  $B$ is a crystal
with the maps
$$ \varepsilon_i(b) = \mathsf{max}\{ k\ge0 \mid \tilde e_i^k b \ne 0 \},
\qquad \varphi_i(b) =  \mathsf{max}\{ k\ge0 \mid \tilde f_i^k b \ne 0 \}. $$
We denote by $B(\lambda)$ the crystal of the irreducible
highest weight module $V(\lambda) \in O_{int}^q$ with highest weight $\lambda \in
\mathsf{P}^+$ and write $b_\lambda$ for the highest weight element of $B(\lambda)$.

\item Let $(L(\infty), B(\infty))$ be the crystal basis of $U_q^-(\g)$.
Then $B(\infty)$ is a crystal with the maps given by
$$ \varepsilon_i(b) =  \mathsf{max}\{ k\ge0 \mid \tilde e_i^k b \ne 0 \}, \qquad \varphi_i(b)=\varepsilon_i(b)+\langle h_i, \wt(b) \rangle .$$
Let $\mathsf{1}$ denote the highest weight element of $B(\infty)$.

\item For $\lambda \in \mathsf{P}$, the set $T_\lambda=\{ t_\lambda \}$ is a crystal with  \begin{align}
& \wt(t_\lambda)=\lambda,\quad \tilde e_i t_\lambda = \tilde f_i t_\lambda = 0 \  \text{  for  }\  i\in I, \nonumber \\
& \varepsilon_i(t_\lambda) = \varphi_i(t_\lambda) = -\infty \text{  for  } i \in I. \nonumber
\end{align}

\item The set $C = \{c \}$ is a crystal with
$$ \wt(c)=0,\quad \tilde e_i c= \tilde f_i c = 0, \quad \varepsilon_i(c) =  \varphi_i(c) = 0 \ \  (i \in I).$$
\end{enumerate}
\end{examples}

We refer to \cite{HK02,Kash93a} for more details on $U_q(\g)$-crystals.
It was shown in \cite{JKKS07,Kash93a} that for each $\lambda \in \mathsf{P}^+$ there is a unique strict crystal embedding
\begin{align} \label{Eq: crystal embedding}
\iota_\lambda : B(\lambda) \rightarrow B(\infty)\otimes T_\lambda \otimes C
\end{align}
sending $b_{\lambda}$ to $\mathsf{1} \otimes t_{\lambda} \otimes c$.

\subsection{Khovanov-Lauda-Rouquier algebras}\

For $\alpha \in \mathsf{Q}^+$ with $|\alpha|=m$, we define
$ I^\alpha = \{ {\bf i}=(i_1,\ldots,i_m) \in I^{m} \mid \sum_{k=1}^{m} \alpha_{i_k}=\alpha \}.$
Given $ \mathbf{i}=(i_1,\ldots,i_m) \in I^\alpha$ and
$ \mathbf{j}=(j_1,\ldots,j_{m'}) \in I^\beta$, let ${\bf i}*{\bf j}$ denote the concatenation of ${\bf i}$ and ${\bf j}$:
${\bf i}*{\bf j}=(i_1,\ldots,i_m,j_1,\ldots,j_{m'}) \in I^{\alpha+\beta}.$
Let $S_m $ be the symmetric group on $m$ letters with simple transpositions $\sigma_i\  (i=1,\ldots,m-1)$. Then
$S_m$ acts on $I^m$ in a natural way. For $d_1,\ldots,d_n \in \Z_{\ge 0}$, we denote
by $S_{d_1+\cdots +d_n} /S_{d_1} \times \cdots \times S_{d_n}$ the set of minimal left coset representatives of
$S_{d_1} \times \cdots \times S_{d_n}$ in $S_{d_1+\cdots +d_n}$.

Let $\mathsf{u},\mathsf{v}$ be indeterminates.
For each $i,j\in I$, we choose $ \zeta_{ij} \in \C \setminus \{0\}$ such that $\zeta_{ij} = \zeta_{ji}$ if ${a}_{ij}=0$
and elements $\eta_{ij}  \in \C$ with $\eta_{ij}^{pq} = \eta_{ji}^{qp}$ for all $p,q \in \mathbb Z_{>0}$ such that
$\delta_ip +\delta_jq = -(\alpha_i | \alpha_j)$ and set
\begin{align*}
\mathcal{Q}_{ij}(\mathsf{u},\mathsf{v}) =
\begin{cases}
0 & \text{ if } i=j, \\
\zeta_{ij} & \text{ if } i \ne j,\ {a}_{ij} =0, \\
\zeta_{ij} \mathsf{u}^{-{a}_{ij}} + {\displaystyle \sum_{ \substack{ p,q > 0,\\ \delta_ip + \delta_jq = -(\alpha_i | \alpha_j)}}} \eta_{ij}^{pq} \mathsf{u}^p \mathsf{v}^q + \zeta_{ji} \mathsf{v}^{-{a}_{ji}} & \text{ otherwise.}
\end{cases}
\end{align*}

\begin{definition} \
\begin{enumerate}
\item
Let $\alpha \in \mathsf{Q}^+$ with $|\alpha|=m$.
The {\it homogeneous Khovanov-Lauda-Rouquier algebra} $R(\alpha)$ at $\alpha$ associated with
$\mathfrak{A}$ and $(\mathcal{Q}_{ij})_{i,j \in I}$ is the associative graded $\C$-algebra generated by $\e(\mathbf{i})$ ($\mathbf{i} =(i_1,\ldots,i_m) \in I^\alpha$),
$x_\ell$ ($1\le \ell \le m$), $\tau_k$ ($1\le k < m$) satisfying the following defining relations:
\begin{align*}
& \e(\mathbf{i})\e(\mathbf{j})=\delta_{\mathbf{i},\mathbf{j}}\e(\mathbf{i}), \ \   \sum_{\mathbf{i} \in I^\alpha}\e(\mathbf{i})=1, \ \
x_k x_\ell=x_\ell x_k, \ \   x_\ell \e(\mathbf{i}) = \e(\mathbf{i}) x_\ell, \\
& \tau_k e(\mathbf{i}) = e(\sigma_k(\mathbf{i})) \tau_k , \ \   \tau_k\tau_\ell = \tau_\ell \tau_k \ \text{ if } |k-\ell| > 1, \\
&  \tau_k^2 \e(\mathbf{i}) =\mathcal{Q}_{i_k,i_{k+1}}(x_k,x_{k+1})\e(\mathbf{i}),  \\
& (\tau_k x_\ell-x_{\sigma_k(\ell)}\tau_k) \e(\mathbf{i}) =
            \begin{cases}
             - \e(\mathbf{i}) & \text{ if } \ell=k,\ \  i_k=i_{k+1}, \\
            \e(\mathbf{i}) & \text{ if } \ell=k+1,\ \  i_k=i_{k+1}, \\
            0 & \text{ otherwise,} \end{cases}  \allowdisplaybreaks   \\
& (\tau_{k+1}\tau_{k}\tau_{k+1}-\tau_{k}\tau_{k+1}\tau_{k})e(\mathbf{i}) \\
& \hspace{.5truecm} = \begin{cases}
\dfrac{\mathcal{Q}_{i_k,i_{k+1}}(x_{{k+2}},x_{{k+1}})-\mathcal{Q}_{i_k,i_{k+1}}(x_{{k}},x_{{k+1}})}{x_{k+2}-x_{k}}\e(\mathbf{i}) & \text{ if } i_k=i_{k+2}\neq i_{k+1}, \\
0 & \text{ otherwise.}
\end{cases}
\end{align*}
The algebra
$$R:=\bigoplus_{\alpha \in \mathsf{Q}^{+}}R(\alpha)$$
is called the {\it Khovanov-Lauda-Rouquier
algebra} ({\it KLR algebra}) associated with $\mathfrak{A}$ and $(\mathcal{Q}_{ij})_{i,j \in I}$.
\item Let $\lambda \in \mathsf{P}^{+}$. The {\it homogeneous cyclotomic Khovanov-Lauda-Rouquier algebra} $R^\lambda(\alpha)$ at $\alpha$ of weight $\lambda$ is
the quotient algebra of $R(\alpha)$ by the two-sided ideal $I^{\lambda}(\alpha)$ of $R(\alpha)$
generated by $x_m^{\langle h_{i_m}, \lambda \rangle}e(\mathbf{i})$ ($\mathbf{i} \in I^\alpha$). The algebra
$$R^{\lambda} := \bigoplus_{\alpha \in Q^+} R^\lambda(\alpha)$$
is the \emph{cyclotomic Khovanov-Lauda-Rouquier algebra} (\emph{cyclotomic KLR algebra}) of
weight $\lambda$.
\end{enumerate}
\end{definition}
The $\Z$-grading on $R(\alpha)$ is given by
$$ \mathsf{deg}( \e(\mathbf{i}))=0, \quad   \mathsf{deg}(x_\ell \e(\mathbf{i}))= (\alpha_{i_\ell}|\alpha_{i_\ell}),
\quad  \mathsf{deg}(\tau_k \e( \mathbf{i}))= -(\alpha_{i_k}|\alpha_{i_{k+1}}).$$
Here $R(0)=\C$.
Let $R(\alpha)$-$\mathsf{fmod}$ (resp.\ $R^{\lambda}(\alpha)$-$\mathsf{fmod}$) be the category of finite-dimensional $\Z$-graded
$R(\alpha)$-modules (resp.\ $R^{\lambda}(\alpha)$-modules).
Any $N \in R^{\lambda}(\alpha)$-$\mathsf{fmod}$ can be viewed as a
graded $R(\alpha)$-module annihilated by $I^{\lambda}(\alpha)$. We write $\infl^{\lambda}N$
when considering $N \in R^{\lambda}(\alpha)$-$\mathsf{fmod}$ as an $R(\alpha)$-module.
Any $ M \in R(\alpha)\fmod$ gives rise to the $R^\lambda(\alpha)$-module $\pr^{\lambda}M := M/I^{\lambda}(\alpha)M$.
From now on, when there is no possibility of confusion, we identify irreducible graded $R^\lambda(\alpha)$-modules with irreducible
graded $R(\alpha)$-modules annihilated by $I^\lambda(\alpha)$ via $\infl^\lambda$.

Let
$$ G_0(R)=\bigoplus_{\alpha \in \mathsf{Q}^+} G_0(R(\alpha)\text{-}\mathsf{fmod}), \quad
G_0(R^{\lambda})=\bigoplus_{\alpha \in \mathsf{Q}^+} G_0(R^{\lambda}(\alpha)\text{-}\mathsf{fmod}),$$
where $G_0(R(\alpha)\text{-}\mathsf{fmod})$ (resp. $G_0(R^{\lambda}(\alpha)\text{-}\mathsf{fmod})$) is the Grothendieck group
of $R(\alpha)\text{-}\mathsf{fmod}$ (resp. of $R^{\lambda}(\alpha)\text{-}\mathsf{fmod}$). For $M \in R(\alpha)\text{-}\mathsf{fmod}$ (resp.\ $R(\alpha)^\lambda\text{-}\mathsf{fmod}$), $[M]$ stands for
the isomorphism class of $M$ in $G_0(R(\alpha))$ (resp.\ $G_0(R^\lambda(\alpha))$). When no
confusion can arise, we write $M$ for $[M]$.

Given $M = \bigoplus_{i \in \Z} M_i$, let $M \langle k \rangle = \bigoplus_{i \in \Z} M \langle k \rangle _i$ denote the graded
module obtained from $M$ by shifting the grading by $k$, where $M \langle k \rangle_i := M_{i+k} $ for $i \in \Z $.
The {\it $q$-character} $\qch(M)$ and {\it character}
$\ch(M)$ of $M$ are defined by
$$ \qch(M) := \sum_{\mathbf{i} \in I^\alpha} \qdim( \e(\mathbf{i}) M ) \ \mathbf{i}, \qquad
 \ch(M) := \sum_{\mathbf{i} \in I^\alpha} \dm( \e(\mathbf{i}) M)  \ \mathbf{i}, $$
where $\qdim(N) := \sum_{i \in \Z} (\dm N_i) q^i $ for any graded module $N = \bigoplus_{i\in \Z} N_i$.
For $\mathbf{i}\in I^\alpha$, we write
$\mathbf{i} \in \ch (M) \  \text{(resp.\ $\mathbf{i} \in  \qch(M)$)}$
when $\mathbf{i}$ occurs in $\ch M$ (resp.\ $\qch(M)$) with a nonzero coefficient.
\smallskip

\textbf{Notation and Conventions.\ }
{\it In this paper, by a {\rm homomorphism} we mean a homogeneous homomorphism of some degree $k \in \Z$.
The symbol $\simeq$ (resp.\ $\twoheadrightarrow$, $\hookrightarrow$) will be used to denote an isomorphism
(resp.\ a surjective homomorphism, an injective homomorphism) up to a grading shift, and
the notation $\cong$ will be reserved for a degree preserving isomorphism.
For $M \in R(\alpha)\text{-}\mathsf{fmod}$ and $N \in R(\beta)\text{-}\mathsf{fmod}$, $M \boxtimes N$ will denote the outer tensor product
of $M$ and $N$.}

For $M,N \in R(\alpha)\text{-}\mathsf{fmod}$, let $\Hom(M,N)$ be the $\C$-vector space of
homogeneous homomorphisms of degree $0$, and let $ \HOM(M,N)=\bigoplus_{k \in \Z} \Hom(M,N\langle k \rangle).$
For $\beta_1,\ldots,\beta_m \in \mathsf{Q}^{+}$, we define
\begin{equation}\label{eq:ebeta} e({\beta_1, \ldots, \beta_m}) = \sum_{\mathbf{i}_j\in I^{\beta_j}} \e({\mathbf{i}_1 * \cdots * \mathbf{i}_m}).
\end{equation}
The natural embedding
$ R(\beta_1) \otimes \cdots \otimes R(\beta_m) \hookrightarrow R(\beta_1+\cdots+\beta_m)$ gives the following functors:
\begin{align*}
\ind_{\beta_1,\ldots,\beta_m}\ {\underline{\hspace{.25truecm}}} \ &:= \  R(\beta_1+\cdots+\beta_m) \otimes_{R(\beta_1) \otimes \cdots \otimes R(\beta_m)} \  {\underline{\hspace{.25truecm}}}\,, \\
\coind_{\beta_1,\ldots,\beta_m}\ {\underline{\hspace{.25truecm}}} \ &:= \  \HOM_{R(\beta_1) \otimes \cdots \otimes R(\beta_m)}(R(\beta_1+\cdots+\beta_m), \ {\underline{\hspace{.25truecm}}}\,), \\
\res_{\beta_1,\ldots,\beta_m}\ {\underline{\hspace{.25truecm}}} \ &:= \  e({\beta_1, \ldots, \beta_m})\ {\underline{\hspace{.25truecm}}}\,.
\end{align*}

\begin{lemma}[\cite{LV09}] \label{Lem: Frobenius reciprocity}
\begin{enumerate}
\item Let $M_k \in R(\beta_k)\fmod$ for $k=1, \ldots, m$. Then
$$ \ind_{\beta_1, \ldots, \beta_m} (M_1\boxtimes \cdots \boxtimes M_m) \simeq \coind_{\beta_m, \ldots, \beta_1} (M_m \boxtimes \cdots \boxtimes M_1).$$
\item For $M \in R(\beta_1) \otimes \cdots \otimes R(\beta_m)$-$\fMod$ and
$N \in R(\beta_1+\cdots+\beta_k)$-$\fMod$,
\begin{align*}
\HOM_{R(\beta_1 + \cdots + \beta_m)} ( \ind_{\beta_1, \ldots \beta_m}M, N) &\cong \HOM_{R(\beta_1)\otimes \cdots \otimes R(\beta_m)} ( M, \res_{\beta_1,\ldots,\beta_m}N),\\
\HOM_{R(\beta_1 + \cdots + \beta_m)} (N, \coind_{\beta_1, \ldots, \beta_m} M ) &\cong \HOM_{R(\beta_1) \otimes \cdots \otimes R(\beta_m)}(\res_{\beta_1, \ldots ,\beta_m}N, M).
\end{align*}
\end{enumerate}
\end{lemma}

For $m \in \Z_{\ge 0}$ and $i\in I$, let
\begin{equation*}
L(i^m) = \ind_{\C[x_1, \ldots, x_m]}^{R(m\alpha_i)} \ \phi,
\end{equation*}
where $\phi$ is the 1-dimensional trivial module over $\C[x_1, \ldots, x_m]$ with $x_j \phi = 0$ for  all $j=1,\ldots,m$, and
$ \qdim(\phi) = 1$.
Let $\B(\infty)$ (resp.\ $\B(\lambda)$) be the set of all isomorphism classes of irreducible graded $R$-modules
(resp.\ $R^{\lambda}$-modules). Let $\mathbf{1}$ denote the $1$-dimensional trivial $R(0)$-module.
For a graded $R(\beta)$-module (resp.\ $R^\lambda(\beta)$-module) $M$, it was shown in \cite{KK11,KL09,LV09} that
the operators defined by
\begin{align*}
e_{i}(M) &= \res^{\alpha_i, \beta-\alpha_i}_{\beta-\alpha_i} \big( \e({\alpha_i, \beta-\alpha_i}) M\big), \\
f_{i}(M) &= \left\{
              \begin{array}{ll}
                R(\beta+\alpha_i)\e(\alpha_i, \beta) \otimes_{R(\beta)}M & \hbox{ if } \ \ M \in R(\beta)\text{-}\mathsf{mod}, \\
                R^\lambda(\beta+\alpha_i)\e(\alpha_i, \beta) \otimes_{R^\lambda(\beta)}M  & \hbox{ if } \ \  M \in R^\lambda(\beta)\text{-}\mathsf{mod}.
              \end{array}
            \right.
\end{align*} satisfy
\begin{align} \label{Eq: Serre}
\sum_{k=0}^{1-{a}_{ij}} (-1)^k e_i^{(1-{a}_{ij}-k)}e_je_i^{(k)}[M]=0, \quad \sum_{k=0}^{1-{a}_{ij}} (-1)^k f_i^{(1-{a}_{ij}-k)}f_jf_i^{(k)}[M]=0,
\end{align}
where $e_i^{(r)}[M]:=\dfrac{1}{[r]_{q_i}!}[e_i^{r}M]$ and $f_i^{(r)}[M]:=\dfrac{1}{[r]_{q_i}!}[f_i^{r}M]$ for $i,j \in I$, $r \in \Z_{\ge 0}$.

We define
\begin{align*}
  \wt(M) &= \begin{cases} -\beta &\ \  \text{ if } \ \ M \in R(\beta)\text{-}\mathsf{fmod}, \\
                       \lambda-\beta &\ \  \text{ if }\ \  M \in R^{\lambda}(\beta)\text{-}\mathsf{fmod}, \end{cases} \ \
 \\
 \varepsilon_{i} (M) &= \mathsf{max}\{ k \ge 0 \mid e_i^k M \ne 0 \}, \quad \
 \varphi_{i}(M) = \varepsilon_i(M)+ \wt(M)(h_i), \\
  \ke_{i}(M) &= \soc(e_i M), \ \ \hbox{\rm and} \ \
   \kf_{i}(M) =
  \begin{cases}
  \hd\,\ind_{\alpha_i, \beta} (L(i) \boxtimes M) & \ \  \text{ if } M \in R(\beta)\text{-}\mathsf{fmod},   \\
  \pr^{\lambda} \circ \kf_i \circ \infl^{\lambda} M & \ \  \text{ if }  M \in R^{\lambda}(\beta)\text{-}\mathsf{fmod}.
  \end{cases}
\end{align*}
The Kashiwara operators $\ke_i$ and $\kf_i$ are defined  in the
opposite manner in \cite{KP10,KL09}.  We have made this change so
they can be viewed as operators acting on the left of irreducible
modules. In Section \ref{Sec: construction for classical type}, we
will explain further why we want the Kashiwara operators defined in
this way.
\begin{theorem}[{\cite[Thm. 7.4, Thm. 7.5]{LV09}}] \label{Thm: iso-LV}
The sextuple $(\B(\infty),\wt, \ke_i, \kf_i, \varepsilon_i,
\varphi_i)$ (resp.\ $(\B(\lambda),\wt, \ke_i, \kf_i, \varepsilon_i,
\varphi_i))$ is a crystal, which is isomorphic to the crystal
$B(\infty)$ (resp.\ $B(\lambda)$).
\end{theorem}

 \subsection{Construction of irreducible $R$-modules}\

In this subsection, we give a construction of irreducible graded $R$-modules as the irreducible heads of certain induced modules,
which is compatible with the Kashiwara operators. More precisely, set $I_{(n+1)} = I$ and let $I_{(k)} \subset I\ (k=1, \ldots,n)$ be subsets of $I$ such that
$I_{(k)} \subseteq I_{(k+1)}$ and $|I_{(k)}| = k$.  Let
$\mathfrak{B}_k$ be the crystal obtained from $B(\infty)$ by forgetting the $i$-arrows for $i \notin  I_{(k)}$.
Using highest weight vectors of connected components of $\mathfrak{B}_k$,
we define irreducible $R$-modules $\mathcal{N}_k(v)$ for $k=1,\ldots,n$ and $v \in B(\infty)$.
Then, we construct a crystal isomorphism $\Phi: B(\infty) \rightarrow \B(\infty)$
defined by $\Phi(v) = \hd\,\ind (\mathcal{N}_1(v) \boxtimes \cdots \boxtimes \mathcal{N}_n(v))$ (Proposition \ref{Prop: crystal iso}).
We emphasize that this crystal isomorphism $\Phi$ exists for an arbitrary symmetrizable Cartan datum.

For $M \in R(\beta)\fmod$, $M_k \in R(\beta_k)\fmod$ ($k =1,\ldots, m$) and $d \in \Z_{>0}$, set
\begin{align*}
M^{\boxtimes 0} = \C, \qquad M^{\boxtimes d} = \underbrace{M\boxtimes \cdots \boxtimes M}_d \qquad
\dbt{k=1}{m} M_k = M_1 \boxtimes \cdots \boxtimes M_m.
\end{align*}

\begin{lemma} \label{Lem: irr of induced module}
 Let $M_k \in R(\beta_k)\fmod$, \ $d_k \in \Z_{\ge 0}$, and $\mathbf{i}_k \in I^{d_k \beta_k}$ for $k = 1, \ldots, m$. Suppose that
$\mathbf{i}_k$ occurs in $\ch\,\ind (M_k^{\boxtimes d_k})$ with multiplicity $\xi_k \in \Z_{>0}$ for $k = 1, \ldots, m$. Assume
\begin{enumerate}
\item[(i)] $\ind \left(M_k \boxtimes M_{k'}\right) \simeq \ind \left(M_{k'} \boxtimes M_k\right) $ for $k,k'=1, \ldots,m$,
\item[(ii)] $\ind \left(M_k^{\boxtimes d_k}\right)$ is irreducible for $k=1, \ldots,m$, and
\item[(iii)] $\mathbf{i} := \mathbf{i}_1 * \cdots *\mathbf{i}_m$ occurs in $\ch\,\ind \left(\dbt{k=1}{m} M_k^{\boxtimes d_k}
\right)$ with multiplicity $\xi_1 \xi_2 \cdots \xi_m$.
\end{enumerate}
Then $\ind\left(\dbt{k=1}{m} M_k^{\boxtimes d_k}\right)$ is irreducible.
\end{lemma}
\begin{proof}
Let $N_k = \ind \left(M_k^{\boxtimes d_k}\right)$ for $k=1, \ldots, m$,  and let
$Q$ be a nonzero quotient of the module $\ind\left(\dbt{k=1}{m}
N_k\right)$. By (ii)  and Lemma \ref{Lem: Frobenius reciprocity}
(2), we have an injective homomorphism $$ \dbt{k=1}{m} N_k    \hookrightarrow \res_{d_1 \beta_1, \ldots, d_m \beta_m} Q.
$$
Hence, by (iii), \ $\mathbf{i}$ occurs in $\ch Q$ with multiplicity $\xi_1 \xi_2\cdots \xi_m$. Since $Q$ is arbitrary, we obtain that
$ \hd\,\ind \left(\dbt{k=1}{m} N_k\right) $ is irreducible.

Now it follows from (i)  that
$$ \ind\left(\dbt{k=1}{m} N_k\right) \simeq \ind\left(N_m \boxtimes N_{m-1} \boxtimes \cdots \boxtimes N_1\right). $$
Let $L$ be a nonzero submodule of $\ind \left(\dbt{k=1}{m} N_k\right)$.
By $\rm (ii)$ and Lemma \ref{Lem: Frobenius reciprocity}, there exists a surjective homomorphism
$$ \res L \twoheadrightarrow \dbt{k=1}{m} N_k, $$
which implies that $\mathbf{i}$ occurs in $\ch L$ with multiplicity $\xi_1 \xi_2\cdots \xi_m$ by $\rm (iii)$. Therefore,
we conclude $\ind \left(\dbt{k=1}{m} N_k\right) = \ind\left(\dbt{k=1}{m} M_k^{\boxtimes d_k}\right)$
is irreducible.
\end{proof}

\begin{lemma} \label{Lem: sum of fi}
 Let $\emptyset = I_{(0)} \subsetneq I_{(1)} \subsetneq I_{(2)} \subsetneq \cdots \subsetneq I_{(m)} \subset I $, and assume $\beta_k \in \sum_{i \in I_{(k)}} \Z_{\ge0}\alpha_i$ for $k=1,\ldots, m$.
Let $M_k \in R(\beta_{k})\fmod\ \ (1 \le k \le m)$ be such that
\begin{enumerate}
\item[(i)] $\varepsilon_i(M_k)=0$ for $i \in I_{(k-1)}$,
\item[(ii)] $\hd M_k$ is irreducible, and
\item[(iii)] $\hd M_k$ occurs with multiplicity one as a composition factor of $M_k$.
\end{enumerate}
Then
\begin{enumerate}
\item $\hd\,\ind\left(\dbt{k=1}{m}M_k\right)$ is irreducible,
\item $\hd\,\ind\left(\dbt{k=1}{m}M_k\right)$ occurs with multiplicity one as a composition factor of $\ind\left(\dbt{k=1}{m}M_k\right)$,
\item if $\kf_{\mathbf{i}_k} \mathbf{1} \simeq \hd M_k $ for some $\mathbf{i}_k \in I^{\beta_{k}}$, then
$\kf_{\mathbf{i}_1}\kf_{\mathbf{i}_2} \cdots \kf_{\mathbf{i}_{m}} \mathbf{1} \simeq \hd\,\ind\left(\dbt{k=1}{m}M_k\right)$.
\end{enumerate}
\end{lemma}
\begin{proof}
 We may assume that $M_k \ne \mathbf{1}$ for $k=1, \ldots, m$. Let $d_k = |\beta_k|$ for $k=1,\ldots,m$,  and set $d = d_1 + \cdots + d_m $.
For $w\in S_d$ and a reduced expression $w = \sigma_{i_1}\cdots \sigma_{i_s} $, let
$$ \tau_w = \tau_{i_1}\cdots \tau_{i_s}. $$
For $k=1,\ldots,m$, take $\mathbf{i}_k = (i_{k,1}, \ldots,  i_{k,d_{k}}) \in I^{\beta_{k}}$ such that $ \e({\mathbf{i}_k}) M_k \ne 0$.
It follows from (i)  that $i_{k,1} \in I_{(k)} \setminus I_{(k-1)}$.
Since $M_k \in R(\beta_{k})\fmod$, we have
$$ \e({\beta_1, \ldots, \beta_{m}}) \tau_w \e({\mathbf{i}_1* \cdots *\mathbf{i}_{m}}) \ne 0 \ \ \text{ if and only if } w = \id $$
for $w \in S_{d} / S_{d_1} \times \cdots \times S_{d_m}$. Therefore,
$$ \res_{\beta_1, \ldots, \beta_m} \Big(\ind\Big( \dbt{k=1}{m}M_k\Big) \Big) \simeq \dbt{k=1}{m}M_k. $$

Let $L =  \dbt{k=1}{m} \hd M_k$.  It follows from
(ii) and (iii) that $[L]$ occurs with multiplicity one in $[
\res_{\beta_1, \ldots, \beta_m} \Big(\ind\Big( \dbt{k=1}{m}M_k\Big)
\Big) ]$ in the Grothendieck group $G_0(R(\beta_1) \otimes \cdots
\otimes R(\beta_m))$. Let $Q$ be a nonzero quotient of $\ind\Big(
\dbt{k=1}{m}M_k\Big)$. By Lemma \ref{Lem: Frobenius reciprocity}, we
have a nontrivial homomorphism
$$\dbt{k=1}{m}M_k  \rightarrow \res_{\beta_1, \ldots, \beta_m} Q,  $$
which implies that $[L]$ occurs in $[\res_{\beta_1, \ldots, \beta_m} Q]$ with multiplicity one. Since $Q$ is arbitrary,
we obtain assertions (1) and (2).

Now assume that $N = \hd\,\ind \Big(\dbt{k=1}{m}M_k\Big )$, and let $N_k = \hd M_k$ for $k=1,\ldots, m$. Then there is a surjective homomorphism
$$ \ind\Big(\dbt{k=1}{m}M_k\Big) \twoheadrightarrow \ind \Big(\dbt{k=1}{m}N_k\Big) ,$$
which implies that $N = \hd\,\ind\Big(\dbt{k=1}{m}N_k\Big)$. We will use induction on $|\wt(N)|$. If $|\wt(N)|=0$, there is nothing to prove, so
suppose that $|\wt(N)|>0$. Since $N_1$ is nontrivial,
we can take $i \in I_{(1)}$ such that $\varepsilon_i(N_1) \ne 0$. Let $\varepsilon = \varepsilon_i(N_1).$ Then, by $\rm (i)$ and Lemma \ref{Lem: Frobenius reciprocity}, we have
$$ \varepsilon = \varepsilon_i\Big(\ind \Big(\dbt{k=1}{m}N_k\Big)\Big) = \varepsilon_i(N) .$$
Since $e_i$ is exact,  there is a surjective homomorphism
$$ e_i^\varepsilon \left( \ind \Big(\dbt{k=1}{m}N_k\Big) \right) \twoheadrightarrow e_i^{\varepsilon}N.$$
Then, by \cite[Lem.~3.9]{KP10}, we know that $[e_i^{\varepsilon} N]  = q_{i}^{-\varepsilon + 1}\, [\varepsilon]_{q_i} !\,[ \ke_i^{\varepsilon}N]$ and
\begin{align*}
\Big[e_i^\varepsilon\Big ( \ind\Big(\dbt{k=1}{m}N_k\Big)\Big)\Big] & =  [\ind \left((e_i^\varepsilon N_1)\boxtimes N_2 \boxtimes \cdots \boxtimes N_{m})\right] \\
& = q_i^{-\varepsilon + 1} [\varepsilon]_{q_i} ! \, [\ind\left((\ke_i^\varepsilon N_1)\boxtimes N_2 \boxtimes \cdots \boxtimes N_{m}\right)]
\end{align*}
at the level of the Grothendieck group $G_0(R( \beta_1 + \cdots + \beta_m - \varepsilon \alpha_i))$. Since $\hd\,\ind ( (\ke_i^\varepsilon N_1)\boxtimes N_2 \boxtimes \cdots \boxtimes N_{m})$ is
irreducible by $(1)$,  we obtain
$$ \hd\,\ind (\ke_i^{\varepsilon}N_1 \boxtimes \cdots \boxtimes N_{m}) \simeq \ke_i^{\varepsilon} N. $$
Therefore, assertion (3) follows from a standard induction argument.
\end{proof}

Let $n = |I|$ be the rank of $U_q(\g)$, and let $I_{(k)} \subset I\ (k=1,\ldots,n)$ be subsets of $I = I_{(n+1)}$ such that $ I_{(k)} \subset I_{(k+1)}$ and $|I_{(k)}|=k$ for all $k$.
We denote by $U_q(\g_k)$ the subalgebra of $U_q(\g)$ generated by $e_i,\ f_i\ (i\in I_{(k)})$ and $q^h\ (h \in \mathsf{P}^{\vee})$.
Let $\mathfrak{B}_k$ be the crystal obtained
from $B(\infty)$ by forgetting the $i$-arrows for $i \notin  I_{(k)}$. Then $\mathfrak{B}_k$ can be viewed as a $U_q(\g_k)$-crystal.

\begin{lemma}
For each $k=1,\ldots, n$, every connected component of $\mathfrak{B}_k$ has a unique highest weight vector.
\end{lemma}
\begin{proof}
Let $\mathcal{C}$ be a connected component of $\mathfrak{B}_k$. For $\lambda \in \mathsf{P}^+$, let $\iota_\lambda : B(\lambda) \rightarrow B(\infty)\otimes T_\lambda \otimes C$ be the embedding in
$\eqref{Eq: crystal embedding}$ and let
$$\widetilde{ \mathcal{C}} = \{ v \otimes t_\lambda \otimes c \mid  v \in \mathcal{C} \} \subset B(\infty) \otimes T_\lambda \otimes C.$$
By taking  $\lambda \gg 0$,  we can assume  that $ \widetilde{ \mathcal{C}} \cap \im (\iota_\lambda) \ne \emptyset. $
Since $\mathcal{C}_\lambda: = \iota_\lambda^{-1}(\widetilde{ \mathcal{C}})$ is a nontrivial highest weight subcrystal of the crystal obtained from $B(\lambda)$ by forgetting the $i$-arrows with $i \notin I_{(k)}$,
there is an element $u_{\mathcal{C}} \in \mathcal{C}$ such that $\iota_\lambda^{-1}(u_{\mathcal{C}} \otimes t_\lambda \otimes c)$ is the highest weight vector of the subcrystal $\mathcal{C}_\lambda$. Note that
$u_\mathcal{C}$ does not depend on the choice of $\lambda$ if $ \widetilde{ \mathcal{C}} \cap \im (\iota_\lambda) \ne \emptyset$.
By construction, the element $u_\mathcal{C}$ is a unique highest weight vector of $\mathcal{C}$.
\end{proof}

Take $v \in B(\infty)$. Let $u_0 = v$ and let $u_k$ be the highest weight vector of the connected component $\mathcal{C}_k$ of $\mathfrak{B}_k$ containing $v$
for $k=1,\ldots, n$.
By construction, there is a chain of injective maps
$$ \mathcal{C}_1 \hookrightarrow \mathcal{C}_2 \hookrightarrow \cdots \hookrightarrow \mathcal{C}_{n-1} \hookrightarrow B(\infty).$$
For $k=1,\ldots, n$, let $\mathbf{i}_k$ be a sequence of $I$ such that $ u_{k-1} = \kf_{\mathbf{i}_k} u_{k}$.
Note that $v = \kf_{\mathbf{i}_1} \kf_{\mathbf{i}_1} \cdots \kf_{\mathbf{i}_{n}} \mathsf{1} $.
Let
\begin{align} \label{Eq: def of Nk}
\mathcal{N}_k(v) = \kf_{\mathbf{i}_k} \mathbf{1} \in \B(\infty) .
\end{align}
Hence, for each $v\in B(\infty)$, we have the corresponding $n$-tuple $(\mathcal{N}_1(v), \mathcal{N}_2(v), \ldots, \mathcal{N}_{n}(v))$ of
modules in $\B(\infty)$.

\begin{proposition} \label{Prop: crystal iso} \
\begin{enumerate}
\item For $v \in B(\infty)$, let $\mathcal{N}_k(v)$ be the irreducible module defined by $\eqref{Eq: def of Nk}$. Then
$ \hd\,\ind \left( \dbt{k=1}{n} \mathcal{N}_k(v)\right) $ is irreducible.
\item The map $\Phi: B(\infty) \longrightarrow \B(\infty)$  defined by
$$ \Phi(v) = \hd\,\ind \left( \dbt{k=1}{n} \mathcal{N}_k(v)\right) \ \  \text{ for } v \in B(\infty) $$
is a crystal isomorphism.
\end{enumerate}
\end{proposition}

\begin{proof}
Since it is obvious that $\Phi(\mathsf{1}) = \mathbf{1}$, we assume
$\mathsf{1} \ne v \in B(\infty)$ and let $ \mathcal{N}_k(v) = \kf_{\mathbf{i}_k} \mathbf{1} $ be the corresponding
irreducible module given in $\eqref{Eq: def of Nk}$
for $k=1, \cdots, n$.
By construction, we have
$$ v = \kf_{\mathbf{i}_1} \kf_{\mathbf{i}_1} \cdots \kf_{\mathbf{i}_{n}} \mathsf{1}\ \  \text{ and }\ \   \varepsilon_i( \mathcal{N}_k(v)) = 0
\ \ \text{ for all }  i\in I_{(k-1)}. $$
Therefore, it follows from Lemma \ref{Lem: sum of fi} that $ \hd\,\ind \left( \dbt{k=1}{n} \mathcal{N}_k(v) \right) $ is irreducible and
$$ \Phi(  \kf_{\mathbf{i}_1} \kf_{\mathbf{i}_2} \cdots \kf_{\mathbf{i}_{n}} \mathsf{1} )
= \hd\,\ind\left( \dbt{k=1}{n} \mathcal{N}_k(v)\right) = \kf_{\mathbf{i}_1} \kf_{\mathbf{i}_2}
\cdots \kf_{\mathbf{i}_{n}} \Phi(\mathsf{1}).$$    \end{proof}

\section{Description of $\mathcal{N}_k(v)$ for Finite Type Via Adapted Strings} \label{Sec: Description Nk}

In this section, we give an explicit description of the modules $\mathcal{N}_k(v)$ for finite type. To that end, we choose a particular reduced expression of the longest element $w_0$
of the Weyl group $W$.   Using the notion of adapted strings  for crystals given in \cite{Lit98} with respect to these specially chosen expressions,
we describe $\mathcal{N}_k(v)$ explicitly  in terms of Kashiwara operators $\kf_i$ (Proposition \ref{Prop: description of Nk}).

Throughout the section, we assume that the Cartan datum $(\mathfrak{A},\mathsf{P}, \Pi,\mathsf{P}^{\vee},\Pi^{\vee})$ is of finite type.   For a sequence $\mathbf{m} = (m_1, \ldots, m_k)$ of elements in $I$, we write
$ r_\mathbf{m} = r_{m_1} r_{m_2} \cdots r_{m_k}$ for the corresponding element in the Weyl group.
Using this convention, we  fix a particular reduced expression for
the longest element $w_0$ of $W$
\begin{equation}\label{eq:reduced} w_0  = r_{\mathbf{s}_1}r_{\mathbf{s}_2} \cdots r_{\mathbf{s}_n} = r_{s_1} r_{s_2} \cdots r_{s_\ell} \end{equation}
where the sequences $\mathbf{s}_k$ are displayed in Table $\ref{Tb: w0}$ below.
Let $l_k$ be  the length of the sequence $\mathbf{s}_k$ in
\eqref{eq:reduced}  for $k=1,\ldots,n$.   Then  $r_{\mathbf{s}_1} = r_{s_1} \cdots r_{s_{l_1}}$,
$r_{\mathbf{s}_2} = r_{s_{l_1+1}} \cdots r_{s_{l_1 + l_2}}$ and so forth.
  Note that $\ell$ is the
length of the reduced expression.

 \begin{definition} Let $v \in B(\infty)$ (resp.\ $v \in B(\lambda)$). An $\ell$-tuple $\mathbf{a}(v)=(a_1, \ldots, a_\ell) \in (\Z_{\ge0})^\ell$ is called the {\it adapted string} of $v$
with respect to the expression $r_{s_1} \cdots r_{s_\ell}$ if
$$a_1 = \ep_{s_1}(v),\ \   a_2 = \ep_{s_2}\left( \ke_{s_1}^{a_1} v\right), \ \ \ldots\,,\ \ a_\ell = \ep_{s_\ell}\left( \ke_{s_{\ell-1}}^{a_{\ell-1}} \cdots \ke_{s_{1}}^{a_{1}} v\right) .$$
We denote by $\mathcal{S} = \{ \mathbf{a}(v) \mid v \in B(\infty) \}$ (resp.\ $\mathcal{S}^\lambda = \{ \mathbf{a}(v) \mid v \in B(\lambda) \}$)
the set of all adapted strings of elements in $B(\infty)$ (resp.\ $B(\lambda)$).
\end{definition}

{\begin{table} [ht]
\caption{ }
\centering
\begin{tabular}{ c c l }
\hline
 & Dynkin diagrams   & $w_0 = r_{\mathbf{s}_1} \cdots r_{\mathbf{s}_n} $
\\ \hline \hline
 & &  \\
$(\mathsf{A}_n)$ &    \raisebox{0.5em}{\DynkinA}   & $ \mathbf{s}_k = (n+1-k, \ldots, n-1,n)$ \quad for $k=1,\ldots,n$    \\ [3ex]
\raisebox{-0.8em}{$(\mathsf{B}_n)$}  &    {\DynkinB}   & $ \mathbf{s}_k = (n+1-k, \ldots, n-1,n,n-1,\ldots,n+1-k)$    \\ [-0.5em]
        &               &   \qquad \ \  for $k=1,\ldots,n$    \\  [3ex]
{$(\mathsf{C}_n)$}  &    \raisebox{0.3em}{\DynkinC}   & $\mathbf{s}_k$ is the same as for type $\mathsf{B}_n$ for $k=1,\ldots,n$   \\ [3ex]
\raisebox{-1.4em}{$(\mathsf{D}_n)$}  &    \raisebox{-1em}{\DynkinD}   & $ \mathbf{s}_1 = (n)$, $ \mathbf{s}_2 = (n-1)$,    \\ [-2em]
          &                                & $ \mathbf{s}_k = (n+1-k,\ldots,n-2,n,n-1,n-2,\ldots,n+1-k)$   \\
          &                                &   \qquad \ \  for $k=3,\ldots,n$   \\ [3ex]
\raisebox{-1.0em}{$(\mathsf{E}_6)$} &    \raisebox{-1em} {\DynkinEs}   &   $\mathbf{s}_k$ is the same as for type $\mathsf{D}_5$  for $k=1,\ldots,5$,   \\ [-1.5em]
          &                                &  $ \mathbf{s}_6 = (6,5,3,4,2,1,3,2,5,3,4,6,5,3,2,1)$,   \\  [4ex]
\raisebox{-1.2em}{$(\mathsf{E}_7)$} &    \raisebox{-1.2em}{\DynkinEsv}   & $\mathbf{s}_k$ is the same as for type $\mathsf{E}_6$  for $k=1,\ldots,6$,    \\[-1.5em]
          &                                &   $ \mathbf{s}_7 = (7,6,5,3,4,2,1,3,2,5,3,4,6,5,3,2,1,7,6,5,3,4,2,$ \\
          &                                &  \qquad \ \ $3,5,6,7)$    \\ [3ex]
\raisebox{-2em}{$(\mathsf{E}_8)$} &    \raisebox{-2em}{\DynkinEe}   & $\mathbf{s}_k$ is the same as for type $\mathsf{E}_7$  for $k=1,\ldots,7$,   \\ [-2.5em]
          &                                &  $  \mathbf{s}_8 = (8,7,6,5,3,4,2,1,3,2,5,3,4,6,5,3,2,1,7,6,5,3,4,$   \\
          &                                & \qquad \ \ $2,3,5,6,7,8,7,6,5,3,4,2,1,3,2,5,3,4,6,5,3,2,1,$    \\
          &                                &  \qquad \ \ $7,6,5,3,4,2,3,5,6,7,8)$   \\[4ex]
\raisebox{-1em}{$(\mathsf{F}_4)$}  &   \raisebox{-0.5em}{\DynkinF}   &  $\mathbf{s}_k$ is the same as for type $\mathsf{B}_3$  for $k=1,\ldots,3$,  \\  [-1.0em]
          &                                &    $\mathbf{s}_4 = (4,3,2,1,3,2,3,4,3,2,1,3,2,3,4)$  \\ [4ex]
$(\mathsf{G}_2)$  &    \raisebox{0.5em}{\DynkinG}   &   $\mathbf{s}_1 = (1), \ \mathbf{s}_2 = (2,1,2,1,2)$  \\
 & &  \\ \hline
\end{tabular}
\label{Tb: w0}
\end{table}
 }

Observe that if $\iota_{\lambda}(v) = v' \otimes t_\lambda \otimes c$ for some $v\in B(\lambda) $ and $ v' \in B(\infty)$, then $\mathbf{a}(v) = \mathbf{a}(v')$.
It follows from \cite{Kash93a,Lit94} that $\mathcal{S}$ and $\mathcal{S}^\lambda$ are
in one-to-one correspondence with the crystals $B(\infty)$ and $B(\lambda)$, respectively,
by the maps
\begin{equation} \label{Eq: one to one}
 \begin{aligned}
\mathcal{S} &\buildrel \text{1-1} \over\longrightarrow B(\infty) \ \ \text{ given by } \ \  \mathbf{a}=(a_1, \ldots, a_\ell) \mapsto \kf^{a_1}_{s_1} \cdots \kf^{a_\ell}_{s_\ell}  \mathsf{1}, \\
 \mathcal{S}^\lambda &\buildrel \text{1-1} \over\longrightarrow B(\lambda) \ \ \ \text{ given by } \ \   \mathbf{a}=(a_1, \ldots, a_\ell) \mapsto \kf^{a_1}_{s_1} \cdots \kf^{a_\ell}_{s_\ell} b_\lambda.
 \end{aligned}
\end{equation}

For $\mathbf{a} = (a_1, \ldots, a_\ell) \in (\Z_{\ge 0})^\ell$,  set
\begin{align} \label{Eq: subseq of a}
\ad_{k,j} = a_{l_1 + \cdots + l_{k-1}+j} \quad \hbox{\rm and} \quad
\mathbf{a}_k = (\ad_{k,1}, \ad_{k,2}, \ldots,  \ad_{k,l_k}  ),
\end{align}
where $l_k$ is,  as above, the length of $\mathbf{s}_k$. Then  $\mathbf{a} = \mathbf{a_1} * \cdots * \mathbf{a_n}$.

In particular, for $v \in B(\infty)$,    we have $\mathbf{a}(v) = \mathbf{a}(v)_1 * \cdots * \mathbf{a}(v)_n$ and
$$v = \kf_{\mathbf{s}_1}^{\mathbf{a}(v)_1} \kf_{\mathbf{s}_2}^{\mathbf{a}(v)_2} \cdots \kf_{\mathbf{s}_n}^{\mathbf{a}(v)_n} \mathsf{1}.$$

\begin{proposition}[{\cite{Lit98}}]  \label{Prop: adapted strings}
 Let $(\mathfrak{A}, \mathsf{P}, \Pi, \mathsf{P}^{\vee}, \Pi^{\vee})$ be a Cartan datum of finite type.
 Then $\mathcal{S}$ has the following description:
\begin{align*}(\mathsf{A}_n): \ \   \{ \mathbf{a} \in (\Z_{\ge0})^{\ell} \mid & \  \ad_{i,1} \ge \ad_{i,2} \ge \cdots \ge \ad_{i,i}\  (1 \le i \le n) \}\\
(\mathsf{B}_n): \ \  \{ \mathbf{a} \in (\Z_{\ge0})^{\ell} \mid & \  2\ad_{i,1} \ge \cdots \ge 2\ad_{i,i-1} \ge \ad_{i,i} \ge 2\ad_{i,i+1} \ge \cdots \ge 2\ad_{i,2i-1}\ (1 \le i \le n) \}  \\
(\mathsf{C}_n): \ \ \{ \mathbf{a} \in (\Z_{\ge0})^{\ell} \mid & \  \ad_{i,1} \ge \cdots \ge \ad_{i,i-1} \ge \ad_{i,i} \ge \ad_{i,i+1} \ge \cdots \ge \ad_{i,2i-1}\  (1 \le i \le n) \}  \\
(\mathsf{D}_n): \ \  \{ \mathbf{a} \in (\Z_{\ge0})^{\ell} \mid & \  \ad_{i,1} \ge \cdots \ge \ad_{i,i-1}, \ad_{i,i} \ge \ad_{i,i+1} \ge \cdots \ge \ad_{i,2i-2}\  (3 \le i \le n) \}  \\
(\mathsf{E}_6): \ \  \{ \mathbf{a} \in (\Z_{\ge0})^{\ell} \mid &\ \ad_{6,1} \ge \ad_{6,2} \ge \ad_{6,3} \ge \ad_{6,4}, \ad_{6,5} \ge \ad_{6,7} \ge \ad_{6,8}, \ad_{6,9} \ge \ad_{6,10}
\ge \ad_{6,11},\ad_{6,13} \\
 &\ \ge \ad_{6,14} \ge \ad_{6,15} \ge \ad_{6,16};\ \ad_{6,5} \ge \ad_{6,6} \ge \ad_{6,8};\ \ad_{6,9} \ge \ad_{6,12} \ge \ad_{6,13}; \\
&\ \text{$\ad_{i,j}$'s satisfy the inequalities for $\mathsf{D}_5$ for $i=1,\ldots,5$} \}\\
 (\mathsf{E}_7): \ \  \{ \mathbf{a} \in (\Z_{\ge0})^{\ell} \mid &\ \ad_{7,1} \ge \ad_{7,2} \ge \ad_{7,3} \ge \ad_{7,4} \ge \ad_{7,5}, \ad_{7,6}  \ge \ad_{7,8} \ge \ad_{7,9}, \ad_{7,10}  \\
 &\ \ge \ad_{7,11} \ge \ad_{7,12},\ad_{7,14} \ge \ad_{7,15} \ge \ad_{7,16}, \ad_{7,20} \ge \ad_{7,21}\ge \ad_{7,22}, \ad_{7,23}  \\
 &\ \ge \ad_{7,24} \ge \ad_{7,25} \ge \ad_{7,26} \ge \ad_{7,27};\ \ad_{7,6} \ge \ad_{7,7}\ge \ad_{7,8} ;\ \\
 &\ \ad_{7,10} \ge \ad_{7,13}\ge \ad_{7,14}, \ad_{7,18} \ge \ad_{7,19} \ge \ad_{7,20};\ \ad_{7,16} \ge \ad_{7,17} \ge \ad_{7,23}; \\
&\ \text{$\ad_{i,j}$'s satisfy the inequalities for $\mathsf{E}_6$ for $i=1,\ldots,6$} \} \\
(\mathsf{E}_8): \ \  \{ \mathbf{a} \in (\Z_{\ge0})^{\ell} \mid &\ (\ad_{8,2}, \ldots, \ad_{8,28}) \ \text{and} \ (\ad_{8,30}, \ldots, \ad_{8,56}) \text{ satisfy the inequalities of $\mathbf{a}_7$ for $\mathsf{E}_7$};\\
&\hspace{-.4truein} \ad_{8,1} \ge \ad_{8,2};\ \ad_{8,29} \ge \ad_{8,30};\ \ad_{8,56} \ge \ad_{8,57};  \\
&\hspace{-.4truein} \ad_{8,19} \ge \mathsf{max}\{ \ad_{8,30}, \ad_{8,29} - \ad_{8,28} \};\  \mathsf{min}\{ \ad_{8,23}, \ad_{8,25} +\ad_{8,33} - \ad_{8,34} \} \ge \ad_{8,35}; \\
&\hspace{-.4truein}  \ad_{8,20} \ge \mathsf{max}\{ \ad_{8,31}, \ad_{8,28} + \ad_{8,30} - \ad_{8,27} \};\  \mathsf{min}\{ \ad_{8,25}, \ad_{8,26} +\ad_{8,32} - \ad_{8,33} \} \ge \ad_{8,37}; \\
&\hspace{-.4truein} \ad_{8,21} \ge \mathsf{max}\{ \ad_{8,32}, \ad_{8,27} + \ad_{8,31} - \ad_{8,26} \};\  \mathsf{min}\{ \ad_{8,26}, \ad_{8,27} +\ad_{8,31} - \ad_{8,32} \} \ge \ad_{8,39}; \\
&\hspace{-.4truein} \ad_{8,22} \ge  \mathsf{max}\{ \ad_{8,33}, \ad_{8,26} + \ad_{8,32} - \ad_{8,25} \};\  \mathsf{min}\{ \ad_{8,27}, \ad_{8,28} +\ad_{8,30} - \ad_{8,31} \} \ge \ad_{8,42}; \\
&\hspace{-.4truein}  \ad_{8,24} \ge  \mathsf{max}\{ \ad_{8,34}, \ad_{8,25} + \ad_{8,33} - \ad_{8,23} \};\  \mathsf{min}\{ \ad_{8,28}, \ad_{8,29} - \ad_{8,30} \} \ge \ad_{8,47};   \\
&\hspace{-.4truein} \text{$\ad_{i,j}$'s satisfy the inequalities for $\mathsf{E}_7$ for $i=1,\ldots,7$} \}  \allowdisplaybreaks  \\
(\mathsf{F}_4): \ \  \{ \mathbf{a} \in (\Z_{\ge0})^{\ell} \mid &\ \ad_{4,1} \ge \ad_{4,2} \ge \ad_{4,3} \ge \ad_{4,4}, \ad_{4,5} \ge \ad_{4,6} \ge \ad_{4,7}; \\
 &\ \ad_{4,9} \ge \ad_{4,10} \ge \ad_{4,11},\ad_{4,12}  \ge \ad_{4,13} \ge \ad_{4,14} \ge \ad_{4,15}; \\
&\ \ad_{4,5} \ge \ad_{4,9};\ \ad_{4,7} \ge \ad_{4,12};\ \ad_{4,5} + \ad_{4,7} \ge \ad_{4,8} \ge \ad_{4,9} + \ad_{4,12};\\
&\  2\ad_{4,6} \ge \ad_{4,7} + \ad_{4,9} \ge 2\ad_{4,10}; \\
&\ \text{$\ad_{i,j}$'s satisfy the inequalities for $\mathsf{B}_3$ for $i=1, 2, 3$} \}   \\
(\mathsf{G}_2): \ \  \{ \mathbf{a} \in (\Z_{\ge0})^{\ell} \mid &\ 6\ad_{2,1} \ge 2\ad_{2,2} \ge 3\ad_{2,3} \ge 2\ad_{2,4} \ge 6\ad_{2,5} \}.
\end{align*}  \end{proposition}
We now give a description of $\mathcal{N}_k(v)$ using Proposition \ref{Prop: adapted strings}.

\begin{proposition} \label{Prop: description of Nk}
Let $(\mathfrak{A}, \mathsf{P}, \Pi, \mathsf{P}^{\vee}, \Pi^{\vee})$ be a Cartan datum of finite type. For $v \in B(\infty)$ and $k=1,\ldots, n$, the module
$\mathcal{N}_k(v)$ defined by $\eqref{Eq: def of Nk}$ is given by
$$ \mathcal{N}_k(v) = \kf_{\mathbf{s}_k} ^{\mathbf{a}(v)_k} \mathbf{1},  $$
where $\mathbf{a}(v)$ is the adapted string of $v$ with respect to the expression $w_0 = r_{\mathbf{s}_1} \cdots r_{\mathbf{s}_n}$ in Table $\ref{Tb: w0}$.
\end{proposition}
\begin{proof}
Let $I_{(k)} = \{ n+1-k, \ldots, n-1,n \}$ for $k=1,\ldots, n$ and set $I_{(0)} = \emptyset$. Let $\mathfrak{B}_k$ denote
the crystal obtained from $B(\infty)$ by forgetting the $i$-arrows with $i \notin I_{(k)} $.
Take $v \in B(\infty)$ and let $\mathbf{a} = \mathbf{a}(v) = \mathbf{a}_1 * \cdots * \mathbf{a}_n$ be the adapted string of $v$ with respect to the expression $w_0 = r_{\mathbf{s}_1} \cdots r_{\mathbf{s}_n}$  in Table $\ref{Tb: w0}$.
For $k=1,\ldots, n$, let
$ \mathbf{0}_k = (\underbrace{0,\ldots, 0}_{l_k}) $
where $l_k$ is the length of $\mathbf{s}_k$ as before.
Then, by Proposition \ref{Prop: adapted strings},
$$ \mathbf{b}_k :=  \mathbf{0}_1 *  \ldots *  \mathbf{0}_{k} * \mathbf{a}_{k+1} * \ldots *  \mathbf{a}_{n-1} * \mathbf{a}_n $$
is contained in $\mathcal{S}$.
Let $u_k = \kf_{\mathbf{s}_{k+1}}^{\mathbf{a}_{k+1}} \cdots \kf_{\mathbf{s}_{n}}^{\mathbf{a}_{n}} \mathsf{1}$ for $k=0,\ldots, n-1$ and let $u_n = \mathsf{1}$.
Since $\mathbf{b}_k $ is the adapted string of $u_k$, by the definition of adapted strings we have
$$ v = \kf_{\mathbf{s}_{1}}^{\mathbf{a}_{1}} \cdots \kf_{\mathbf{s}_{k}}^{\mathbf{a}_{k}} u_k , \qquad  \varepsilon_i(u_k)=0 \ \text{ for }i \in I_{(k)}, $$
which implies that $u_k$ is the highest weight vector of the connected component $\mathcal{C}_k$ of $\mathfrak{B}_k$ containing $v$.
Therefore,  $\mathcal{N}_k(v) = \kf_{\mathbf{s}_k} ^{\mathbf{a}_k} \mathbf{1}$ for $k=1,\ldots,n$.
\end{proof} \medskip

\section{Explicit Construction of Irreducible Modules for Finite Classical Type} \label{Sec: construction for classical type}

In Sections \ref{Sec: construction for classical type} and \ref{Sec:
final}, we  assume that $\mathfrak{A}$ is of finite classical type.
We maintain the notation from Sections \ref{Sec: Crystals and KLR
algs} and \ref{Sec: Description Nk}.  Our aim in this section is to
present an explicit construction of irreducible $R$-modules (resp.\
$R^\lambda$-modules) for finite classical type using a single
induction step.   Let $\cB$ be the crystal of the irreducible
highest weight $U_q(\g)$-module $V(\Lambda_n)$ if $\mathfrak{A}$ is
of type $\mathsf{A}_n$ and of $V(\Lambda_1)$ if $\mathfrak{A}$ is of
type $\mathsf{B}_n, \mathsf{C}_n, \mathsf{D}_n$. We define the $1$
or $2$-dimensional irreducible $R$-module $\Delta_{(a,b)}$ for $a,b
\in \cB$ with $a \succ b$ by using the structure of $\cB$. Combining
Proposition \ref{Prop: crystal iso} with some facts about the
modules $\Delta_{(a,b)}$ and the descriptions of the adapted
strings, we define the outer tensor product $\Delta(\mathbf{a}(v)):=
\Delta(\mathbf{a}(v);1) \boxtimes \cdots \boxtimes
\Delta(\mathbf{a}(v);n)$ of $\Delta_{(a,b)}$'s for $v \in B(\infty)$
(resp.\ $v \in B(\lambda)$). For $v \in B(\lambda)$, the number
$\eta(v)$ of $\Delta_{(a,b)}$'s in $\Delta(\mathbf{a}(v))$ has an
upper bound; i.e., $ \eta(v) \le n \lambda(h), $ where $h$  is as in
Lemma \ref{Lem: upper bound}  below. Then, we construct a map $
\Psi: B(\infty) \rightarrow \B(\infty)$ (resp.\ $\Psi^{\lambda}:
B(\lambda) \rightarrow  \B(\lambda)$) by taking the head of $\ind
\Delta(\mathbf{a}(v))$ (Theorem \ref{Thm: main theorem}).
The proof that this map is indeed a
crystal isomorphism, hence is compatible with the Kashiwara
operators,  will be provided in the next section.

We first give a detailed description of $\mathcal{S}$ and
$\mathcal{S}^\lambda$ for finite classical type with respect to the
expression of $w_0 =  r_{\mathbf{s}_1} \cdots r_{\mathbf{s}_n}$ in
Table \ref{Tb: w0}. Let $\triangle_{\mathsf{A}_n}$ be the triangle
consisting of right justified rows  of boxes with $1$ box in the
first (bottom) row, $2$ boxes in the second row, $\ldots$, and $n$
boxes in the top row.  Let $\triangle_{\mathsf{B}_n}$ (resp.\
$\triangle_{\mathsf{D}_n}$) be the triangle consisting of centered
rows of boxes having $1$ box (resp.\ $2$ boxes) in the first row,
$3$ boxes (resp.\ $4$ boxes) in the second row, \ldots , and
$(2n-1)$ boxes (resp.\ $(2n-2)$ boxes) in the top row.    Set
$\triangle_{\mathsf{C}_n} = \triangle_{\mathsf{B}_n}$. When it is
not necessary to specify the type, we omit the subscript and simply
write $\triangle$.  For $\mathbf{a} \in \mathcal{S}$, let
$\triangle(\mathbf{a})$ denote the filling of the triangle
$\triangle$ with entries of $\mathbf{a}$ from left to right in each
row, and from bottom to top. Let $t_{ij}$ be the $j$th entry of the
$i$th row in $\triangle(\mathbf{a})$ and let $\triangle(\mathbf{a})
= \{t_{ij}\}_{ 1 \le i \le n' ,\ p_i \le j \le p_i'}$,  where
\begin{eqnarray*} n' &=& \begin{cases} n & \qquad \quad  (\mathsf{A}_n, \mathsf{B}_n, \mathsf{C}_n) \\
 n-1 & \quad  \qquad (\mathsf{D}_n) \end{cases} \\
 p_i &=& \begin{cases}  n+1-i & \qquad  (\mathsf{A}_n, \mathsf{B}_n, \mathsf{C}_n),\\
  n-i & \qquad (\mathsf{D}_n) \end{cases} \\
  p_i' &=& \begin{cases} n & \qquad (\mathsf{A}_n), \\
   n-1+i & \qquad  (\mathsf{B}_n, \mathsf{C}_n, \mathsf{D}_n). \end{cases} \end{eqnarray*}
Set $t_{ij}=0$ except when $1 \le i \le n' ,\ p_i \le j \le p_i'$.
For example, if $\mathbf{a}=(2,3,1,0,9,8,4,2,1) \in \mathcal{S}$ for type $\mathsf{B}_3$ and $\mathbf{a}'=(5,2,7,4,3,1,9,6,4,5,3,2) \in \mathcal{S}$ for type $\mathsf{D}_4$, then
$t_{1,3}=2$ in $\triangle(\mathbf{a})$, $t_{3,4}=5$ in $\triangle(\mathbf{a}')$ and
$$\triangle(\mathbf{a}) =  \exTriangleB, \quad \quad \ \triangle(\mathbf{a}') =  \exTriangleD. $$

Let $\tilde{\mathsf{c}}(t_{i,n-1}) = \sum_{k=i}^{n-1} t_{k,n-1}$, \ $ \tilde{\mathsf{c}}(t_{i,n}) = \sum_{k=i}^{n-1} t_{k,n}$,  and
\begin{align}\label{eq:cdef}
\mathsf{c}(t_{i,j}) = \left\{
                      \begin{array}{ll}
                        \sum_{k=i}^n t_{k,j} & \hbox{if } j \le n\ (\mathsf{A}_n),\ \  j = n\ (\mathsf{B}_n), \\
                        \sum_{k=i}^n(t_{k,j} + t_{k,2n-j}) & \hbox{if } j < n\ (\mathsf{B}_n, \mathsf{C}_n), \\
                        \sum_{k=i}^{n-1} (t_{k,j} + t_{k,2n-1-j}) & \hbox{if } j < n-1\ (\mathsf{D}_n), \\
                        \sum_{k=i}^n 2t_{k,j} & \hbox{if } j = n\ (\mathsf{C}_n), \\
                        \sum_{k=i}^{n-1} (t_{k,n-1}+t_{k,n}) & \hbox{if } j=n-1,n\  (\mathsf{D}_n), \\
                        t_{i,j} + \sum_{k=i+1}^{n}(t_{k,2n-j}+t_{k,j}) & \hbox{if } j > n\ (\mathsf{B}_n, \mathsf{C}_n),\\
                        t_{i,j} + \sum_{k=i+1}^{n-1}(t_{k,2n-1-j}+t_{k,j}) & \hbox{if } j > n\  (\mathsf{D}_n).
                      \end{array}
                    \right.
\end{align}

\begin{proposition}[{\cite{Lit98}}]  \label{Prop: S lambda}
Let $(\mathfrak{A}, \mathsf{P}, \Pi, \mathsf{P}^{\vee}, \Pi^{\vee})$ be a Cartan datum of finite classical type,
and let $\lambda = \lambda_1 \Lambda_1 + \cdots + \lambda_n \Lambda_n \in \mathsf{P}^+$.
For $\mathbf{a} \in (\Z_{\ge0})^{\ell}$,  let $\triangle(\mathbf{a}) = \{t_{ij}\}_{ 1 \le i \le n' ,\ p_i \le j \le p_i'}$ be as above
and assume $\mathsf{c}(t_{i,j})$ is as in \eqref{eq:cdef}.   Then
\begin{align*}
(\mathsf{A}_n):\  \mathcal{S} & = \{ \mathbf{a} \in (\Z_{\ge0})^{\ell} \mid  t_{i,n+1-i}  \ge t_{i,n+2-i}  \ge \cdots \ge t_{i,n}\ \text{ for }1 \le i \le n \}, \\
\mathcal{S}^\lambda & = \{ \mathbf{a} \in \mathcal{S} \mid  t_{i,j} \le \lambda_j + \mathsf{c}(t_{i+1,j-1})   -2\mathsf{c}(t_{i+1,j})+  \mathsf{c}(t_{i,j+1})\\
& \qquad  \qquad \qquad \qquad  \qquad \qquad \qquad  \qquad \qquad   \text{ for }
1 \le i\le n ,\ n+1-i \le j\le n \}  \\
(\mathsf{B}_n): \ \mathcal{S} & =  \{ \mathbf{a} \in (\Z_{\ge0})^{\ell} \mid
2t_{i,n+1-i} \ge \ldots \ge 2t_{i,n-1} \ge t_{i,n} \ge 2t_{i,n+1} \ge  \ldots  \ge 2t_{i,n-1+i} \\
& \qquad  \qquad \qquad \qquad \qquad \qquad \qquad \qquad \qquad \qquad \qquad \qquad
\text{ for } 1 \le i\le n \}, \\
\mathcal{S}^\lambda & = \{ \mathbf{a} \in \mathcal{S} \mid \  t_{i,j} \le \lambda_j + \mathsf{c}(t_{i,j+1}) -2\mathsf{c}(t_{i,2n-j})+\mathsf{c}(t_{i,2n+1-j}), \\
& \qquad  \qquad
t_{i,2n-j} \le \lambda_j + \mathsf{c}(t_{i+1,j+1}) -2\mathsf{c}(t_{i+1,j})+ \mathsf{c}(t_{i,2n+1-j}), \\
& \qquad  \qquad
t_{i,n} \le \lambda_n + 2 \mathsf{c}(t_{i,n+1})-2 \mathsf{c}(t_{i+1,n})\ \text{ for }1\le i \le n,\ n+1-i \le j < n  \}  \\
(\mathsf{C}_n):\ \mathcal{S} & = \{ \mathbf{a} \in (\Z_{\ge0})^{\ell} \mid t_{i,n+1-i} \ge \ldots  \ge t_{i,n} \ge  \ldots \ge t_{i,n-1+i} \ \text{ for } 1 \le i \le n \}\\
\mathcal{S}^\lambda & = \{ \mathbf{a} \in \mathcal{S} \mid \  t_{i,j} \le \lambda_j + \mathsf{c}(t_{i,j+1})-2\mathsf{c}(t_{i,2n-j})+\mathsf{c}(t_{i,2n+1-j}), \\
& \qquad \qquad
t_{i,2n-j} \le \lambda_j + \mathsf{c}(t_{i+1,j+1}) -2\mathsf{c}(t_{i+1,j})+ \mathsf{c}(t_{i,2n+1-j}), \\
& \qquad \qquad
t_{i,n} \le \lambda_n + \mathsf{c}(t_{i,n+1})-\mathsf{c}(t_{i+1,n})\ \text{ for } 1 \le i \le n,\ n+1-i \le j < n \}  \allowdisplaybreaks \\
(\mathsf{D}_n):\ \mathcal{S} &= \{ \mathbf{a} \in (\Z_{\ge0})^{\ell} \mid t_{i,n-i} \ge \ldots  \ge t_{i,n-2} \ge
  t_{i,n-1}, t_{i,n} \ge t_{i,n+1} \ge  \ldots  \ge t_{i,n-1+i} \\
& \qquad  \qquad  \qquad \qquad \qquad \qquad \qquad \qquad \qquad \qquad \qquad \qquad
\text{ for } 1 \le i \le n-1 \}, \\
\mathcal{S}^\lambda & = \{ \mathbf{a} \in \mathcal{S} \mid \  t_{i,j} \le \lambda_j + \mathsf{c}(t_{i,j+1})-2\mathsf{c}(t_{i,2n-1-j})+\mathsf{c}(t_{i,2n-j}), \\
& \quad \quad
t_{i,2n-1-j} \le \lambda_j + \mathsf{c}(t_{i+1,j+1}) -2\mathsf{c}(t_{i+1,j})+ \mathsf{c}(t_{i,2n-j}), \\
& \quad \quad
t_{i,n-1} \le \lambda_{n-1} + \mathsf{c}(t_{i,n+1})-2\tilde{\mathsf{c}}(t_{i+1,n-1}),\\
& \quad \quad
t_{i,n} \le \lambda_n + \mathsf{c}(t_{i,n+1})-2\tilde{\mathsf{c}}(t_{i+1,n})\ \text{ for }1 \le i \le n-1,\ n-i \le j< n-1 \}.
\end{align*}
\end{proposition}

\indent Let $\cB$ be the crystal given by

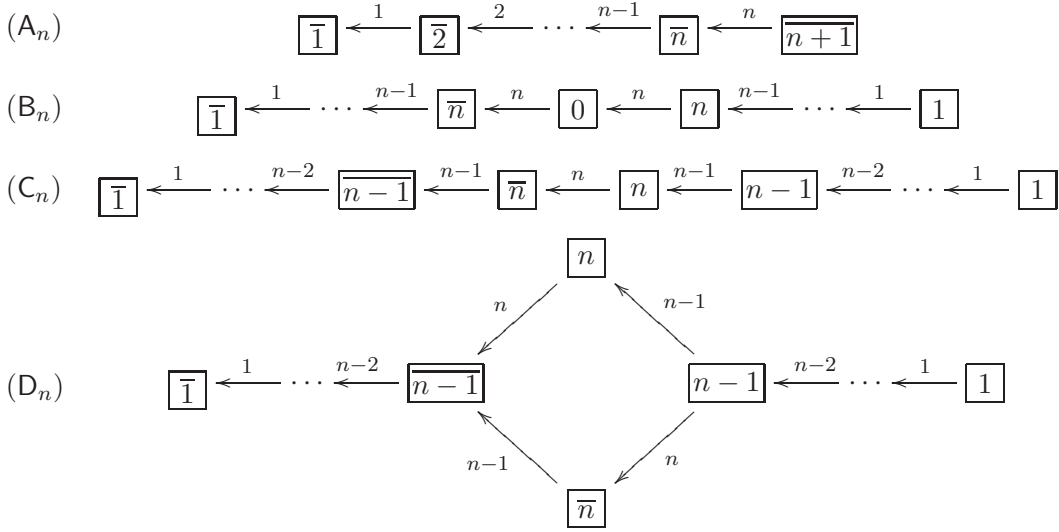
\begin{tabular}{ l c }
 $(\mathsf{A}_n)$ &  $\CrystalA$     \\ [2ex]
 $(\mathsf{B}_n)$  &  $\CrystalB$    \\ [2ex]
 $(\mathsf{C}_n)$  &  $\CrystalC$
 \end{tabular}

\begin{tabular}{ l c }
 \raisebox{-4.5em}{$(\mathsf{D}_n)$}  & \qquad \   $\CrystalD$
 \end{tabular}

\noindent   with the entries ordered by

\begin{align*}
\ \ \quad(\mathsf{A}_n) \qquad \qquad &   \overline{1} \succ \overline{2} \succ \cdots \succ \overline{n+1} , \\\quad(\mathsf{B}_n)\qquad \qquad & \overline{1} \succ \overline{2} \succ \cdots \succ \overline{n} \succ 0 \succ n \succ \cdots \succ 2 \succ 1,  \\
\ \ \quad(\mathsf{C}_n) \qquad \qquad &  \overline{1} \succ \overline{2} \succ \cdots \succ \overline{n-1} \succ \overline{n} \succ {n} \succ {n-1} \succ \cdots \succ {1} ,
\qquad\qquad\qquad\qquad\qquad\qquad\qquad\qquad  \\
\ \ \quad(\mathsf{D}_n) \qquad \qquad & \overline{1} \succ \overline{2} \succ \cdots \succ \overline{n-1} \succ    \overline{n}, n
\succ {n-1} \succ \cdots \succ {1} .
\end{align*}
Let
\begin{align} \label{Eq: def of overhat i}
\widehat{\imath} =  \left\{
     \begin{array}{ll}
       \overline{\imath} & \hbox{ if } 1 \le i \le n+1\ \  (\mathsf{A}_n),  \ \ 1 \le i \le n\ \ (\mathsf{B}_n,\mathsf{C}_n,\mathsf{D}_n), \\
       0 & \hbox{ if } i=n+1\ \  (\mathsf{B}_n), \\
       2n+1-i & \hbox{ if } n+1 \le i \le 2n\ \  (\mathsf{C}_n, \mathsf{D}_n), \\
       2n+2-i & \hbox{ if } n+2 \le i \le 2n+1\ \ (\mathsf{B}_n).
     \end{array}
   \right.
\end{align}

For $a,b \in \cB$ with $a \succ b$, let $\mathbf{i}(a,b)$ be a sequence of elements in $I$ such that $a = \kf_{\mathbf{i}(a,b)}b$, and
define the irreducible graded $R$-module
\begin{align} \label{Eq: def of nabla}
\Delta_{(a,b)} = \kf_{\mathbf{i}(a,b)} \mathbf{1}.
\end{align}
Note that $\Delta_{(a,b)}$ in general is not a cuspidal
representation given in \cite{HMM09,KR09}.

\textbf{Remark.\ } {\it
The set of cuspidal representations is in one-to-one correspondence with  the set of positive roots,
via correspondence given by the weights of the cuspidal representations.  But there are some $\Delta(a,b)$'s which do not correspond to positive roots.   For example, in type $\mathsf{B}_n$,
the module $\Delta_{(\overline{1},1)} \in R(\sum_{i=1}^n 2\alpha_i)\text{-} \fMod$ is not a cuspidal representation,
since $\sum_{i=1}^n 2\alpha_i$ is not a positive root
(see Example \ref{Ex: example}). } 

The action of the KLR
algebra can be described  explicitly as follows.

If one of the following holds:  $a \succ b$ \ $(\mathsf{A}_n, \mathsf{C}_n)$, either $b \succeq 0$ or $0 \succeq a$ \ $(\mathsf{B}_n)$, either $b \succ n-1$ or $\overline{n-1} \succ a$ \ $(\mathsf{D}_n)$, then
the module $\Delta_{(a,b)}$ is the 1-dimensional $R$-module $\C v$ specified by
$$ x_i v = 0 , \quad \tau_j v = 0, \quad \e(\mathbf{i}) v = \left\{
                                                            \begin{array}{ll}
                                                              v & \hbox{ if } \mathbf{i} = \mathbf{i}(a,b),  \\
                                                              0 & \hbox{ otherwise.}
                                                            \end{array}
                                                          \right.
 $$

If $a \succ 0 \succ b$ for type $\mathsf{B}_n$, then $\Delta_{(a,b)}$ is the 2-dimensional $R$-module $\C u \oplus \C v$ with $R$-action
\begin{align*}
x_i u &= 0, \qquad
\tau_ju = \left\{
                                \begin{array}{ll}
                                  v & \hbox{ if } j = d, \\
                                  0 & \hbox{ otherwise,}
                                \end{array}
                              \right.
 \qquad \e(\mathbf{i})u = \left\{
                           \begin{array}{ll}
                             u & \hbox{ if } \mathbf{i} = \mathbf{i}(a,b), \\
                             0 & \hbox{ otherwise,}
                           \end{array}
                         \right. \\
x_iv &= \left\{
           \begin{array}{ll}
            - u & \hbox{ if } i = d, \\
             u & \hbox{ if } i = d + 1, \\
             0 & \hbox{ otherwise}
           \end{array}
         \right.
\qquad \tau_jv= 0,
 \qquad \e(\mathbf{i})v = \left\{
                           \begin{array}{ll}
                             v & \hbox{ if } \mathbf{i} = \mathbf{i}(a,b), \\
                             0 & \hbox{ otherwise,}
                           \end{array}
                         \right.
\end{align*}
where $d$ is an integer such that $\sigma_{d}(\mathbf{i}(a,b)) = \mathbf{i}(a,b)$.

If $a \succeq \overline{n-1} $ and $n-1 \succeq b$ for type $\mathsf{D}_n$, then the module $\Delta_{(a,b)}$ is the 2-dimensional $R$-module $\C u \oplus \C v$ given by
\begin{align*}
x_i u &= 0, \quad \tau_j u = \left\{
                              \begin{array}{ll}
                                \mathcal{Q}_{n,n-1}(x_{n},x_{n-1}) v & \hbox{ if } j = d, \\
                                0 & \hbox{ otherwise, }
                              \end{array}
                            \right.
\quad \e(\mathbf{i})u = \left\{
                         \begin{array}{ll}
                           u & \hbox{ if } \mathbf{i} = \mathbf{i}(a, {n}) * \mathbf{i}({n}, b), \\
                           0 & \hbox{ otherwise,}
                         \end{array}
                       \right. \\
x_i v &= 0, \quad \tau_j v = \left\{
                              \begin{array}{ll}
                                u & \hbox{ if } j = d, \\
                                0 & \hbox{ otherwise, }
                              \end{array}
                            \right.
\ \ \qquad\quad\qquad\qquad \quad \e(\mathbf{i})v = \left\{
                         \begin{array}{ll}
                           v & \hbox{ if } \mathbf{i} = \mathbf{i}(a, \overline{n}) * \mathbf{i}(\overline{n}, b), \\
                           0 & \hbox{ otherwise,}
                         \end{array}
                       \right.
\end{align*}
where $d$ is an integer such that $\sigma_{d}(\mathbf{i}(a, n) * \mathbf{i}(n, b)) = \mathbf{i}(a, \overline{n}) * \mathbf{i}(\overline{n}, b)$.
Note that $ \mathcal{Q}_{n,n-1}(x_{n},x_{n-1}) = \zeta_{n,n-1} \in \C \setminus\{0 \}. $ 

It follows from the description above that, for $a,b \in \cB$ with $a \succ b$,
\begin{align} \label{Eq: cha of nabla}
\ch \Delta_{(a,b)} = \left\{
                       \begin{array}{ll}
                         \mathbf{i}(a, n) * \mathbf{i}(n, b) + \mathbf{i}(a, \overline{n}) * \mathbf{i}(\overline{n}, b) & \hbox{ if } a \succeq \overline{n-1},\ {n-1} \succeq b\ \ (\mathsf{D}_n), \\
                           2 \mathbf{i}(a,b) & \hbox{ if } a\succ 0 \succ b\ \  (\mathsf{B}_n),\\
                         \mathbf{i}(a,b) & \hbox{ otherwise. }  \\
                       \end{array}
                     \right.
\end{align}

\textbf{Remark.\ } {\it
We see from the expression for the character in  \eqref{Eq: cha of nabla} that  $\Delta_{(a,b)}$ can be identified with the sequence $\mathbf{i}(a,b)$,  or equivalently, with the segment of $\cB$ between $a$ and $b$, which is another reason why we  defined the Kashiwara operators in the opposite manner to \cite{KP10,KL09}. }

Given $\mathbf{a} \in \mathcal{S}$ (resp.\ $\mathbf{a} \in \mathcal{S}^\lambda$) and $i=1,\ldots,n$, let
\begin{align} \label{Eq: nabla}
\Delta(\mathbf{a};i) = \left\{
                            \begin{array}{ll}
                              \dbt{j=n+1-i}{n_i} \left(\Delta_{(\widehat{n+1-\imath},\ \widehat{\jmath+1})}^{\boxtimes \vartheta_{ij}}\right) &\ \ \hbox{if } \ \ 1 \le i \le n\ \  (\mathsf{A}_n,\mathsf{B}_n,\mathsf{C}_n), \\
                              \Delta_{(\overline{n-1},\ {n})}^{\boxtimes t_{1,n-1}} &\ \  \hbox{if }\ \  i=1\ \  (\mathsf{D}_n),\\
                              \Delta_{(\overline{n-1},\ \overline{n})}^{\boxtimes t_{1,n}} &\ \  \hbox{if }\ \  i= 2 \ \ (\mathsf{D}_n), \\
                              \dbt{j=n+1-i}{n_i} \left(\Delta_{(\widehat{n+1-\imath},\ \widehat{\jmath+1})}^{\boxtimes \vartheta_{i-1,j}}\right) &\ \  \hbox{if } \ \ 3 \le i \le n \ (\mathsf{D}_n),
                            \end{array}
                          \right.
\end{align}
where $n_i= n \ \hbox{\rm for all}\ i\ \ (\mathsf{A}_n)$, $n_i = n+i\ \ (\mathsf{B}_n)$, $n_i = n-1+i\ \ (\mathsf{C}_n, \mathsf{D}_n)$, $\triangle(\mathbf{a}) = \{ t_{i,j} \}$,  and

\begin{align} \label{Eq: theta}
\vartheta_{ij} = \left\{
                   \begin{array}{ll}
                     t_{i,j} - t_{i,j+1} & \hbox{if }\ \  j \le n_i \ \ (\mathsf{A}_n,\mathsf{C}_n),\  j \le n-2\ \  (\mathsf{B}_n),\  j \le n-3\ \ (\mathsf{D}_n), \\
                     t_{i,n-1} - \lceil \frac{t_{i,n}}{2} \rceil & \hbox{if } \ \ j=n-1\ \ (\mathsf{B}_n), \\
                     \lceil \frac{t_{i,n}}{2} \rceil - \lfloor \frac{t_{i,n}}{2} \rfloor & \hbox{if }\ \  j=n\ \ (\mathsf{B}_n), \\
                     \lfloor \frac{t_{i,n}}{2} \rfloor - t_{i,n+1} & \hbox{if } \ \ j=n+1\ \  (\mathsf{B}_n), \\
                      t_{i, n-2} -  \mathsf{max}\{ t_{i, n-1}, t_{i,n} \} & \hbox{if } \ \ j=n-2\ \ (\mathsf{D}_n), \\
                      \mathsf{max}\{ 0, t_{i,n}-t_{i,n-1} \} & \hbox{if }\ \  j=n-1\ \ (\mathsf{D}_n), \\
                     \mathsf{max}\{ 0, t_{i,n-1}-t_{i,n} \} & \hbox{if }\ \  j=n\ \ (\mathsf{D}_n), \\
                     \mathsf{min}\{ t_{i, n-1}, t_{i,n} \} - t_{i, n+1} & \hbox{if } \ \ j = n+1\  \ (\mathsf{D}_n), \\
                      t_{i,j-1} - t_{i,j} & \hbox{if } \ \ j \ge n+2\ \ (\mathsf{B}_n, \mathsf{D}_n).
                   \end{array}
                 \right.
\end{align}

For $v \in B(\infty)$ (resp.\ $v \in B(\lambda)$), let
$\mathbf{a}(v)$ be the adapted string of $v$ with respect to the
expression $w_0 = r_{\mathbf{s}_1} \cdots r_{\mathbf{s}_n}$ given in
Table $\ref{Tb: w0}$. Now we are ready to state the main theorem in
this section.

\begin{theorem} \label{Thm: main theorem} Let $(\mathfrak{A}, \mathsf{P}, \Pi, \mathsf{P}^{\vee}, \Pi^{\vee})$ be a Cartan datum of finite classical type.  For $v\in B(\infty)$ (resp.\ $v\in B(\lambda)$), let $\mathbf{a}(v)$ be the adapted string of $v$
with respect to the expression $w_0 = r_{\mathbf{s}_1} \cdots r_{\mathbf{s}_n}$ given in Table $\ref{Tb: w0}$
and let $\Delta(\mathbf{a}(v)) = \Delta(\mathbf{a}(v);1) \boxtimes \cdots \boxtimes \Delta(\mathbf{a}(v);n)$. Then
\begin{enumerate}
\item  $\hd\,\ind  \Delta(\mathbf{a}(v) )$ is irreducible.
\item The map $\Psi: B(\infty) \longrightarrow \B(\infty)$ defined by
$$ \Psi(v) = \hd\,\ind \Delta(\mathbf{a}(v))  \ \ \ \text{ for } v \in B(\infty) $$  is a crystal isomorphism.
\item The map $\Psi^\lambda: B(\lambda) \longrightarrow \B(\lambda)$  defined by
$$ \Psi^\lambda(v) = \hd\,\ind  \Delta(\mathbf{a}(v)) \ \ \ \text{ for } v \in B(\lambda)$$  is a crystal isomorphism.
\end{enumerate}
\end{theorem}

Theorem \ref{Thm: main theorem} and Proposition \ref{Prop: S lambda} combine to give the following explicit description of the irreducible graded modules over $R$ and $R^\lambda$.

\begin{corollary} \label{Cor: complete lists}
 Let $(\mathfrak{A}, \mathsf{P}, \Pi, \mathsf{P}^{\vee}, \Pi^{\vee})$ be a Cartan datum of finite classical type. Then
\begin{enumerate}
\item the set  $$ \mathcal{A} = \{ \hd\ind \Delta(\mathbf{a}) \mid \mathbf{a} \in \mathcal{S}  \}$$ is the complete list of all irreducible graded $R$-modules up to isomorphism and grading shift.
\item For $\lambda \in \mathsf{P}^+$,  the set
$$ \mathcal{A}^\lambda = \{ \hd\,\ind \Delta(\mathbf{a}) \mid \mathbf{a} \in \mathcal{S}^\lambda   \}$$ is the complete list of all  irreducible graded $R^\lambda$-modules up to isomorphism and grading shift.
\end{enumerate}
\end{corollary}

In the following proposition, we give an upper bound for the number
of $\Delta_{(a,b)}$'s that can occur in
    $\Delta(\mathbf{a}(v))$ for $v\in B(\lambda)$.

\begin{proposition} \label{Lem: upper bound}
For $v \in B(\lambda)$, let $\eta(v)$ be the number of
$\Delta_{(a,b)}$'s in the outer tensor product
$\Delta(\mathbf{a}(v)) = \Delta(\mathbf{a}(v);1) \boxtimes \cdots
\boxtimes \Delta(\mathbf{a}(v);n)$. Then $ \eta(v) \le n
\lambda(h)$,    where
$$h = \begin{cases} h_1+ \cdots + h_n & \quad  (\mathsf{A}_n), \\
 2h_1+ \cdots + 2h_{n-1} + h_n & \quad  (\mathsf{B}_n), \\
 2(h_1+ \cdots  + h_n) & \quad  (\mathsf{C}_n), \\
 2h_1+ \cdots + 2h_{n-2} + h_{n-1} + h_n & \quad  (\mathsf{D}_n). \end{cases}$$
\end{proposition}
\begin{proof}
Let $\mathbf{a} = \mathbf{a}(v)$ and write $\mathbf{a} =
\mathbf{a}_1 * \cdots * \mathbf{a}_n$, where $\mathbf{a}_k =
(\mathsf{a}_{k,1}, \mathsf{a}_{k,2}, \ldots, \mathsf{a}_{k,l_{k}})$
is the subsequence of $\mathbf{a}$ defined in $\eqref{Eq: subseq of
a}$. Since $\mathsf{a}_{k,1} = t_{k,n+1-k}$ if $\mathfrak{A}$ is of
type $\mathsf{A}_n, \mathsf{B}_n,\mathsf{C}_n$ and $\mathsf{a}_{1,1}
= t_{1,n-1}$, $\mathsf{a}_{2,1} = t_{1,n}$, $\mathsf{a}_{k,1} =
t_{k-1,n+1-k}$ when $n \ge 3$ for type $\mathsf{D}_n$, it follows
from $\eqref{Eq: nabla}$ and $\eqref{Eq: theta}$ that
$$ \eta(v) \le  \mathsf{a}_{1,1} + \cdots + \mathsf{a}_{n-1,1} + \mathsf{a}_{n,1} .$$
Thus, it suffices to show that
$$ \mathsf{a}_{1,1} + \cdots + \mathsf{a}_{n-1,1} + \mathsf{a}_{n,1} \le n \lambda(h). $$

Let $\lambda = \lambda_1 \Lambda_1 + \cdots + \lambda_n \Lambda_n$.
If $n > 1$ for type $\mathsf{A}_n, \mathsf{B}_n,\mathsf{C}_n$ and $n
> 2$ for type $\mathsf{D}_n$, then using the description in
Proposition \ref{Prop: S lambda}, we will obtain
\begin{align}  \label{Eq: upper bound for a1}
\mathsf{a}_{n,1} \le \lambda(h)
\end{align}
from the following inequalities:

\begin{align*}
(\mathsf{A}_n): \  \mathsf{a}_{n,1} = t_{n,1} &\le \lambda_1 + t_{n,2} \le \lambda_1 + \lambda_2 + t_{n,3} \le \cdots \le \lambda_1 + \cdots + \lambda_n = \lambda(h),  \\
(\mathsf{B}_n): \  \mathsf{a}_{n,1} = t_{n,1} &\le \lambda_1 + t_{n,2} + t_{n,2n-2} - 2 t_{n,2n-1} \le \lambda_1 + \lambda_2 + t_{n,3}+ t_{n,2n-3} - t_{n,2n-2} - t_{n,2n-1} \\
& \le \quad   \cdots \quad \le \lambda_1 + \cdots + \lambda_{n-1} + t_{n,n} - t_{n,n+1} - t_{n,2n-1} \\
& \le \lambda_1 + \cdots + \lambda_{n} + t_{n,n+1} - t_{n,2n-1} \le \lambda_1 + \cdots + \lambda_{n} + \lambda_{n-1} + t_{n,n+2} - t_{n,2n-1} \\
& \le \quad  \cdots \quad \le 2\lambda_1 + 2\lambda_2 + \cdots + 2 \lambda_{n-1} + \lambda_n  = \lambda(h), \\
(\mathsf{C}_n): \  \mathsf{a}_{n,1} = t_{n,1} &\le \lambda_1 + t_{n,2} + t_{n,2n-2} - 2 t_{n,2n-1} \le \lambda_1 + \lambda_2 + t_{n,3}+ t_{n,2n-3} - t_{n,2n-2}  - t_{n,2n-1} \\
& \le \quad   \cdots \quad \le \lambda_1 + \cdots + \lambda_{n-1} + 2 t_{n,n}  - t_{n,n+1} - t_{n,2n-1}  \\
& \le \lambda_1 + \cdots + \lambda_{n-1} + 2\lambda_n + t_{n,n+1} - t_{n,2n-1}  \\
& \le \quad  \cdots \quad \le 2\lambda_1 + 2\lambda_2 + \cdots + 2
\lambda_{n-1} + 2\lambda_n  = \lambda(h),  \\
(\mathsf{D}_n): \  \mathsf{a}_{n,1} = t_{n-1,1} &\le \lambda_1 + t_{n-1,2} + t_{n-1,2n-3} - 2 t_{n-1,2n-2} \\
& \le \lambda_1 + \lambda_2 + t_{n-1,3}+ t_{n-1,2n-4} - t_{n-1,2n-3}  - t_{n-1,2n-2} \\
& \le \quad   \cdots \quad \le  \lambda_1 + \cdots + \lambda_{n-2} + t_{n-1,n-1}+ t_{n-1,n} - t_{n-1,n+1}  - t_{n-1,2n-2} \\
& \le  \lambda_1 + \cdots + \lambda_{n} + t_{n-1,n+1}  - t_{n-1,2n-2} \\
& \le \lambda_1 + \cdots + \lambda_{n} + \lambda_{n-2}+ t_{n-1,n+2}  - t_{n-1,2n-2}  \\
& \le \quad  \cdots \quad \le 2\lambda_1 + 2\lambda_2 + \cdots + 2
\lambda_{n-2} + \lambda_{n-1} + \lambda_n  = \lambda(h).
\end{align*}

We proceed by induction on $n$. If $n=2$ for type $\mathsf{A}_n,
\mathsf{B}_n,\mathsf{C}_n$ or $n=3$ for type $\mathsf{D}_n$, then
the assertion can be proved in the same manner as above. We assume
that $n>2$ for type $\mathsf{A}_n, \mathsf{B}_n,\mathsf{C}_n$ and
$n>3$ for type $\mathsf{D}_n$. Let $U_q(\g_{n-1})$ be the subalgebra
of $U_q(\g)$ generated by $e_i,f_i$ ($i \in I\setminus \{ 1 \} $)
and $q^{h}\ (h\in \mathsf{P}^\vee)$, and let $\mathfrak{B}$ be the
crystal obtained from $B(\lambda)$ by forgetting the $1$-arrows.
Note that $U_q(\g_{n-1})$ is of type $\mathsf{X}_{n-1}$ when
$U_q(\g)$ is of type $\mathsf{X}_n$ $(\mathsf{X=A,B,C,D})$. Let $u =
\kf_{\mathbf{s}_n}^{\mathsf{a}_n} b_\lambda$ where $\mathbf{s}_n$ is
the sequence given in Table \ref{Tb: w0}, and let $\mathcal{C}$ be
the connected component of $\mathfrak{B}$ containing $v$. Then $u$
is the highest weight vector of the $U_q(\g_{n-1})$-crystal
$\mathcal{C}$ with weight
$$ \wt(u) = \left\{
              \begin{array}{ll}
                \lambda - t_{n,1}\alpha_1 - \cdots - t_{n,n}\alpha_n & \quad (\mathsf{A}_n), \\
                \lambda - (t_{n,1}+t_{n,2n-1})\alpha_1 - \cdots - (t_{n,n-1}+t_{n,n+1})\alpha_{n-1} -  (t_{n,n})\alpha_n & \quad (\mathsf{B}_n, \mathsf{C}_n),\\
                \lambda - (t_{n-1,1}+t_{n-1,2n-2})\alpha_1 - \cdots - (t_{n-1,n-2}+t_{n-1,n+1})\alpha_{n-2} \\
\qquad \qquad \qquad \qquad\qquad \qquad \ \ \  \qquad \qquad -
t_{n-1,n}\alpha_{n-1} - t_{n-1,n-1}\alpha_{n} & \quad
(\mathsf{D}_n).
              \end{array}
            \right.$$

 Now suppose that $h' = h-h_1 \ (\mathsf{A}_n)$,  and $h' = h-2h_1 \ (\mathsf{B}_n, \mathsf{C}_n,
            \mathsf{D}_n)$.    Then we see that
\begin{align*}
& \alpha_1(h') = -1,\ \alpha_2(h') = 1, \quad   \alpha_i(h') = 0\ (i=3,\ldots,n) \quad (\mathsf{A}_n), \\
& \alpha_1(h') = -2,\ \alpha_2(h') = 2, \quad \alpha_i(h') = 0\
(i=3,\ldots,n)\quad  (\mathsf{B}_n, \mathsf{C}_n, \mathsf{D}_n).
\end{align*}
Proposition \ref{Prop: S lambda} implies \begin{align} \label{Eq:
upper bound for wt u} \wt(u)(h') \le \lambda(h)
\end{align}
by the following calculations:
\begin{align*}
(\mathsf{A}_n): \ \ \wt(u)(h') &= \lambda(h') + t_{n,1} - t_{n,2} \le \lambda(h), \\
(\mathsf{B}_n,\mathsf{C}_n ): \ \ \wt(u)(h') &= \lambda(h') + 2(t_{n,1} + t_{n,2n-1}) - 2(t_{n,2} + t_{n,2n-2}) \\
& = \lambda(h') + 2(t_{n,1} - t_{n,2} -t_{n,2n-2} + 2t_{n,2n-1}) - 2t_{n,2n-1}  \le \lambda(h), \\
(\mathsf{D}_n): \ \
\wt(u)(h') &= \lambda(h') + 2(t_{n,1} + t_{n,2n-2}) - 2(t_{n,2} + t_{n,2n-3}) \\
& = \lambda(h') + 2(t_{n,1} - t_{n,2} -t_{n,2n-3} + 2t_{n,2n-2}) -
2t_{n,2n-2}  \le \lambda(h) .
\end{align*}
Since the adapted string of $u$ is $\mathbf{a}_1*\cdots *
\mathbf{a}_{n-1}$, we obtain by the induction hypothesis that
\begin{align*}
\mathsf{a}_{1,1} + \cdots + \mathsf{a}_{n-1,1} \le (n-1) \wt(u)(h'),
\end{align*}
which yields, by $\eqref{Eq: upper bound for a1}$ and $\eqref{Eq:
upper bound for wt u}$,
\begin{align*}
\mathsf{a}_{1,1} + \cdots + \mathsf{a}_{n-1,1} + \mathsf{a}_{n,1}
\le (n-1) \lambda(h) + \lambda(h) = n \lambda(h).
\end{align*} \end{proof}

\vskip 1em

\begin{example} \label{Ex: example}
Assume $(\mathfrak A, \mathsf{P}, \Pi, \mathsf{P}^{\vee}, \Pi^{\vee})$ is type $\mathsf{B}_3$ and $\lambda=\Lambda_1+\Lambda_2+3\Lambda_3 \in \mathsf{P}^+$.  In this case,
$w_0=r_3(r_2r_3r_2)(r_1r_2r_3r_2r_1)$ and
$$\cB = \CrystalBf . $$
Kashiwara and Nakashima \cite{KN94} constructed combinatorial
realizations of highest weight crystals for finite classical types using
certain semistandard tableaux, which are now referred to as  {\it Kashiwara-Nakashima tableaux}.
We take this combinatorial model as a realization of the crystal $B(\lambda)$.
Let $T$ denote the following element in $B(\lambda)$
$$
T \ =\ \KNT{\overline{1}}{\overline{3}}{\overline{1}}{3}{0}{\overline{3}}{1}{\overline{3}}{\overline{2}}.
$$
Since $w_0=r_3(r_2r_3r_2)(r_1r_2r_3r_2r_1)$, using the crystal structure given in \cite{KN94}, we have
$$\mathbf{a}(T)=(3,3,3,0,4,3,5,2,1) \in \mathcal{S}^{\lambda},\qquad \triangle(\mathbf{a}(T)) = \exTriangleBB ,$$
$$
\begin{tabular}{ c c c c c}
\ $T_5 := \ke_1^4 (T_4)$ \ & \ $T_6 := \ke_2^3 (T_5)$\ &\ $T_7 := \ke_3^5 (T_6)$\ &\ $T_8 := \ke_2^2 (T_7)$\ &\ $T_9 := \ke_1^1 (T_8)$\ \\ [2ex]
 $\KNT{\overline{2}}{1}{\overline{2}}{1}{3}{0}{1}{2}{3}$ & $\KNT{\overline{3}}{1}{\overline{3}}{1}{2}{0}{1}{2}{3}$ &
$\KNT{3}{1}{3}{1}{2}{3}{1}{2}{3}$ & $\KNT{2}{1}{2}{1}{2}{3}{1}{2}{3}$
& $\KNT{1}{1}{2}{1}{2}{3}{1}{2}{3}$ \\
\end{tabular}
$$

\vskip 1ex

$$
\begin{tabular}{ c c c }
 \qquad $T_2 := \ke_2^3 (T_1)$  \qquad & \qquad $T_3 := \ke_3^3 (T_2)$  \qquad & \qquad $T_4 := \ke_2^0 (T_3)$ \qquad \\ [2ex]
 $\KNT{\overline{1}}{2}{\overline{1}}{2}{0}{\overline{3}}{1}{2}{\overline{3}}$ & $\KNT{\overline{1}}{2}{\overline{1}}{2}{3}{0}{1}{2}{3}$ &
$\KNT{\overline{1}}{2}{\overline{1}}{2}{3}{0}{1}{2}{3}$ \\
\end{tabular}
$$

\vskip 1ex

$$
\begin{tabular}{ c }
 \qquad $T_1 := \ke_3^3 (T)$  \qquad  \\ [2ex]
 $\KNT{\bar{1}}{3}{\overline{1}}{3}{0}{\overline{3}}{1}{3}{\overline{2}}$ \\
\end{tabular}
$$
Note that $\varepsilon_i(T_9)=0\ (i=1,2,3)$, $\varepsilon_i(T_4)=0\ (i=2,3)$ and $\varepsilon_i(T_1)=0\ (i=3)$.
It follows from $\eqref{Eq: nabla}$ and $\eqref{Eq: theta}$ that
\begin{align*}
&\vartheta_{1,3} = 1,\ \vartheta_{1,4} = 1,  \\
&\vartheta_{2,2} = 1,\ \vartheta_{2,3} = 1,\  \vartheta_{2,4} = 1,\ \vartheta_{2,5} = 0,   \\
&\vartheta_{3,1} = 1,\ \vartheta_{3,2} = 0,\  \vartheta_{3,3} = 1,\ \vartheta_{3,4} = 0,\  \vartheta_{3,5} = 1,\ \vartheta_{3,6} = 1,
\end{align*}

\noindent and using the definition of $\widehat \imath$ in \eqref{Eq: def of overhat i},  we have
\begin{align*}
 \Delta(\mathbf{a}(T); 1) & = \Delta_{(\widehat{3},\widehat{4})}^{\boxtimes\vartheta_{1,3}} \boxtimes \Delta_{(\widehat{3},\widehat{5})}^{\boxtimes\vartheta_{1,4}}\\
&= \Delta_{(\overline{3},0)} \boxtimes \Delta_{(\overline{3},3)},\\
 \Delta(\mathbf{a}(T); 2) &= \Delta_{(\widehat{2},\widehat{3})}^{\boxtimes\vartheta_{2,2}} \boxtimes \Delta_{(\widehat{2},\widehat{4})}^{\boxtimes\vartheta_{2,3}}
\boxtimes \Delta_{(\widehat{2},\widehat{5})}^{\boxtimes\vartheta_{2,4}} \boxtimes \Delta_{(\widehat{2},\widehat{6})}^{\boxtimes\vartheta_{2,5}} \\
& = \Delta_{(\overline{2},\overline{3})} \boxtimes \Delta_{(\overline{2},0)} \boxtimes \Delta_{(\overline{2},3)} , \\
 \Delta(\mathbf{a}(T); 3) &= \Delta_{(\widehat{1},\widehat{2})}^{\boxtimes\vartheta_{3,1}} \boxtimes \Delta_{(\widehat{1},\widehat{3})}^{\boxtimes\vartheta_{3,2}}
\boxtimes \Delta_{(\widehat{1},\widehat{4})}^{\boxtimes\vartheta_{3,3}} \boxtimes \Delta_{(\widehat{1},\widehat{5})}^{\boxtimes\vartheta_{3,4}}
\boxtimes \Delta_{(\widehat{1},\widehat{6})}^{\boxtimes\vartheta_{3,5}} \boxtimes \Delta_{(\widehat{1},\widehat{7})}^{\boxtimes\vartheta_{3,6}} \\
& = \Delta_{(\overline{1},\overline{2})} \boxtimes \Delta_{(\overline{1},0)}   \boxtimes \Delta_{(\overline{1},2)} \boxtimes \Delta_{(\overline{1},1)} .
\end{align*}
By Theorem \ref{Thm: main theorem}, the $R^\lambda$-module $M$ corresponding to $T$ is given as follows:
\begin{align*}
M &= \Psi^\lambda(T) \\
&= \hd\,\ind \big( \Delta(\mathbf{a}(T); 1) \boxtimes \Delta(\mathbf{a}(T); 2) \boxtimes \Delta(\mathbf{a}(T); 3)  \big)  \\
&=\hd\,\ind \left(
\Delta_{(\overline{3},0)} \boxtimes \Delta_{(\overline{3},3)} \boxtimes
\Delta_{(\overline{2},\overline{3})} \boxtimes \Delta_{(\overline{2},0)} \boxtimes \Delta_{(\overline{2},3)} \boxtimes
\Delta_{(\overline{1},\overline{2})} \boxtimes \Delta_{(\overline{1},0)}  \boxtimes \Delta_{(\overline{1},2)} \boxtimes \Delta_{(\overline{1},1)}
\right).
\end{align*}
In this example,  $\eta(T) = 9 <  n\lambda(h) = 3(\Lambda_1+\Lambda_2 + 3 \Lambda_3)(2h_1 + 2 h_2 + h_3) = 21$.

Recall the definition of $\mathcal{N}_k(T)$ in $\eqref{Eq: def of Nk}$ (here the element $v$ in $B(\lambda)$ is the tableau $T$).  Then
$$\mathcal{N}_1(T) = \hd\,\ind \big( \Delta(\mathbf{a}(T); 1)\big),\quad \mathcal{N}_2(T) = \hd\,\ind \big( \Delta(\mathbf{a}(T); 2)\big),\quad \mathcal{N}_3(T) = \hd\,\ind \big(\Delta(\mathbf{a}(T); 3)\big). $$
We will prove that such a realization of  these modules exists for all finite classical types  in the next section (Lemma \ref{Lem: gluing lemma}).
In this example, we give an intuitive picture for the case of $\mathcal{N}_3(T)$. For $a,b \in \cB$ with $a \succ b$, by the definition of $\Delta_{(a,b)}$ in $\eqref{Eq: def of nabla}$ we can identify
$\Delta_{(a,b)}$ with the segment of $\cB$ between $a$ and $b$.
Then we have the following diagram.

\begin{align*}
\begin{tabular}{c l}
  $\Delta_{(\overline{1},\overline{2})}$ & $\CrystalBfi$ \\
  $\Delta_{(\overline{1},0)}$ & $\CrystalBfii$ \\
  $\Delta_{(\overline{1},2)}$ & $\CrystalBfiii$ \\
  $\Delta_{(\overline{1},1)}$ & $\CrystalBfiv$ \\  [2ex] \hline
 & \\  [-2ex]
  & \qquad  $4 \times \kf_1$ \quad \ $3 \times \kf_2$ \quad \ \qquad $5 \times \kf_3$ \ \ \quad \qquad $2 \times \kf_2$ \ \  \quad  $1 \times \kf_1$
\end{tabular}
\end{align*}
The Kashiwara operators at the bottom of the diagram are obtained by adding up vertically the number of  $i$-colored arrows in the segments.
By Proposition \ref{Prop: description of Nk}, we know
$$\mathcal{N}_3(T) = \kf_{1}^4 \kf_{2}^3 \kf_{3}^5 \kf_{2}^2 \kf_{1}^1 \mathbf{1}.$$
Since $\mathcal{N}_3(T) = \hd\,\ind \Big( \Delta_{(\overline{1},\overline{2})} \boxtimes \Delta_{(\overline{1},0)}
 \boxtimes \Delta_{(\overline{1},2)} \boxtimes \Delta_{(\overline{1},1)} \Big)$, taking the head of
$\ind\Big( \Delta_{(\overline{1},\overline{2})} \boxtimes \Delta_{(\overline{1},0)}   \boxtimes \Delta_{(\overline{1},2)} \boxtimes \Delta_{(\overline{1},1)}\Big)$
can be understood as summing vertically the $i$-colored arrows in the segments corresponding to the modules
$\Delta_{(\overline{1},\overline{2})},\  \Delta_{(\overline{1},0)},\ \Delta_{(\overline{1},2)}$ and $ \Delta_{(\overline{1},1)}$.

Now from the crystal structure in \cite{KN94}, we know that
$$ \kf_1 T = \KNT{\overline{1}}{\overline{3}}{\overline{1}}{3}{0}{\overline{3}}{2}{\overline{3}}{\overline{1}}, \qquad
\kf_2 T = \KNT{\overline{1}}{\overline{3}}{\overline{1}}{3}{0}{\overline{2}}{1}{\overline{3}}{\overline{2}}, \qquad
\kf_3 T = 0, $$
which yield $\mathbf{a}(\kf_1 T)=(3,2,1,0,5,4,7,2,1) \in \mathcal{S}^{\lambda}$, $\mathbf{a}(\kf_2 T)=(3,4,3,0,4,3,5,2,1) \in \mathcal{S}^{\lambda}$ and
\begin{align*}
 \triangle(\mathbf{a}(\kf_1 T)) = \exTriangleBBi, \qquad \triangle(\mathbf{a}(\kf_2 T)) = \exTriangleBBii.
\end{align*}
Therefore, by Theorem \ref{Thm: main theorem}, the irreducible modules $\kf_1( M ) = \hd\,\ind \left( L(1) \boxtimes M \right)$ and
$\kf_2( M ) = \hd\,\ind \left( L(2) \boxtimes M \right)$ can be gotten as follows:
\begin{align*}
\kf_1(M) &= \kf_1(\Psi^\lambda(T)) = \Psi^\lambda(\kf_1 T) \\
&=\hd\,\ind
\Big(
\Delta_{(\overline{3},0)}
\boxtimes \Delta_{(\overline{3},{3})}
\boxtimes \Delta_{(\overline{2},\overline{3})}
\boxtimes \Delta_{(\overline{2},0)}
\boxtimes \Delta_{(\overline{1},\overline{2})}
\boxtimes \Delta_{(\overline{1},0)}
\boxtimes \Delta_{(\overline{1},3)}
\boxtimes \Delta_{(\overline{1},2)}
\boxtimes \Delta_{(\overline{1},1)}
\Big), \\
\kf_2(M) &= \kf_2(\Psi^\lambda(T)) = \Psi^\lambda(\kf_2 T) \\
&=\hd\,\ind \Big(
\Delta_{(\overline{3},0)} \boxtimes \Delta_{(\overline{3},3)} \boxtimes
\Delta_{(\overline{2},\overline{3})}^{\boxtimes 2} \boxtimes \Delta_{(\overline{2},0)} \boxtimes \Delta_{(\overline{2},3)} \boxtimes
\Delta_{(\overline{1},\overline{2})} \boxtimes \Delta_{(\overline{1},0)}  \boxtimes \Delta_{(\overline{1},2)} \boxtimes \Delta_{(\overline{1},1)}
\Big).
\end{align*}

Note that by the same argument as above, one can compute the action of Kashiwara operators on $\mathbb{B}(\infty)$ explicitly
using the combinatorial realizations of $B(\infty)$ for finite classical type  in \cite{Cli98,HL08}.
\end{example}
\section{Proof of Theorem \ref{Thm: main theorem}} \label{Sec: final}

This section is devoted to a proof of Theorem \ref{Thm: main theorem}.
We first give a sufficient condition for  $\ind \left( \Delta_{(a,b)} \boxtimes \Delta_{(c,d)}\right ) \simeq \ind \left( \Delta_{(c,d)} \boxtimes \Delta_{(a,b)}\right)$ to hold  (Lemma \ref{Lem: commuting lemma}). Then we extend the result in \cite[Lem.~4.3]{KP10} for type $\mathsf{A}_n$
to the other finite classical types in Lemma \ref{Lem: seq of kf} below, and prove Lemma \ref{Lem: gluing lemma} which yields $\mathcal{N}_n(v) = \hd\,\ind \Delta(\mathbf{a}(v);n)$
for $v \in B(\infty)$. From that we can deduce $\mathcal{N}_k(v) = \hd\,\ind \Delta(\mathbf{a}(v);k)$ for all $k$.
Using Lemma \ref{Lem: sum of fi}, Proposition \ref{Prop: crystal iso},  and  Lemma \ref{Lem: gluing lemma},
we argue that the map $\Psi: B(\infty) \longrightarrow \B(\infty)$ is a crystal isomorphism. It then follows
from the crystal embedding $\iota_\lambda: B(\lambda) \rightarrow B(\infty)\otimes T_\lambda \otimes C$ in \eqref{Eq: crystal embedding} that
the map $\Psi^\lambda: B(\lambda) \longrightarrow \B(\lambda)$ is a crystal isomorphism too.

Given $\mathbf{i} \in I^\alpha$ and $\mathbf{j} \in
I^\beta$, a sequence $\mathbf{k}\in I^{\alpha+\beta}$ is
called a {\em shuffle} of $\mathbf{i}$ and $\mathbf{j}$ if
$\mathbf{k}$ is a permutation of $\mathbf{i} * \mathbf{j}$ such that
$\mathbf{i}$ and $\mathbf{j}$ are subsequences of $\mathbf{k}$.
For $X = \sum x_\mathbf{i}\ \mathbf{i}$ and $Y = \sum
y_\mathbf{j}\ \mathbf{j}$, we define $X \star Y$ by $$ X \star Y = \sum_{\mathbf{k}}
x_\mathbf{i} y_\mathbf{j}    \ \mathbf{k}, $$ where
$\mathbf{k}$ runs over all the shuffles of $\mathbf{i}$ and
$\mathbf{j}$.  Suppose $M\in R(\alpha)\fmod$ and $N \in R(\beta)\fmod$, and assume
$Q$ is a quotient (resp. $L$ is a submodule) of $\ind(M \boxtimes N)$.
Then
\begin{align}
&\text{$ \ch\big(\ind(M \boxtimes N)\big) = \ch(M) \star \ch(N),$ } \label{Eq: shuffle 1} \\
& \text{any term of $\ch Q$ (resp.\ $\ch L$) is a shuffle of some $\mathbf{i} \in \ch (M)$ and $\mathbf{j} \in \ch (N)$.} \label{Eq: shuffle 2}
\end{align}
The quantum Serre relations $\eqref{Eq: Serre}$ imply that
\begin{align}
& \dm (\e(\ldots, i,j,\ldots) M) =   \dm(\e(\ldots,j,i, \ldots) M) \qquad  \hspace{2.1truecm}  \text{ if } {a}_{ij} = 0, \label{Eq: Serre1}  \\
&  2 \dm (\e(\ldots, i,j,i,\ldots) M)  \label{Eq: Serre2}  \\
& \ \ \ \  = \dm(\e(\ldots,  i,i,j,\ldots) M)  + \dm (\e(\ldots, j,i,i,\ldots) M) \qquad \ \  \ \text{ if } {a}_{ij} = -1,
\nonumber  \\
&\dm  (\e(\ldots,  i,i,i,j,\ldots) M)  + 3 \dm (\e(\ldots, i,j,i,i,\ldots) M) \label{Eq: Serre3}  \\
& \ \   \ \   = \dm(\e(\ldots, j,i,i,i,\ldots) M) + 3 \dm (\e(\ldots, i,i,j,i,\ldots) M)  \ \, \ \text{ if } {a}_{ij} = -2,  \nonumber
\end{align}
for $M \in R\fmod$.

Using the description of $\Delta_{(a,b)}$ in Section \ref{Sec: construction for classical type}, we obtain a surjective homomorphism
\begin{align} \label{Eq: surj between nablas}
\ind\left (\Delta_{(a,b)} \boxtimes \Delta_{(b,c)}\right) \twoheadrightarrow \Delta_{(a,c)}
\end{align}
for $a,b,c \in \cB$ with $a \succ b \succ c$.



\begin{lemma} \label{Lem: commuting lemma}
Let $a, b, c, d \in \cB$ with $a \succeq c  \succ d \succeq b$. Suppose that one of
the following conditions hold:
\begin{enumerate}
\item $\mathfrak{A}$ is of type $\mathsf{A}_n$,
\item $\mathfrak{A}$ is of type $\mathsf{B}_n$ and one of the following three holds:
\begin{enumerate}
\item[(i)] $0 \succeq a$ and $c \ne 0$,
\item[(ii)] $b \succeq 0$ and $d \ne 0$,
\item[(iii)] $a \succeq \overline{n} $, ${n}\succeq b  $ and either $c=0$ or $d=0$,
\end{enumerate}
\item $\mathfrak{A}$ is of type $\mathsf{C}_n$ and one of the following three holds:
\begin{enumerate}
\item[(i)] $b \succeq \overline{n}$,
\item[(ii)] $a=c= \overline{n}$,
\item[(iii)] $a=\overline{1}$ and $c= \overline{n}$,
\end{enumerate}
\item $\mathfrak{A}$ is of type $\mathsf{D}_n$ and one of the following three holds:
\begin{enumerate}
\item[(i)] $b \succeq \overline{n-1}$,
\item[(ii)] $a=c= \overline{n-1}$,
\item[(iii)] $a=\overline{1}$ and $c= \overline{n-1}$.
\end{enumerate}
\end{enumerate}
Then
$$ \ind \left(\Delta_{(a, b)} \boxtimes \Delta_{(c,d)}\right) \simeq \ind\left(\Delta_{(c, d)} \boxtimes \Delta_{(a,b)}\right)$$
and these modules are irreducible.
\end{lemma}
\begin{proof}
Let
$$ M =  \ind \left( \Delta_{(a, b)} \boxtimes \Delta_{(c,d)}\right), \quad N = \ind \left( \Delta_{(c, d)} \boxtimes \Delta_{(a,b)}\right),$$
and let $Q$ (resp.\ $L$) be any nonzero quotient (resp.\ submodule) of $M$. We will assume that $\mathbf{k}$
is a certain sequence of elements  in $I$ such that $ \mathbf{k} \in \ch M$, $ \mathbf{k} \in \ch Q$, and $ \mathbf{k} \in \ch L $.
Let $\xi_M$ (resp.\ $\xi_Q$, $\xi_L$) be the multiplicity of $\mathbf{k}$ in $\ch M$ (resp.\ $\ch Q$, $\ch L$).
If we can show that $\xi_M = \xi_Q = \xi_L$, we can conclude that $M$ is irreducible  and $M \simeq N$, since $\ch M = \ch N$.

For type $\mathsf{A}_n$, the assertion follows from \cite[Lem.~4.1]{KP10}.
We assume that $\mathfrak{A}$ is of type $\mathsf{B}_n$.
If either (i) or (ii) holds, the proof is identical to that for type $\mathsf{A}_n$. So suppose that (iii) holds. Without loss
of generality, we may assume that $d=0$. Write $\mathbf{i} = \mathbf{i}(a,b) = (i,\ldots, n,n,\ldots, j) $,
$\mathbf{j} = \mathbf{i}(c,d) = (p,p+1,\ldots, n-1, n)$, and let
$$ \mathbf{k} = (i,\ldots, p,p,p+1,p+1, \ldots, n-1,n-1,n,n,n,n-1,n-2, \ldots, j).$$
Note that $\mathbf{k}$ is a shuffle of $\mathbf{i}$ and $\mathbf{j}$. Let $l$ be
the length of $\mathbf{j}$.
It follows from $\eqref{Eq: cha of nabla}$ and $\eqref{Eq: shuffle 1}$ that
the term $\mathbf{k}$ appears in $\ch M$ with multiplicity $\xi_M  = 2 \times 2^{l-1} 3$.

By Lemma \ref{Lem: Frobenius reciprocity} and $\eqref{Eq: cha of nabla}$,
$\mathbf{i} * \mathbf{j}$ appears in the character $\ch Q$ of any quotient module $Q \subseteq M$.  We claim that $\mathbf{k}$ also appears in $\ch Q$.
Note that $a_{n, n-1}=-2$ for type $\mathsf{B}_n$.

We first assume that $ p \ne n-1,n.$
By the quantum Serre relations $\eqref{Eq: Serre}$, if $ j > p+1$, then
\begin{align*}
\mathbf{i} * \mathbf{j} =& (i,\ldots, n,n,\ldots, j, p,p+1, \ldots, n) \in \ch Q , \\
\fourthree &\ (i,\ldots,p, p+1, p, \ldots, n,n,\ldots, j, p+1, \ldots, n) \in \ch Q, \\
\ftf &\ (i,\ldots,p, p, p+1, \ldots, n,n,\ldots, j, p+1, \ldots, n) \in \ch Q.
\end{align*}
If $ j = p+1 $, then
\begin{align*}
\mathbf{i} * \mathbf{j} =& (i,\ldots, n,n,\ldots, p+1, p,p+1, \ldots, n) \in \ch Q , \\
\ftf &\ (i,\ldots, n,n,\ldots, p, p+1,p+1, \ldots, n) \in \ch Q, \\
\fourthree &\ (i,\ldots, p, p+1,p,\ldots,  n,n,\ldots, p+1,p+1, \ldots, n) \in \ch Q,\\
\ftf &\ (i,\ldots, p, p,p+1, \ldots, n,n,\ldots, p+1,p+1, \ldots, n) \in \ch Q.
\end{align*}
If $ j = p $, then
\begin{align*}
\mathbf{i} * \mathbf{j} =& (i,\ldots, n,n,\ldots, p+1, p, p,p+1 \ldots, n) \in \ch Q , \\
\fourfour &\ (i,\ldots, n,n,\ldots, p, p+1,p, p+1, \ldots, n) \in \ch Q, \\
\fourthree &\ (i,\ldots, p, p+1,p, \ldots, n,n,\ldots, p,p+1, \ldots, n) \in \ch Q, \allowdisplaybreaks  \\
\ftf &\ (i,\ldots, p, p,p+1, \ldots, n,n,\ldots, p,p+1, \ldots, n) \in \ch Q.
\end{align*}
If $ j < p $, then
\begin{align*}
\mathbf{i} * \mathbf{j} =& (i,\ldots, n,n,\ldots, j, p,p+1 \ldots, n) \in \ch Q , \\
\fourfour &\ (i,\ldots, n,n,\ldots, p+1, p,p-1, p, \ldots,j,p+1,\ldots, n) \in \ch Q, \\
\ftf &\ (i,\ldots, n,n,\ldots, p+1, p,p, p-1, \ldots,j,p+1,\ldots, n) \in \ch Q, \\
\fourfour &\ (i,\ldots,  n,n,\ldots, p,p+1,p,p-1, \ldots,j,p+1,\ldots, n) \in \ch Q,\\
\fourthree &\ (i,\ldots, p, p+1,p, \ldots, n,n,\ldots, ,j,p+1,\ldots, n) \in \ch Q,\\
\ftf &\ (i,\ldots, p, p,p+1, \ldots, n,n,\ldots, ,j,p+1,\ldots, n) \in \ch Q.
\end{align*}
Applying this argument repeatedly, we determine that
\begin{align*}
& \hspace{-.8truecm} (i,\ldots, p, p, \ldots,n-2,n-2, n-1, n,n, \ldots  ,j, n-1, n) \in \ch Q \\
& \ftthf  (i,\ldots, p, p, \ldots,n-1, n,n,n-1,n-1, \ldots  ,j,  n) \in \ch Q,  \\
& \ftf \   (i,\ldots, p, p, \ldots,n-1, n,n-1,n,n-1, \ldots  ,j,  n) \in \ch Q  \\
&\ftf  (i,\ldots, p, p, \ldots,n-1, n-1,n,n,n-1, \ldots  ,j,  n) \in \ch Q, \allowdisplaybreaks \\
& \fourthree \  (i,\ldots, p, p, \ldots,n-1, n-1,n,n,n-1,  n, \ldots  ,j) \in \ch Q, \\
&  \ftfi  \  \mathbf{k} = (i,\ldots, p, p, \ldots,n-1, n-1,n,n,n,n-1, \ldots  ,j) \in \ch Q.
\end{align*}
We suppose that $ p = n-1,n$. Then, the same argument as above gives
$$
\mathbf{k} = (i,\ldots, n-1, n-1,n,n,n,n-1, \ldots  ,j) \in \ch Q.
$$

Now for any nonzero submodule $L$ of $M$, $\ch L$ contains $\mathbf{j}*\mathbf{i}$ by Lemma \ref{Lem: Frobenius reciprocity}
and $\eqref{Eq: cha of nabla}$. If $a = c$, then
\begin{align*}
 & \mathbf{j} * \mathbf{i}  = (i,i+1, \ldots, n, i,\ldots, n,n,\ldots, j) \in \ch L , \\
& \quad \  \fourthree  \ (i,i+1, i,  \ldots, n, i+1,\ldots, n,n,\ldots, j) \in \ch L,  \\
& \ \  \ftf \ (i,i, i+1 ,  \ldots, n, i+1,\ldots, n,n,\ldots, j) \in \ch L,   \\
 &\hspace{4.5truecm}  \vdots \\
& \ \  \ftf \ (i,i, i+1 , i+1,  \ldots, n-2,n-2, n-1, n, n-1,n,n,\ldots, j) \in \ch L,    \\
& \ \  \ftf \ \mathbf{k} = (i,i, i+1 , i+1,  \ldots, n-2,n-2, n-1, n-1, n,n,n,\ldots, j) \in \ch L. \allowdisplaybreaks
\end{align*}
If $a \succ c$, then
\begin{align*}
& \mathbf{j} * \mathbf{i}  = (p, \ldots, n-1, n, i,\ldots, n,n,\ldots, j) \in \ch L ,   \\
& \quad \fourthree \ (p,  \ldots, n-1, i,\ldots, n, n-1, n,n,\ldots, j) \in \ch L,  \\
& \  \ftfi \ (p,  \ldots, n-1, i,\ldots, n-1, n, n,n,\ldots, j) \in \ch L, \\
&\hspace{4.5truecm}  \vdots \\
& \  \ftf \ \mathbf{k} = (i,\ldots, p , p,p+1 , p+1,  \ldots, n-1,
n-1, n,n,n,\ldots, j) \in \ch L. \end{align*} Hence, $\ch L$
contains $\mathbf{k}$ in either event.   From the structure  of
$\mathbf{k}$ we see that it occurs in $\ch Q$ (resp.\ $\ch L$) with
multiplicity $\xi_Q \ge 2^{l-1} 3!$ (resp.\ $\xi_L \ge 2^{l-1} 3!$).
Therefore, since $\xi_M =  2 \times 2^{l-1} 3$, we have $\xi_M =
\xi_Q = \xi_L$,  which implies the result for type
$\mathsf{B}_n$.

Now assume that $\mathfrak{A}$ is of type $\mathsf{C}_n$.
If either (i) or (ii) holds, the proof is identical to that for type $\mathsf{A}_n$. We assume that (iii) holds.
Write $\mathbf{i} = \mathbf{i}(a,b) = (1,2,\ldots,n,\ldots, j)$, $\mathbf{j} = \mathbf{i}(c,d) = (n,n-1,\ldots, q)$.  Then
$$
\mathbf{k} := (1,2, \ldots, n-1, n,n,n-1,n-1, \ldots, q,q,\ldots, j).
$$  is a shuffle of $\mathbf{i}$ and $\mathbf{j}$. Let $l$ be the length of $\mathbf{j}$.
Then  $\eqref{Eq: cha of nabla}$ and $\eqref{Eq: shuffle 1}$ imply that the term $\mathbf{k}$ appears in $\ch M$ with
multiplicity $\xi_M = 2^l$.

For any nonzero quotient $Q$ of $M$, $\ch Q$ contains $\mathbf{i}*\mathbf{j}$
by $\eqref{Eq: cha of nabla}$ and Lemma \ref{Lem: Frobenius reciprocity}.
We claim that $\mathbf{k}$ occurs in $\ch Q$ as a term.  Note that $a_{n-1, n}=-2$.
The quantum Serre relations $\eqref{Eq: Serre}$ and $\eqref{Eq: shuffle 2}$ imply
\begin{align*}
\mathbf{i}*\mathbf{j}\ = & \ \  (1,\ldots,n,\ldots,j,n,\ldots,q) \in \ch Q,  \\
  \fourthree &\  \ (1,\ldots,n-1,n,n-1,n,\ldots,j,n-1,\ldots, q) \in \ch Q, \\
 \ftf &\  \ (1,\ldots,n-1,n,n,n-1,\ldots,j,n-1,\ldots, q) \in \ch Q.
\end{align*}
Continuing this reasoning gives
$$ \mathbf{k} = (1, \ldots, n,n, \ldots, q,q,\ldots, j) \in \ch Q. $$

For any nonzero submodule $L$ of $M$, $\ch L$ contains $\mathbf{j}*\mathbf{i}$ by Lemma \ref{Lem: Frobenius reciprocity}
and $\eqref{Eq: cha of nabla}$.
Then, using a similar argument, we have
\begin{align*}
\mathbf{j}*\mathbf{i}\ = & \ \  (n,\ldots,q,1,2,\ldots,n,\ldots,j) \in \ch L, \\
\ftthf &(n,n-1,1,\ldots, n-1,n,n-1,\ldots,q,q,\ldots, j) \in \ch L,   \\
\ftthf &(n,1,\ldots,n-2,n-1, n-1,n,n-1,\ldots,q,q,\ldots, j) \in \ch L,  \allowdisplaybreaks\\
\ftfi &\ (n,1,\ldots,n-2,n-1, n,n-1,n-1,\ldots,q,q,\ldots, j) \in \ch L, \\
\fourthree &\ (1,\ldots,n,n-1, n,n-1,n-1,\ldots,q,q,\ldots, j) \in \ch L,  \\
\ftf &\ \mathbf{k} = (1,\ldots,n, n,\ldots,q,q,\ldots, j) \in \ch L.
\end{align*}
As  $\mathbf{k}$
occurs in $\ch Q$ (resp.\ $\ch L$) with multiplicity $\xi_Q \ge 2^l$ (resp.\ $\xi_L \ge 2^l$), and
since $\xi_M= 2^l $, the assertion follows from $\xi_M = \xi_Q = \xi_L$.

Lastly, we assume that $\mathfrak{A}$ is of type $\mathsf{D}_n$.
If (i) holds, then the proof is identical to that for type $\mathsf{A}_n$.

Suppose that (ii) holds.
Write $\mathbf{i} = \mathbf{i}(a,b) = (n,n-1,n-2, \ldots, j)$, $\mathbf{j} = \mathbf{i}(c,d) = (n,n-1,n-2,\ldots, q)$.
Then
$$
\mathbf{k} := (n,n,n-1,n-1,n-2,n-2, \ldots, q,q, \ldots, j).
$$
is a shuffle of $\mathbf{i}$ and $\mathbf{j}$. Let $l$ be the length of $\mathbf{j}$.
The equations  $\eqref{Eq: cha of nabla}$ and $\eqref{Eq: shuffle 1}$ imply that the term $\mathbf{k}$ appears in $\ch M$ with
multiplicity $\xi_M = 2^l$.

By Lemma \ref{Lem: Frobenius reciprocity} and $\eqref{Eq: cha of nabla}$,
$\mathbf{i} * \mathbf{j}$ appears in the character $\ch Q$ of any quotient module $Q \subseteq M$. We will show that $\mathbf{k}$ also appears in $\ch Q$.
Note  that $ a_{n, n-2}= a_{n-2, n} = a_{n-1, n-2} = a_{n-2, n-1} = -1$
and $a_{n-1,n} = a_{n,n-1} = 0$. The quantum Serre relations $\eqref{Eq: Serre}$ imply
\begin{equation*}
\begin{aligned}
\mathbf{i}*\mathbf{j}\ = & \ \  (n,n-1,n-2\ldots,j,n,n-1,n-2,\ldots,q) \in \ch Q,  \\
  \fourthree &\  \ (n-1,n,n-2, n,\ldots,j, n-1,n-2,\ldots,q) \in \ch Q, \\
   \ftthf &\  \ (n,n,n-1, n-2,\ldots,j, n-1,n-2,\ldots,q) \in \ch Q, \\
  \ftf &\  \ (n,n,n-1, n-2,n-1\ldots,j, n-2,\ldots,q) \in \ch Q, \\
    \ftfi &\  \ (n,n,n-1, n-1,n-2\ldots,j, n-2,\ldots,q) \in \ch Q.
\end{aligned}
\end{equation*}
Continuing this reasoning gives
$$ \mathbf{k} = (n,n,n-1,n-1,n-2,n-2, \ldots, q,q, \ldots, j) \in \ch Q. $$
For any nonzero submodule $L$, since $\ch L$  contains $\mathbf{j} * \mathbf{i}$ by Lemma \ref{Lem: Frobenius reciprocity}, the same argument shows
that $\ch L$ contains $\mathbf{k}$. Since  $\mathbf{k}$
occurs in $\ch Q$ (resp.\ $\ch L$) with multiplicity $\xi_Q \ge 2^l$ (resp.\ $\xi_L \ge 2^l$), the assertion follows from $\xi_M = \xi_Q = \xi_L$.

We now assume that (iii) holds.
Since the proof is identical to that for type $\mathsf{A}_n$ if $b=\overline{n}$ or $b=n$, we may assume that $n-1 \succeq b$.
Write $\mathbf{i} = \mathbf{i}(a,b) = (1, \ldots, n-2,n,n-1,n-2, \ldots, j)$, $\mathbf{j} = \mathbf{i}(c,d) = (n,n-1,n-2,\ldots, q)$ and
$$
\mathbf{k} := (1, \ldots, n-2, n,n,n-1,n-1,n-2,n-2, \ldots, q,q, \ldots, j).
$$
Let $l$ be the length of $\mathbf{j}$. Then $\mathbf{k}$ is a shuffle of $\mathbf{i}$ and $\mathbf{j}$ and appears in $\ch M$ with
multiplicity $\xi_M = 2^l$.

For any nonzero quotient $Q$ of $M$, $\ch Q$ contains $\mathbf{i}*\mathbf{j}$
by $\eqref{Eq: cha of nabla}$ and Lemma \ref{Lem: Frobenius reciprocity}.
Using the same argument as in case (ii), we obtain
$$ \mathbf{k} = (1, \ldots, n-2,n,n,n-1, n-1,n-2\ldots,j, n-2,\ldots,q) \in \ch Q. $$

For any nonzero submodule $L$ of $M$, $\ch L$ contains $\mathbf{j}*\mathbf{i}$ by Lemma \ref{Lem: Frobenius reciprocity}
and $\eqref{Eq: cha of nabla}$. By the same argument as in the $\mathsf{C}_n$ case, we have
\begin{equation*}
\begin{aligned}
\mathbf{j}*\mathbf{i}\ = & \ \  (n,\ldots,q-1,q, 1, \ldots, n-2, n,n-1,n-2\ldots,j) \in \ch L,   \\
     \ftthf &\  (n,n-1,n-2,1,\ldots, n-2,n,n-1,n-2,\ldots,q,q,\ldots, j)   \in \ch L, \\
     \ftf &\  (n,n-1,1,\ldots,n-2,n-2,n,n-1,n-2,\ldots,q,q,\ldots, j)   \in \ch L, \\
     \fourfour &\   (n,n-1,1,\ldots,n-2,n,n-2, n-1,n-2,\ldots,q,q,\ldots, j) \in \ch L, \\
     \ftf &\   (n,n-1,1,\ldots,n-2,n,n-1,n-2, n-2,\ldots,q,q,\ldots, j) \in \ch L, \\
     \fourthree &\   (n,1,\ldots,n-1,n-2,n-1,n,n-2, n-2,\ldots,q,q,\ldots, j) \in \ch L, \\
     \ftf &\   (n,1,\ldots,n-2,n,n-1,n-1,n-2, n-2,\ldots,q,q,\ldots, j) \in \ch L, \\
     \fourthree  &\  (1,\ldots,n,n-2,n,n-1,n-1,n-2, n-2,\ldots,q,q,\ldots, j) \in \ch L, \\
     \ftf  &\  \mathbf{k} = (1,\ldots,n-2,n,n,n-1,n-1,n-2, n-2,\ldots,q,q,\ldots, j) \in \ch L. \\
\end{aligned}
\end{equation*}

Since the sequence  $\mathbf{k}$ occurs in $\ch Q$
(resp.\ $\ch L$) with multiplicity $\xi_Q \ge 2^l$ (resp.\ $\xi_L
\ge 2^l$), by $\xi_M = 2^l $, the assertion follows
from $\xi_M = \xi_Q = \xi_L$.
\end{proof}

For $a,b \in \cB$, we define
\begin{equation} \label{eq:delta} \delta(a\succeq b) = \left\{
                   \begin{array}{ll}
                     1 & \hbox{if } a \succeq b, \\
                     0 & \hbox{otherwise}.
                   \end{array}
                 \right.
\end{equation}
In the $\mathsf{A}_n$ case, it follows from \cite[Lem.~4.3]{KP10} that, for $d_k \in \Z_{\ge0}$ and $b_k \in \cB$ with $b_1 \succ b_2 \succ \cdots \succ b_m$,
\begin{align} \label{Eq: seq of kf for An}
\kf_1^{t_1} \kf_2^{t_2} \cdots \kf_n^{t_n} \mathbf{1} \simeq \ind \Big(\dbt{k=1}{m} \Delta_{(\overline{1},b_k)}^{\boxtimes d_k}\Big),
\end{align}
where $t_i = \sum_{j=1}^m d_j \delta(\overline{\imath} \succ b_j)$ for $i=1,\ldots, n$.
Next we will extend this result to the other finite classical types.
We first divide the crystal $\cB$ for types $\mathsf{B}_n, \mathsf{C}_n, \mathsf{D}_n$ into two pieces so that
each part is almost the same as the crystal or the dual of the crystal
for type $\mathsf{A}$.
For type $\mathsf{B}_n$ (resp.\ $\mathsf{C}_n$, $\mathsf{D}_n$), the crystal $\cB$ is cut at the element $0$ (resp.\ $\overline{n}$, $\overline{n-1}$).
In the next result  we give an analogue of $\eqref{Eq: seq of kf for An}$ for each part.

\begin{lemma} \label{Lem: seq of kf}
Let $b_k \in \cB$, and $d_k \in \Z_{\ge 0}$ for $k=1, \ldots, m$. Set
$$ b_{\bullet} = \left\{
                   \begin{array}{ll}
                     \overline{n+1} & \qquad  (\mathsf{A}_n), \\
                     0 & \qquad (\mathsf{B}_n),\\
                     \overline{n} & \qquad (\mathsf{C}_n),\\
                     \overline{n-1} & \qquad (\mathsf{D}_n).
                   \end{array}
                 \right.  $$
\begin{enumerate}
\item If $b_1 \succ b_2 \succ \cdots \succ b_m \succeq b_{\bullet}$, then
$$ \kf_1^{t_1} \kf_2^{t_2} \cdots \kf_l^{t_l} \mathbf{1} \simeq \ind \Big(\dbt{k=1}{m} \Delta_{(\overline{1},b_k)}^{\boxtimes d_k}\Big) ,$$
where $ l = n\  \ (\mathsf{A}_n, \mathsf{B}_n)$,  \ $l = n-1\ \ (\mathsf{C}_n)$, \ $ l = n-2\ \ (\mathsf{D}_n)$ and $t_i := \sum_{j=1}^m d_j \delta(\overline{\imath} \succ b_j)$ for $i=1,\ldots, l$.
\item Let
$$ t_i = \left\{
           \begin{array}{ll}
             \sum_{j=1}^m d_j \delta({i} \succeq b_j) & \hbox{ if } \  i=1,\ldots,n \ (\mathsf{B}_n, \mathsf{C}_n) ,\ \ i=1,\ldots,n-2\ (\mathsf{D}_n),\\
             \sum_{j=1}^m d_j \delta(\overline{n} \succeq b_j) & \hbox{ if } \   i=n-1\ (\mathsf{D}_n),\\
             \sum_{j=1}^m d_j \delta({n} \succeq b_j) & \hbox{ if }\  i=n\ (\mathsf{D}_n).
           \end{array}
         \right.
 $$
Then we have
\begin{enumerate}
\item[(i)] if $b_{\bullet} \succ b_1 \succ b_2 \succ \cdots \succ b_m $ for type $\mathsf{B}_n$, then
$$ \kf_n^{s+t_n} \kf_{n-1}^{t_{n-1}} \cdots \kf_1^{t_1} \mathbf{1} \simeq \hd\,\ind \left( L(n^s) \boxtimes  \Big( \dbt{k=1}{m} \Delta_{(b_{\bullet},b_k)}^{\boxtimes d_k}\Big)\right)
\ \ \text{ for all } s \in \Z_{\ge0}, $$
and the head occurs with multiplicity one as a composition factor of
 $$\ind \left( L(n^s) \boxtimes  \Big( \dbt{k=1}{m} \Delta_{(b_{\bullet},b_k)}^{\boxtimes d_k}\Big)\right),$$
\item[(ii)] if $ b_{\bullet} \succ  b_m \succ b_{m-1} \succ \cdots \succ b_1 $ for type $\mathsf{C}_n, \mathsf{D}_n$, then
$$ \kf_n^{t_n} \cdots \kf_2^{t_2} \kf_{1}^{t_{1}} \mathbf{1} \simeq \ind \Big(\dbt{k=1}{m} \Delta_{(b_{\bullet}, b_k)}^{\boxtimes d_k}\Big).$$
\end{enumerate}
\end{enumerate}
\end{lemma}
\begin{proof}
For $k_1, \ldots, k_n \in \Z_{\ge 0}$,
it follows from Lemma \ref{Lem: sum of fi} that
\begin{align}
\kf_1^{k_1} \kf_2^{k_2} \cdots \kf_n^{k_n} \mathbf{1} &\simeq \hd\,\ind \left(L(1^{k_1}) \boxtimes \cdots \boxtimes L(n^{k_n})\right), \label{Eq: simple head 1} \\
\kf_n^{k_n} \kf_{n-1}^{k_{n-1}} \cdots  \kf_{1}^{k_{1}} \mathbf{1} &\simeq \hd\,\ind \left(L(n^{k_n}) \boxtimes \cdots \boxtimes  L(1^{k_1})\right). \label{Eq: simple head 2}
\end{align}

(1) When  $\mathfrak{A}$ is of type $\mathsf{A}_n, \mathsf{C}_n, \mathsf{D}_n$, this can be shown  in the same way as for \cite[Lem.~4.3]{KP10}.

Suppose that $\mathfrak{A}$ is of type $\mathsf{B}_n$. Note that $\Delta_{(\overline{1},b_k)}$ is one of the cuspidal representations given in
\cite[Sec.~6.2]{HMM09} and \cite[Sec.~8.5]{KR09}. By \cite[Lem.~6.6]{KR09}, we have $N_k := \ind \Delta_{(\overline{1},b_k)}^{\boxtimes d_k} $ is irreducible
for $k=1, \ldots, m$. We write $\mathbf{i}(\overline{1}, b_k) = (1,2,\ldots, l_k)$ and let
$$ \mathbf{k}_k = (\underbrace{1,\ldots,1}_{d_k}, \underbrace{2,\ldots,2}_{d_k}, \ldots, \underbrace{l_k,\ldots,l_k}_{d_k} ) $$
for $k=1,\ldots,m$. Then $\mathbf{k}_k$ appears in $\ch N_k$ with
multiplicity $(d_k!)^{l_k}$. Since $b_1 \succ b_2 \succ \cdots \succ
b_m $, the sequence $\mathbf{k} = \mathbf{k}_m * \cdots *
\mathbf{k}_1$ appears in $\ch(\ind (N_m\boxtimes \cdots \boxtimes
N_1)) $ with multiplicity $(d_m!)^{l_m} \cdots (d_1!)^{l_1}$. Lemma
\ref{Lem: irr of induced module} and Lemma \ref{Lem: commuting
lemma}  imply that $\ind \Big( \dbt{k=1}{m} N_k \Big)
\simeq \ind (N_m\boxtimes \cdots \boxtimes N_1)$ is irreducible.

Now for  $\mathbf{t} := (\underbrace{1,\ldots,1}_{t_1},\underbrace{2,\ldots,2}_{t_2},\ldots, \underbrace{n,\ldots,n}_{t_n} )$,
we have $\e(\mathbf{t} )  \left( \ind \Big( \dbt{k=1}{m} N_k \Big) \right) \ne 0$. Since $\ind  \Big( \dbt{k=1}{m} N_k \Big)$ is irreducible,
it follows from Lemma \ref{Lem: Frobenius reciprocity} that there is
a surjective homomorphism
$$ \ind (L(1^{t_1}) \boxtimes \cdots \boxtimes L(n^{t_n}))\twoheadrightarrow
\ind  \Big( \dbt{k=1}{m} N_k \Big). $$
Therefore,  $ \kf_1^{t_1} \cdots \kf_n^{t_n} \mathbf{1} \simeq \ind  \Big( \dbt{k=1}{m} N_k \Big)$ by
\eqref{Eq: simple head 1}.

(2) For this part, note that since $\Delta_{(b_{\bullet}, b_k)}$ is a cuspidal representation as in \cite[Sec.~6]{HMM09} and \cite[Sec.~8]{KR09},  $N_k := \ind \Delta_{(b_{\bullet}, b_k)}^{\boxtimes d_k}$
 is irreducible by \cite[Lem.~6.6]{KR09}.  We now separate considerations according to the type.

(\textbf{Case} $\mathsf{B}_n$)  When  $\mathfrak{A}$ is of type $\mathsf{B}_n$,
let $N = \ind\Big( \Delta_{({n},b_m)}^{\boxtimes d_m} \boxtimes \cdots \boxtimes \Delta_{({n},b_1)}^{\boxtimes d_1}\Big)$.
By \cite[Lem.~4.3]{KP10}, we have $ \kf_{n-1}^{t_{n-1}} \cdots \kf_{1}^{t_{1}} \mathbf{1} \simeq N. $
By $\eqref{Eq: surj between nablas}$ and Lemma \ref{Lem: commuting lemma}, for $d \in \Z_{\ge 0}$ we obtain
\begin{align*}
\ind \left( L(n^d) \boxtimes  \Delta_{({n},b_k)}^{\boxtimes d} \right) & \simeq
\ind \left( L(n^{d-1}) \boxtimes L(n) \boxtimes \Delta_{({n},b_k)} \boxtimes  \Delta_{({n},b_k)}^{\boxtimes d-1}\right) \\
& \twoheadrightarrow \ind\left( L(n^{d-1}) \boxtimes \Delta_{(0,b_k)}  \boxtimes \Delta_{({n},b_k)}^{\boxtimes d-1}\right )
 \simeq \ind \left( L(n^{d-1}) \boxtimes \Delta_{({n},b_k)}^{\boxtimes d-1} \boxtimes \Delta_{(0,b_k)}\right) \\
& \qquad \vdots \\
& \twoheadrightarrow \ind\left( \Delta_{(0,b_k)}^{\boxtimes d}\right).
\end{align*}
Since $t_n = d_1 + \cdots + d_m$,
by the same argument as above, we have the following chain of surjective homomorphisms
\begin{align*}
\ind \left(L(n^{s+t_n}) \boxtimes N \right)  & \twoheadrightarrow \ind \left(L(n^{s+t_n-d_m}) \boxtimes \Delta_{(b_{\bullet},b_m)}^{\boxtimes d_m} \boxtimes \Delta_{({n},b_{m-1})}^{\boxtimes d_{m-1}}
\boxtimes \cdots \boxtimes \Delta_{({n},b_1)}^{\boxtimes d_1} \right) \\
& \ \quad  \simeq \ind \left(L(n^{s+t_n-d_m})  \boxtimes \Delta_{({n},b_{m-1})}^{\boxtimes d_{m-1}}  \boxtimes \cdots \boxtimes \Delta_{({n},b_1)}^{\boxtimes d_1} \boxtimes \Delta_{(b_{\bullet},b_m)}^{\boxtimes d_m} \right) \\
& \twoheadrightarrow \ind \left(L(n^{s+t_n-d_m-d_{m-1}})  \boxtimes \Delta_{({n},b_{m-2})}^{\boxtimes d_{m-2}}  \boxtimes \cdots \boxtimes \Delta_{({n},b_1)}^{\boxtimes d_1} \boxtimes \Delta_{(b_{\bullet},b_{m-1})}^{\boxtimes d_{m-1}} \boxtimes \Delta_{(b_{\bullet},b_m)}^{\boxtimes d_m} \right) \\
& \qquad \qquad \vdots \\
&\twoheadrightarrow \ind \left(L(n^{s}) \boxtimes \Big(\dbt{i=1}{m} \Delta_{(b_{\bullet},b_i)}^{\boxtimes d_i}\Big)\right).
\end{align*}
Therefore, by Lemma \ref{Lem: sum of fi} and \cite[Lem.~3.13]{KL09}, we obtain
$$\kf_{n}^{s+t_{n}} \kf_{n-1}^{t_{n-1}} \cdots \kf_{1}^{t_{1}} \mathbf{1} \simeq \hd\,\ind \left(L(n^{s+t_n}) \boxtimes N\right) \simeq
\hd\,\ind \Big(L(n^{s}) \boxtimes \Big (\dbt{i=1}{m} \Delta_{(b_{\bullet},b_i)}^{\boxtimes d_i}\Big)\Big)$$
and this module has multiplicity one as a composition factor of $\ind \Big(L(n^{s}) \boxtimes \Big(\dbt{i=1}{m} \Delta_{(b_{\bullet},b_i)}^{\boxtimes d_i}\Big)\Big)$.

(\textbf{Case} $\mathsf{C}_n$)   For type $\mathsf{C}_n$,
  we write $\mathbf{i}(b_{\bullet},b_j)=(n, n-1, \ldots,
n - l_j + 1)$ and let
$$ \mathbf{k}_k= (\underbrace{n,\ldots,n}_{d_k}, \underbrace{n-1,\ldots,n-1}_{d_k}, \ldots, \underbrace{n - l_k+ 1,\ldots,n - l_k+ 1}_{d_k}) $$
for $k= 1, \ldots, m$. Then $\mathbf{k}_k$ appears in $\ch N_k$ with multiplicity $(d_k!)^{l_k}$. Since $ b_{\bullet} \succ b_{m} \succ b_{m-1} \succ \cdots \succ b_1$,
the sequence $\mathbf{k} := \mathbf{k}_1 * \cdots * \mathbf{k}_m$ appears in $\ch\,\ind \Big( \dbt{k=1}{m} N_k\Big)$ with multiplicity
$(d_1!)^{l_1} \cdots (d_m!)^{l_m}$. It follows from Lemma \ref{Lem: irr of induced module} and Lemma \ref{Lem: commuting lemma} that
$\ind \Big( \dbt{k=1}{m} N_k\Big)$ is irreducible.

Let $\mathbf{t} = ( \underbrace{n,\ldots,n}_{t_n}, \underbrace{n-1,\ldots,n-1}_{t_{n-1}}, \ldots, \underbrace{1,\ldots,1}_{t_{1}} ).$
Then we have $\e(\mathbf{t}) \Big( \ind \Big( \dbt{k=1}{m} N_k\Big)\Big) \ne 0 $. Since $\ind  \Big( \dbt{k=1}{m} N_k\Big)$ is irreducible,
by Lemma \ref{Lem: Frobenius reciprocity} there exists  a surjective homomorphism
$$ \ind (L(n^{t_n}) \boxtimes \cdots \boxtimes  L(1^{t_1})) \twoheadrightarrow \ind  \Big( \dbt{k=1}{m} N_k\Big). $$
Therefore, it follows from $\eqref{Eq: simple head 2}$ that $ \kf_n^{t_n} \cdots \kf_2^{t_2} \kf_{1}^{t_{1}} \mathbf{1} \simeq  \ind  \Big( \dbt{k=1}{m} N_k\Big). $

(\textbf{Case} $\mathsf{D}_n$)  For $k=1,\ldots,m$,
let $l_k$ be the length of $\mathbf{i}{(\overline{n-1},\overline{n})} * \mathbf{i}{(\overline{n},b_k)}$ if $\overline{n} \succ b_k$,
and $l_k=1$ if $b_k = n$ or $b_k = \overline{n}$. We will show that $\ind \left(\dbt{k=1}{m} \Delta_{(b_{\bullet}, b_k)}^{\boxtimes d_k}\right)$ is irreducible for type $\mathsf{D}_n$.

First suppose that $b_{m-1} \ne {n-1}.$ Let  \begin{align*}
\mathbf{k}_k = \left\{
                  \begin{array}{ll}
                    (\underbrace{n-1,\ldots,n-1}_{d_k}) & \hbox{ if }  b_k = \overline{n}, \\
                    (\underbrace{n,\ldots,n}_{d_k}, \underbrace{n-1,\ldots,n-1}_{d_k}, \ldots, \underbrace{n - l_k + 1,\ldots,n - l_k + 1}_{d_k}) & \hbox{otherwise.}
                  \end{array}
                \right.
\end{align*}
for $ k= 1, \ldots, m$. Then $\mathbf{k}_k$ appears in $\ch N_k$ with multiplicity $(d_k!)^{l_k}$. Since $b_{m} \succ b_{m-1} \succ \cdots \succ b_1$ and $b_{m-1} \ne {n-1}$,
the sequence $\mathbf{k} := \mathbf{k}_1 * \cdots * \mathbf{k}_m$ appears in $\ch\,\ind  \Big( \dbt{k=1}{m} N_k\Big)$ with multiplicity
$(d_1!)^{l_1} \cdots (d_m!)^{l_m}$. It follows from Lemma \ref{Lem: irr of induced module} and Lemma \ref{Lem: commuting lemma} that
$\ind \Big( \dbt{k=1}{m} N_k\Big)$ is irreducible.

We now assume that $b_{m-1} = {n-1}.$ Without loss of generality, we may suppose that $b_{m} = \overline{n}.$ Since
$\ind (L(n-1) \boxtimes L(n)) \simeq \ind (L(n) \boxtimes L(n-1))$, we have
$$\ind \Delta_{(\overline{n-1}, {n-1})}^{\boxtimes d_{m-1}} \boxtimes \Delta_{(\overline{n-1}, \overline{n})}^{\boxtimes d_{m}} \simeq
\ind \Delta_{(\overline{n-1}, \overline{n})}^{\boxtimes d_{m}} \boxtimes \Delta_{(\overline{n-1}, {n-1})}^{\boxtimes d_{m-1}} \simeq \ind \left( L\big(n^{d_{m-1} }\big) \boxtimes L\big((n-1)^{d_{m-1} + d_m}\big) \right) $$
and they are irreducible. Let
$
M_k = \left\{
        \begin{array}{ll}
          \ind \Delta_{(b_{\bullet}, b_k)}^{\boxtimes d_k} & \hbox{ if } k \le m-2, \\
          \ind \Delta_{(b_{\bullet}, b_{m-1})}^{\boxtimes d_{m-1}} \boxtimes \Delta_{(b_{\bullet}, b_{m})}^{\boxtimes d_{m}} & \hbox{ if } k=m-1,
        \end{array}
      \right.
$
and set
\begin{center}
$\mathbf{k}_k = \left\{
                  \begin{array}{ll}
                    (\underbrace{n,\ldots,n}_{d_k}, \underbrace{n-1,\ldots,n-1}_{d_k}, \ldots, \underbrace{n - \ell_k + 1,\ldots,n - \ell_k + 1}_{d_k}) & \hbox{ if }  k \le m-2, \\
                    (\underbrace{n,\ldots,n}_{d_{m-1}},\underbrace{n-1,\ldots,n-1}_{d_{m-1}+d_m}) & \hbox{ if } k = m-1.
                  \end{array}
                \right.$
\end{center}
Then the multiplicity of $\mathbf{k}_k$ in $\ch M_k$ is $(d_k!)^{l_k}$ if $k \le m-2$ and
$d_{m-1}! (d_m + d_{m-1})!$ if $k = m-1$.
Since $b_{\bullet} \succ b_{m} \succ b_{m-1} \succ \cdots \succ b_1$,
the sequence $\mathbf{k} := \mathbf{k}_1 * \cdots * \mathbf{k}_{m-1}$ appears in $\ch\left(\ind (M_1\boxtimes \cdots \boxtimes M_{m-1})\right)$ with multiplicity
$(d_1!)^{l_1} \cdots (d_{m-2}!)^{l_{m-2}} (d_{m-1}!) ((d_m + d_{m-1})!)$. It follows from Lemma \ref{Lem: irr of induced module} and Lemma \ref{Lem: commuting lemma} that
$\ind (M_1\boxtimes \cdots \boxtimes M_{m-1})$ is irreducible.

Now for
$\mathbf{t} = ( \underbrace{n,\ldots,n}_{t_n}, \underbrace{n-1,\ldots,n-1}_{t_{n-1}}, \ldots, \underbrace{1,\ldots,1}_{t_{1}} ),$
we have $\e(\mathbf{t})\left( \ind\Big( \dbt{i=1}{m} \Delta_{(b_{\bullet}, b_i)}^{\boxtimes d_i}\Big)\right) \ne 0$.
Since $\ind \left(\dbt{i=1}{m} \Delta_{(b_{\bullet}, b_i)}^{\boxtimes d_i}\right)$ is irreducible,
 Lemma \ref{Lem: Frobenius reciprocity} gives a surjective homomorphism
$$ \ind (L(n^{t_n}) \boxtimes \cdots \boxtimes  L(1^{t_1})) \twoheadrightarrow \ind\left(\dbt{i=1}{m} \Delta_{(b_{\bullet}, b_i)}^{\boxtimes d_i}\right), $$
which implies the desired conclusion $ \kf_n^{t_n} \cdots \kf_2^{t_2} \kf_{1}^{t_{1}} \mathbf{1} \simeq  \ind\left(\dbt{i=1}{m} \Delta_{(b_{\bullet}, b_i)}^{\boxtimes d_i}\right) $
by $\eqref{Eq: simple head 2}$.
\end{proof}

Recall the definition of $\widehat{\imath} \in \cB$ from $\eqref{Eq: def of overhat i}$. Using the surjective homomorphism in $\eqref{Eq: surj between nablas}$,
we can glue parts (1) and (2) of  Lemma \ref{Lem: seq of kf} together for types $\mathsf{B}_n, \mathsf{C}_n$,
and $\mathsf{D}_n$ to get the following result.

\begin{lemma} \label{Lem: gluing lemma}
Let $t_i \in \Z_{\ge 0}$ be such that
\begin{align*}
& t_1 \ge t_2 \ge \cdots \ge t_{n'-1} \ge t_{n'} \ge  t_{n'+1}=0 \ \  & \qquad (\mathsf{A}_n, \mathsf{C}_n), \\
& 2t_1 \ge 2t_2 \ge \cdots \ge 2t_{n-1} \ge t_{n} \ge 2t_{n+1}\ge \cdots \ge 2t_{n'-1} \ge  t_{n'}=0 \ \  & \qquad (\mathsf{B}_n), \\
& t_1 \ge t_2 \ge \cdots \ge t_{n-2} \ge t_{n-1},t_{n} \ge t_{n+1}\ge \cdots \ge t_{n'-1} \ge  t_{n'}=0 \ \  & \qquad(\mathsf{D}_n),
\end{align*}
where $n'=n\ (\mathsf{A}_n)$, $n' = 2n\ (\mathsf{B}_n)$, $n' = 2n-1\ (\mathsf{C}_n, \mathsf{D}_n)$. Set

$$ \vartheta_{i} = \left\{
                   \begin{array}{ll}
                     t_{i} - t_{i+1} & \hbox{if } i \le n'\ \  (\mathsf{A}_n,\mathsf{C}_n),\ \  i \le n-2\ (\mathsf{B}_n),\ \ i \le n-3\ (\mathsf{D}_n), \\
                     t_{n-1} - \lceil \frac{t_{n}}{2} \rceil & \hbox{if } i=n-1\ \ (\mathsf{B}_n), \\
                     \lceil \frac{t_{n}}{2} \rceil - \lfloor \frac{t_{n}}{2} \rfloor & \hbox{if } i=n\qquad \  (\mathsf{B}_n), \\
                     \lfloor \frac{t_{n}}{2} \rfloor - t_{n+1} & \hbox{if } i=n+1\ \ (\mathsf{B}_n), \\
                      t_{ n-2} - \max\{ t_{n-1}, t_{n} \} & \hbox{if } i=n-2\ \ (\mathsf{D}_n), \\
                      \max\{ 0, t_{n}-t_{n-1} \} & \hbox{if } i=n-1\ \  (\mathsf{D}_n), \\
                      \max\{ 0, t_{n-1}-t_{n} \} & \hbox{if } i=n\qquad \  (\mathsf{D}_n), \\
                      \min\{ t_{ n-1}, t_{n} \} - t_{ n+1} & \hbox{if } i = n+1\  \ (\mathsf{D}_n), \\
                      t_{i-1} - t_{i} & \hbox{if } i \ge n+2\ \ (\mathsf{B}_n, \mathsf{D}_n).
                   \end{array}
                 \right.
 $$
Then
\begin{enumerate}
\item $\hd\,\ind \left( \dbt{i=1}{n'} \Delta_{(\widehat{1}, \widehat{i+1})}^{\boxtimes \vartheta_{i}}\right)$ is irreducible and has
multiplicity one as a composition factor of $\ind \left( \dbt{i=1}{n'} \Delta_{(\widehat{1}, \widehat{i+1})}^{\boxtimes \vartheta_{i}}\right)$,
\item $ \kf_{\mathbf{i}} \mathbf{1} \simeq  \hd\,\ind \left ( \dbt{i=1}{n'} \Delta_{(\widehat{1}, \widehat{i+1})}^{\boxtimes \vartheta_{i}}\right)$,
where
$$\mathbf{i} = \left\{
                 \begin{array}{ll}
                   (\underbrace{1,\ldots,1}_{t_1}, \underbrace{2,\ldots,2}_{t_2}, \ldots, \underbrace{n,\ldots,n}_{t_n} ) & \ \  (\mathsf{A}_n), \\
                   (\underbrace{1,\ldots,1}_{t_1},  \ldots, \underbrace{n-1,\ldots,n-1}_{t_{n-1}},\underbrace{n,\ldots,n}_{t_n},
\underbrace{n-1,\ldots,n-1}_{t_{n+1}}, \ldots, \underbrace{1,\ldots,1}_{t_{2n-1}} ) & \ \  (\mathsf{B}_n, \mathsf{C}_n), \\
                   (\underbrace{1,\ldots,1}_{t_1},  \ldots, \underbrace{n-2,\ldots,n-2}_{t_{n-2}},\underbrace{n,\ldots,n}_{t_{n-1}},
\underbrace{n-1,\ldots,n-1}_{t_{n}}, \ldots, \underbrace{1,\ldots,1}_{t_{2n-2}} ) & \ \ (\mathsf{D}_n).
                 \end{array}
               \right.
$$
\end{enumerate}
\end{lemma}
\begin{proof}    By Lemma \ref{Lem: commuting lemma} and $\eqref{Eq: surj between nablas}$, for $k \in \Z_{\ge0}$  we have
\begin{equation} \label{Eq: gluing}
 \begin{aligned}
\ind\left(\Delta_{(\widehat{1},b)}^{\boxtimes k} \boxtimes \Delta_{(b,c)}^{ \boxtimes k}\right) & \simeq
\ind\left(\Delta_{(\widehat{1},b)}^{\boxtimes k-1} \boxtimes \Delta_{(\widehat{1},b)} \boxtimes \Delta_{(b,c)}  \boxtimes  \Delta_{(b,c)}^{ \boxtimes k-1}\right) \\
& \twoheadrightarrow
\ind\left(\Delta_{(\widehat{1},b)}^{\boxtimes k-1} \boxtimes \Delta_{(\widehat{1},c)} \boxtimes  \Delta_{(b,c)}^{ \boxtimes k-1}\right)\\
&\quad \  \simeq
\ind\left( \Delta_{(\widehat{1},b)}^{\boxtimes k-1} \boxtimes   \Delta_{(b,c)}^{ \boxtimes k-1} \boxtimes \Delta_{(\widehat{1},c)}\right)\\
& \hspace{1.28truecm}  \vdots \\
& \twoheadrightarrow  \ind\left( \Delta_{(\widehat{1},c)}^{\boxtimes k}\right),
 \end{aligned}
\end{equation}
where $b \in B \ (\mathsf{A}_n)$, \ $b = \widehat{n+1}\ (\mathsf{B}_n)$, \ $b = \widehat{n}\ (\mathsf{C}_n)$,\ $ b = \widehat{n-1}\ (\mathsf{D}_n)$ and $c \in \cB$ with $b \succ c$.

(\textbf{Case} $\mathsf{A}_n$) The assertion in this case follows from Lemma \ref{Lem: seq of kf}.

(\textbf{Case} $\mathsf{B}_n$)   For type $\mathsf{B}_n$, let
\begin{align*}
M(k) &=  \ind \left( \Big(\dbt{i=1}{2n-k} \Delta_{(\widehat{1}, \widehat{\imath+1})}^{\boxtimes \vartheta_i} \Big) \boxtimes
 \Delta_{(\widehat{1},\widehat{n+1})}^{\boxtimes (\vartheta_{2n} + \cdots + \vartheta_{2n+1-k})} \right),  \ \ \hbox{\rm and} \\
N(k) &= \ind \left( \Delta_{(\widehat{n+1}, \widehat{2n+2-k})}^{\boxtimes \vartheta_{2n+1-k}}
\boxtimes \Delta_{(\widehat{n+1}, \widehat{2n+3-k})}^{\boxtimes \vartheta_{2n+2-k}} \boxtimes \cdots \boxtimes \Delta_{(\widehat{n+1}, \widehat{2n+1})}^{\boxtimes \vartheta_{2n}} \right)
\end{align*}
for $k = 1, \ldots,n$. Note that $ t_n = \lceil \frac{t_{n}}{2} \rceil + \lfloor \frac{t_{n}}{2} \rfloor$,
$\lceil \frac{t_{n}}{2} \rceil = \vartheta_{n} + \cdots \vartheta_{2n} $ and $\lfloor \frac{t_{n}}{2} \rfloor = \vartheta_{n+1} + \cdots +\vartheta_{2n} $.
If follows from Lemma \ref{Lem: seq of kf} that
$$ M(n) \simeq \kf_1^{t_1} \cdots \kf_{n-1}^{t_{n-1}} \kf_{n}^{\lceil \frac{t_{n}}{2} \rceil} \mathbf{1}, \qquad
\hd N(n) \simeq \kf_{n}^{\lfloor \frac{t_{n}}{2} \rfloor } \kf_{n-1}^{t_{n+1}} \cdots \kf_{1}^{t_{2n-1}}   \mathbf{1}. $$
By Lemma \ref{Lem: seq of kf} and \cite[Lem.~3.13]{KL09}, there is a surjective homomorphism
$$ \ind \left( \Delta_{(\overline{1},\overline{n})}^{\boxtimes \lceil \frac{t_{n}}{2} \rceil} \boxtimes L(n^{\lceil \frac{t_{n}}{2} \rceil}) \right) \twoheadrightarrow
\ind \Delta_{(\overline{1},0)}^{\boxtimes \lceil \frac{t_{n}}{2} \rceil}, $$
which yields the following surjective homomorphism
\begin{align*}
& \ind \left(  \Big( \Big (\dbt{i=1}{n-1} \Delta_{(\widehat{1}, \widehat{\imath+1})}^{\boxtimes \vartheta_i}\Big) \boxtimes  \Delta_{(\widehat{1},\widehat{n})}^{\boxtimes \lceil \frac{t_{n}}{2} \rceil} \Big)
\boxtimes \left( L(n^{\lceil \frac{t_{n}}{2} \rceil}) \boxtimes N(n) \right) \right) \\
& \  \twoheadrightarrow
\ind \left(  \Big (\dbt{i=1}{n-1} \Delta_{(\widehat{1}, \widehat{i+1})}^{\boxtimes \vartheta_i}\Big) \boxtimes  \Delta_{(\widehat{1},\widehat{n+1})}^{\boxtimes \lceil \frac{t_{n}}{2} \rceil}
\boxtimes N(n) \right) \\ & \quad  \  \simeq
\ind \left( \Big(\dbt{i=1}{n} \Delta_{(\widehat{1}, \widehat{\imath+1})}^{\boxtimes \vartheta_i}\Big) \boxtimes  \Delta_{(\widehat{1},\widehat{n+1})}^{\boxtimes \lfloor \frac{t_{n}}{2} \rfloor}
\boxtimes N(n) \right) \\  & \quad \  \simeq  \ind (M(n) \boxtimes N(n)).
\end{align*}
It follows from Lemma \ref{Lem: seq of kf} that
\begin{align*}
\kf_1^{t_1} \cdots \kf_{n-1}^{t_{n-1}} \mathbf{1} & \simeq
\ind \left(\Big(\dbt{i=1}{n-1} \Delta_{(\widehat{1}, \widehat{i+1})}^{\boxtimes \vartheta_i}\Big) \boxtimes  \Delta_{(\widehat{1},\widehat{n})}^{\boxtimes \lceil \frac{t_{n}}{2} \rceil} \right), \\
\kf_{n}^{t_{n} } \kf_{n-1}^{t_{n+1}} \cdots \kf_{1}^{t_{2n-1}}   \mathbf{1} & \simeq \hd\,\ind \left( L(n^{\lceil \frac{t_{n}}{2} \rceil}) \boxtimes N(n) \right).
\end{align*}
Since $\varepsilon_i \left( \ind \big(L(n^{\lceil \frac{t_{n}}{2} \rceil}) \boxtimes N(n)\big) \right) =0 $ for $i=1,\ldots, n-1$, by Lemma \ref{Lem: sum of fi}
we have
$$  \kf_{\mathbf{i}} \mathbf{1}  \simeq \hd\,\ind\big(M(n) \boxtimes N(n)\big), $$
and this module has  multiplicity one as a composition factor of $ \ind\big (M(n) \boxtimes N(n)\big)$.

Now  by Lemma \ref{Lem: commuting lemma} and $\eqref{Eq: gluing}$, we have the following chain of surjective homomorphisms
\begin{align*}
& \ind(M(n) \boxtimes N(n)) \\
& \qquad \simeq \ind \left( \Big( \dbt{i=1}{n} \Delta_{(\widehat{1}, \widehat{\imath+1})}^{\boxtimes \vartheta_i} \Big) \boxtimes
 \Delta_{(\widehat{1},\widehat{n+1})}^{\boxtimes (\vartheta_{2n}+\cdots + \vartheta_{n+2})}  \boxtimes
\left( \Delta_{(\widehat{1},\widehat{n+1})}^{\boxtimes \vartheta_{n+1}} \boxtimes \Delta_{(\widehat{n+1}, \widehat{n+2})}^{\boxtimes \vartheta_{n+1}} \right)
\boxtimes N(n-1) \right) \\
& \qquad \twoheadrightarrow \ind \left( \Big( \dbt{i=1}{n} \Delta_{(\widehat{1}, \widehat{\imath+1})}^{\boxtimes \vartheta_i} \Big) \boxtimes
 \Delta_{(\widehat{1},\widehat{n+1})}^{\boxtimes (\vartheta_{2n}+\cdots + \vartheta_{n+2})} \boxtimes
 \Delta_{(\widehat{1},\widehat{n+2})}^{\boxtimes \vartheta_{n+1}}
\boxtimes N(n-1) \right) \\
& \qquad \quad \ \  \simeq \ind (M(n-1) \boxtimes N(n-1)) \\
& \hspace{3truecm}  \vdots \\
& \qquad \twoheadrightarrow \ind (M(1) \boxtimes N(1))   \simeq \ind \left (\Big(\dbt{i=1}{2n-1} \Delta_{(\widehat{1}, \widehat{\imath+1})}^{\boxtimes \vartheta_i} \Big) \boxtimes
\Delta_{(\widehat{1},\widehat{n+1})}^{\boxtimes \vartheta_{2n}} \boxtimes \Delta_{(\widehat{n+1},\widehat{2n+1})}^{\boxtimes \vartheta_{2n}} \right) \\
& \qquad \twoheadrightarrow \ind \left(\dbt{i=1}{2n} \Delta_{(\widehat{1}, \widehat{i+1})}^{\boxtimes \vartheta_i}\right ),
\end{align*}
which completes the proof for  $\mathsf{B}_n$.

(\textbf{Case} $\mathsf{C}_n$)  For  type $\mathsf{C}_n$,
set
\begin{align*}
M(k) &= \ind \left( \Big(\dbt{i=1}{n-1} \Delta_{(\widehat{1}, \widehat{i+1})}^{\boxtimes \vartheta_{i}}\Big) \boxtimes
\Delta_{(\widehat{1}, \widehat{n})}^{\boxtimes (\vartheta_{n} + \cdots + \vartheta_{n-1+k})} \right), \ \ \hbox{\rm and} \\
N(k) &= \ind \left( \Delta_{(\widehat{n}, \widehat{k+n})}^{\boxtimes \vartheta_{k+n-1}} \boxtimes  \Delta_{(\widehat{n}, \widehat{k+n-1})}^{\boxtimes \vartheta_{k+n-2}} \boxtimes
\cdots \boxtimes \Delta_{(\widehat{n}, \widehat{n+1})}^{\boxtimes \vartheta_{n}} \right)
\end{align*}
for $ 1 \le k \le n$. It follows from Lemma \ref{Lem: seq of kf} that
$$  M(n) \simeq \kf_{1}^{t_{1}} \kf_{2}^{t_{2}} \cdots \kf_{n-1}^{t_{n-1}} \mathbf{1} , \quad \
 N(n) \simeq \kf_{n}^{t_{n}} \kf_{n-1}^{t_{n+1}} \cdots \kf_{1}^{t_{2n-1}} \mathbf{1} .$$
Since $\varepsilon_i(N(n))=0$ for $i=1,\ldots, n-1$,  Lemma \ref{Lem: sum of fi} implies
$$ \kf_{\mathbf{i}} \mathbf{1} \simeq \hd\,\ind (M(n) \boxtimes N(n)) $$
and this module has multiplicity one as a composition factor  in  $\ind\big(M(n) \boxtimes N(n)\big)$.

By Lemma \ref{Lem: commuting lemma} and $\eqref{Eq: gluing}$, we have a chain of surjective homomorphisms
\begin{align*}
 \ind \big(M(n) \boxtimes N(n)\big) & \simeq  \ind \left(  M(n-1) \boxtimes \Delta_{(\widehat{1},\widehat{n})}^{\boxtimes \vartheta_{2n-1}} \boxtimes
\Delta_{(\widehat{n},\widehat{2n})}^{\boxtimes \vartheta_{2n-1}} \boxtimes N(n-1) \right) \\
&\twoheadrightarrow  \ind \left( M(n-1) \boxtimes \Delta_{(\widehat{1},\widehat{2n})}^{\boxtimes \vartheta_{2n-1}} \boxtimes N(n-1) \right) \\
&\quad \ \  \simeq \ind \left( M(n-1) \boxtimes N(n-1) \boxtimes \Delta_{(\widehat{1},\widehat{2n})}^{\boxtimes \vartheta_{2n-1}} \right) \\
& \hspace{3truecm} \vdots \\
&\twoheadrightarrow  \ind \left( M(1) \boxtimes N(1) \boxtimes \Delta_{(\widehat{1},\widehat{n+2})}^{\boxtimes \vartheta_{n+1}} \boxtimes \cdots
\boxtimes \Delta_{(\widehat{1},\widehat{2n})}^{\boxtimes \vartheta_{2n-1}} \right)  \allowdisplaybreaks \\
&\twoheadrightarrow  \ind \left(\Big(\dbt{i=1}{n-1} \Delta_{(\widehat{1}, \widehat{\imath+1})}^{\boxtimes \vartheta_{i}}\Big) \boxtimes \Delta_{(\widehat{1},\widehat{n+1})}^{\boxtimes \vartheta_{n}} \boxtimes \cdots
\boxtimes \Delta_{(\widehat{1},\widehat{2n})}^{\boxtimes \vartheta_{2n-1}} \right)  \\
&\quad \ \  \simeq  \ind \left( \dbt{i=1}{2n-1} \Delta_{(\widehat{1}, \widehat{\imath+1})}^{\boxtimes \vartheta_{i}}\right),
\end{align*}
which yields
$ \kf_{\mathbf{i}} \mathbf{1} \simeq  \hd\,\ind \left ( \dbt{i=1}{2n-1} \Delta_{(\widehat{1}, \widehat{\imath+1})}^{\boxtimes \vartheta_{i}}\right). $

(\textbf{Case} $\mathsf{D}_n$)  Without loss of generality, we may assume in the $\mathsf{D}_n$ case that $t_{n-1} \ge t_n$. Note that $\vartheta_{n-1} = 0$
and $ \vartheta_n =  \mathsf{max}\{ t_{n-1}, t_{n} \} - \mathsf{min}\{ t_{n-1}, t_{n} \}$. Let
\begin{align*}
M(k) &= \ind \left( \Big(\dbt{i=1}{n-2} \Delta_{(\widehat{1}, \widehat{i+1})}^{\boxtimes \vartheta_{i}}\Big) \boxtimes
\Delta_{(\widehat{1}, \widehat{n-1})}^{\boxtimes (\vartheta_{n}+ \cdots + \vartheta_{n-1+k} )} \right), \\
N(k) &= \ind \left( \Delta_{(\widehat{n-1}, \widehat{k+n})}^{\boxtimes \vartheta_{k+n-1}} \boxtimes  \Delta_{(\widehat{n-1}, \widehat{k+n-1})}^{\boxtimes \vartheta_{k+n-2}} \boxtimes
\cdots \boxtimes \Delta_{(\widehat{n-1}, \widehat{n+1})}^{\boxtimes \vartheta_{n}} \right)
\end{align*}
for $k = 1, \ldots, n$. It follows from Lemma \ref{Lem: seq of kf} that
$$  M(n) \simeq \kf_{1}^{t_{1}} \kf_{2}^{t_{2}} \cdots \kf_{n-2}^{t_{n-2}} \mathbf{1}, \quad \
 N(n) \simeq \kf_{n}^{t_{n-1}} \kf_{n-1}^{t_{n}} \kf_{n-2}^{t_{n+1}} \cdots \kf_{1}^{t_{2n-2}} \mathbf{1}.$$
Since $\varepsilon_i(N(n))=0$ for $i=1,\ldots, n-2$, Lemma \ref{Lem: sum of fi} gives
$$ \kf_{\mathbf{i}} \mathbf{1} \simeq \hd\,\ind (M(n) \boxtimes N(n)), $$
and this module with multiplicity one in $\ind (M(n) \boxtimes N(n))$.

Since  $\Delta_{(\widehat{1},\widehat{n})}^{\boxtimes \vartheta_{n-1}} = \C $,  Lemma \ref{Lem: commuting lemma} and \eqref{Eq: gluing} again give a string of surjective homomorphisms
\begin{align*}
 \ind (M(n) \boxtimes N(n))
& \simeq  \ind \left( M(n-1) \boxtimes \Delta_{(\widehat{1},\widehat{n-1})}^{\boxtimes \vartheta_{2n-1}} \boxtimes
\Delta_{(\widehat{n-1},\widehat{2n})}^{\boxtimes \vartheta_{2n-1}} \boxtimes N(n-1) \right) \\
&\twoheadrightarrow  \ind \left( M(n-1) \boxtimes \Delta_{(\widehat{1},\widehat{2n})}^{\boxtimes \vartheta_{2n-1}} \boxtimes N(n-1) \right) \\
&\quad \  \simeq \ind \left( M(n-1) \boxtimes N(n-1) \boxtimes \Delta_{(\widehat{1},\widehat{2n})}^{\boxtimes \vartheta_{2n-1}} \right) \\
& \hspace{3truecm} \vdots  \\
&\twoheadrightarrow  \ind \left( M(1) \boxtimes N(1) \boxtimes \Delta_{(\widehat{1},\widehat{n+2})}^{\boxtimes \vartheta_{n+1}} \boxtimes \cdots
\boxtimes \Delta_{(\widehat{1},\widehat{2n})}^{\boxtimes \vartheta_{2n-1}} \right)  \allowdisplaybreaks \\
&\twoheadrightarrow  \ind \left(\Big(\dbt{i=1}{n-2} \Delta_{(\widehat{1}, \widehat{\imath+1})}^{\boxtimes \vartheta_{i}}\Big) \boxtimes \Delta_{(\widehat{1},\widehat{n+1})}^{\boxtimes \vartheta_{n}} \boxtimes \cdots
\boxtimes \Delta_{(\widehat{1},\widehat{2n})}^{\boxtimes \vartheta_{2n-1}} \right)  \\
&\quad \ \simeq \ind \left( \dbt{i=1}{2n-1} \Delta_{(\widehat{1}, \widehat{\imath+1})}^{\boxtimes \vartheta_{i}} \right),
\end{align*}
which yields the desired result
$ \kf_{\mathbf{i}} \mathbf{1} \simeq  \hd\,\ind\left( \dbt{i=1}{2n-1} \Delta_{(\widehat{1}, \widehat{\imath+1})}^{\boxtimes \vartheta_{i}}\right). $
\end{proof}

As an immediate consequence of this lemma we have

\begin{corollary} \  $\mathcal{N}_n(v) = \hd\,\ind \Delta(\mathbf{a}(v);n) \  \text{ for } v \in B(\infty),$
where $\mathcal{N}_n(v)$ is the irreducible module given in Proposition \ref{Prop: description of Nk}.
\end{corollary}

We now are ready to prove Theorem \ref{Thm: main theorem}.

\begin{proof}[Proof of Theorem \ref{Thm: main theorem}]

Let $ I_{(k)} = \{ n-k+1, n-k+2, \ldots , n  \}$ for $k=1, \ldots, n$. Note that $I_{(k)} \subset I_{(k+1)}$ and $|I_{(k)}|=k$.
Let $U_q(\g_k)$ be the subalgebra of $U_q(\g)$ generated by $e_i,\ f_i\ (i\in I_{(k)})$ and $q^h\ (h \in \mathsf{P}^{\vee})$,
and let $\mathfrak{B}_k$ be the crystal obtained from $B(\infty)$ by forgetting
the $i$-arrows for $i \notin I_{(k)}$.
It follows from Table \ref{Tb: w0} that $U_q(\g_k)$ is of type $\mathsf{X}_{k}$ when $U_q(\g)$ is of type $\mathsf{X}_n$ ($\mathsf{X = A,B,C,D}$).
Recall the definition $\mathbf{a}_i$ for $\mathbf{a} \in \mathcal{S}$ and the sequences $\mathbf{s}_k$ given in Section \ref{Sec: Description Nk}.
Take $v \in B(\infty)$ and let $\triangle(\mathbf{a}(v)) = \{ t_{ij} \}$.
If $i=1, \ldots, n$ for type $\mathsf{A}_n, \mathsf{B}_n, \mathsf{C}_n$, then
$$ \mathbf{a}(v)_i = (t_{i,n+1-i}, t_{i,n+2-i}, \ldots, t_{i,d_i}), $$
where $d_i = n\ \hbox{\rm for all} \ i \ \ (\mathsf{A}_n)$ and $d_i = n-1+i\ (\mathsf{B}_n, \mathsf{C}_n)$. If $i=1, \ldots, n$ for type $\mathsf{D}_n$, then
$$ \mathbf{a}(v)_i = \left\{
                      \begin{array}{ll}
                        (t_{1,n-1}) & \hbox{ if } i = 1, \\
                        (t_{1,n}) & \hbox{ if } i = 2, \\
                        (t_{i-1, n+1-i}, t_{i-1, n+2-i}, \ldots, t_{i-1, n-2+i}) & \hbox{ if } i = 3, \ldots, n.
                      \end{array}
                    \right.
$$
Let $M_k = \ind \Delta(\mathbf{a}(v);k)$ for $k=1,\ldots, n$. Then,
by Proposition \ref{Prop: description of Nk}, Lemma \ref{Lem: gluing lemma}, and the choice of $I_{(k)}$, we obtain
$$ \mathcal{N}_k(v)  = \hd M_k \ \ \ \text{ for } k=1,\ldots, n.$$
By the construction of $\Delta(\mathbf{a}(v);k)$ and Lemma \ref{Lem:
gluing lemma},  we know
\begin{enumerate}
\item[(i)] $\varepsilon_i( M_k ) = 0 \ \ \text{for } i\in I_{(k-1)}$,
\item[(ii)] $\hd M_k$ is irreducible and occurs with multiplicity one as a composition factor of $ M_k $.
\end{enumerate}
Therefore, by Lemma \ref{Lem: sum of fi} and Proposition \ref{Prop: crystal iso},
\begin{align*}
\Phi(v) & = \hd\,\ind\left( (\hd M_1) \boxtimes (\hd M_2) \boxtimes \cdots \boxtimes (\hd M_{n})\right) \\
&= \hd\,\ind (  M_1 \boxtimes  M_2 \boxtimes \cdots \boxtimes M_{n}) \\
&= \hd\,\ind \big (  \Delta(\mathbf{a}(v);1) \boxtimes \cdots \boxtimes \Delta(\mathbf{a}(v);n) \big),
\end{align*}
which completes the proofs of (1) and (2).

Let $\Psi^\lambda: B(\lambda) \rightarrow \B(\lambda)$ be the canonical crystal isomorphism given by
$\Psi^{\lambda}(\kf_{\mathbf{i}}  b_\lambda) = \kf_{\mathbf{i}} \mathbf{1}$
for any sequence $\mathbf{i}$ of elements in $I$. Then the following diagram commutes:
$$ \CrystalDiagram $$
Here, $\imath_\lambda$ (resp.\ $\jmath_\lambda$) is the crystal embedding $\eqref{Eq: crystal embedding}$
from $B(\lambda)$ (resp.\ $\B(\lambda)$) to $B(\infty)\otimes T_\lambda \otimes C$ (resp.\ $ \B(\infty)\otimes T_\lambda \otimes C$).
Assume for $v\in B(\lambda)$ that  $\iota_{\lambda}(v) = v' \otimes t_\lambda \otimes c$ for some $ v' \in B(\infty)$,
and recall that  $\mathbf{a}(v) = \mathbf{a}(v')$.
Let $\jmath_\lambda^{-1}$ be the inverse from $\im(\jmath_\lambda)$ to $\B(\lambda)$. Then
$$ \Psi^\lambda = \jmath^{-1}_\lambda \circ (\Psi \otimes \id \otimes \id) \circ \imath_\lambda, $$
which yields
\begin{align*}
\Psi^\lambda(v) &= \jmath^{-1}_\lambda \circ (\Psi \otimes \id \otimes \id) \circ \imath_\lambda(v) \\
& = \jmath^{-1}_\lambda \circ (\Psi \otimes \id \otimes \id) (v' \otimes t_\lambda \otimes c) \\
&= \jmath^{-1}_\lambda ( \hd\,\ind (\Delta(\mathbf{a}(v');1) \boxtimes \cdots \boxtimes \Delta(\mathbf{a}(v');n)) \otimes t_\lambda \otimes c) \\
&= \jmath^{-1}_\lambda ( \hd\,\ind (\Delta(\mathbf{a}(v);1) \boxtimes \cdots \boxtimes \Delta(\mathbf{a}(v);n)) \otimes t_\lambda \otimes c) \\
&= \hd\,\ind (\Delta(\mathbf{a}(v);1) \boxtimes \cdots \boxtimes \Delta(\mathbf{a}(v);n)).
\end{align*}  This proves assertion (3) and concludes the proof of the main theorem (Theorem \ref{Thm: main theorem}).
\end{proof}

\vskip 1em

\begin{center} {\textbf{Acknowledgments}}   \end{center}
Work on this paper was facilitated by a visit  by S.-J.~Kang, S.-j.~Oh, and E.~Park  to
the University of Wisconsin-Madison and by a visit by  G.~Benkart to
the Korea Institute for Advanced Study.   We express our gratitude to these
institutions for their hospitality.

\bibliographystyle{amsplain}



\end{document}